\documentclass[a4paper,11pt,reqno]{amsart} 
\usepackage[utf8]{inputenc}
\usepackage{amsfonts}
\usepackage{amssymb}
\usepackage{enumerate}
\usepackage{graphicx}
\usepackage{mathtools}
\usepackage{amsmath, amsthm}
\usepackage{amssymb,amsthm,amsfonts,amsbsy,latexsym}
\usepackage{dsfont}
\usepackage{times}

\usepackage{tikz}
\usetikzlibrary{calc,intersections,through,backgrounds,shapes.geometric}
\usetikzlibrary{graphs}

\usepackage{color}
\usepackage[T1]{fontenc}
\usepackage[sort,numbers]{natbib}
\usepackage{bbm}
\usepackage[a4paper, left = .85in, right = .85in, top = 1in, bottom = 1in]{geometry}
\usepackage{amsmath}
\usepackage{amsfonts}
\usepackage{amssymb}
\usepackage{bbm}
\usepackage{mathrsfs}
\usepackage[utf8]{inputenc}
\usepackage[english]{babel}
\usepackage{hyperref}
\usepackage{relsize}
\usepackage{appendix}
\usepackage{lipsum}
\usepackage{placeins}
\usepackage{todonotes}

\usepackage{caption}
\usepackage{subcaption}
\usepackage{color}

\usepackage[mode = buildnew]{standalone}

\usepackage{enumerate}
\newenvironment{enumeratei}{\begin{enumerate}[\upshape (i)]}{\end{enumerate}}
\newenvironment{enumeratea}{\begin{enumerate}[\upshape (a)]}{\end{enumerate}}

\newcommand\blfootnote[1]{%
  \begingroup
  \renewcommand\thefootnote{}\footnote{#1}%
  \addtocounter{footnote}{-1}%
  \endgroup
}

\makeatletter
  \@addtoreset{chapter}{part}
  \@addtoreset{@ppsaveapp}{part}
\makeatother

\usepackage{palatino}

\makeatletter
\newsavebox\myboxA
\newsavebox\myboxB
\newlength\mylenA

\newcommand*\xoverline[2][0.75]{%
    \sbox{\myboxA}{$\m@th#2$}%
    \setbox\myboxB\null
    \ht\myboxB=\ht\myboxA%
    \dp\myboxB=\dp\myboxA%
    \wd\myboxB=#1\wd\myboxA
    \sbox\myboxB{$\m@th\overline{\copy\myboxB}$}
    \setlength\mylenA{\the\wd\myboxA}
    \addtolength\mylenA{-\the\wd\myboxB}%
    \ifdim\wd\myboxB<\wd\myboxA%
       \rlap{\hskip 0.5\mylenA\usebox\myboxB}{\usebox\myboxA}%
    \else
        \hskip -0.5\mylenA\rlap{\usebox\myboxA}{\hskip 0.5\mylenA\usebox\myboxB}%
    \fi}
\makeatother

\usepackage{hyperref}
\hypersetup{
  colorlinks,
  linkcolor=black,
  citecolor=black,
  filecolor=black,
  urlcolor=black,
  pdftitle={},
  pdfauthor={S. Bhamidi, S. Dhara, R. van der Hofstad, S. Sen},
  pdfcreator={S. Dhara},
  pdfsubject={},
  pdfkeywords={}
}
\numberwithin{equation}{section}
\usepackage{cleveref}

\def\sss{\scriptscriptstyle}

\newcommand*{\vmbar}[2]{\skew{#1}{\bar}{#2}}

\newcommand*{\Mtilde}[1]{\skew{5}{\tilde}{#1}}
\newcommand{\prob}[1]{\ensuremath{\mathbbm{P}\left(#1\right)}}
\newcommand{\expt}[1]{\ensuremath{\mathbbm{E}\left[#1\right]}}
\newcommand{\var}[1]{\ensuremath{\mathrm{Var}\left(#1\right)}}
\newcommand{\refl}[1]{\ensuremath{\mathrm{refl}\left(#1\right)}}

\newcommand{\floor}[1]{\ensuremath{\left\lfloor #1 \right\rfloor}}
\newcommand{\ind}[1]{\ensuremath{\mathbbm{1}\left\{#1\right\}}}
\newcommand{\pto}{\ensuremath{\xrightarrow{\mathbbm{P}}}}
\newcommand{\dto}{\ensuremath{\xrightarrow{d}}}
\newcommand{\eqd}{\stackrel{d}{=}}
\newcommand{\surp}[1]{\ensuremath{\mathrm{SP}(#1)}}

\newcommand{\diam}{\ensuremath{\mathrm{diam}}}
\newcommand{\PR}{\ensuremath{\mathbbm{P}}}
\newcommand{\E}{\ensuremath{\mathbbm{E}}}
\newcommand{\R}{\ensuremath{\mathbb{R}}}
\newcommand{\N}{\ensuremath{\mathbb{N}}}
\newcommand{\1}{\ensuremath{\mathbbm{1}}}

\newcommand{\dst}{\ensuremath{\mathrm{d}}}
\newcommand*{\dGHP}{\ensuremath{\mathrm{d}_{\sss \mathrm{GHP}}}}

\newcommand{\dis}{\ensuremath{\mathrm{dis}}}
\newcommand{\sigmap}{\ensuremath{\sigma(\mathbf{p})}}

\newcommand{\shortarrow}{{\sss \downarrow}}

\newcommand{\bld}[1]{\boldsymbol{#1}}
\newcommand{\oP}{o_{\sss \PR}}
\newcommand{\OP}{O_{\sss \PR}}
\newcommand{\e}{\mathrm{e}}
\newcommand{\dif}{\mathrm{d}}

\newcommand{\bl}{\mathfrak{B}}

\newcommand{\xx}{\bld{x}}
\newcommand{\yy}{\bld{y}}

\newcommand{\rCM}{\mathrm{CM}}
\newcommand{\CM}{\mathrm{CM}_n(\bld{d})}
\newcommand{\CMP}{\mathrm{CM}_n(\bld{d},p_n)}

\newcommand{\Cli}{\mathscr{C}_i}

\newcommand{\Clb}{\mathscr{C}_b}
\newcommand{\pCli}{\mathscr{C}_i'}
\newcommand{\pClj}{\mathscr{C}_j'}
\newcommand{\Clbi}{\mathscr{C}_{b_i}}

\newcommand{\Wli}{\mathscr{W}_i}
\newcommand{\Wlj}{\mathscr{W}_j}

\newtheorem{thm}{Theorem}[section]

\newtheorem{prop}[thm]{Proposition}

\newtheorem*{ass*}{Assumption}
\newtheorem*{theorem*}{Theorem}

\newtheorem{theorem}{Theorem}[section]
\newtheorem{algo}{Algorithm}
 \newtheorem{lemma}[theorem]{Lemma}
\newtheorem{proposition}[theorem]{Proposition}

\newtheorem{assumption}{Assumption}
\newtheorem{remark}{Remark}
\newtheorem{fact}{Fact}
\newtheorem{defn}{Definition}

\theoremstyle{definition}

\usepackage{environ}
\NewEnviron{eq}{%
\begin{equation}\begin{split}
  \BODY
\end{split}\end{equation}
}

\numberwithin{equation}{section}

\newcommand{\set}[1]{\left\{#1\right\}}

\newcommand{\cA}{\mathcal{A}}\newcommand{\cB}{\mathcal{B}}\newcommand{\cC}{\mathcal{C}}
\newcommand{\cF}{\mathcal{F}}
\newcommand{\cG}{\mathcal{G}}
\newcommand{\cJ}{\mathcal{J}}\newcommand{\cL}{\mathcal{L}}

\newcommand{\cP}{\mathcal{P}}\newcommand{\cR}{\mathcal{R}}
\newcommand{\cS}{\mathcal{S}}\newcommand{\cT}{\mathcal{T}}
\newcommand{\cV}{\mathcal{V}}\newcommand{\cX}{\mathcal{X}}

\newcommand{\vY}{\mathbf{Y}}

\newcommand{\vl}{\mathbf{l}}

\newcommand{\vp}{\mathbf{p}}
\newcommand{\vt}{\mathbf{t}}

\newcommand{\mvU}{\boldsymbol{U}}

\newcommand{\fI}{\mathfrak{I}}

\newcommand{\fP}{\mathfrak{P}}

\newcommand{\bG}{\mathbb{G}}

\newcommand{\bR}{\mathbb{R}}
\newcommand{\bT}{\mathbb{T}}

\newcommand{\dA}{\mathfrak{G}} 

\newcommand{\rE}{\mathrm{E}}

\newcommand{\rV}{\mathrm{V}}

\newcommand{\sC}{\mathscr{C}}

\newcommand{\sS}{\mathscr{S}}

\newcommand{\sW}{\mathscr{W}}

\newcommand{\sP}{\mathfrak{P}}

\newcommand{\RC}{\mathrm{RC}}

\DeclareMathOperator{\pr}{\mathbbm{P}}

\DeclareMathOperator{\ord}{ord}

\DeclareMathOperator{\mass}{mass}
\DeclareMathOperator{\con}{con}

\newcommand{\Ep}{\E_{\vp}}

\newcommand{\icrt}{\mathscr{T}_{\sss(\infty)}^{\bld{\beta}}}
\newcommand{\tilicrt}{\mathscr{T}_{\sss(\infty)}^{\bld{\beta},\star}}

\DeclareMathOperator{\shape}{shape}

\begin{document}
\title[Critical heavy-tailed networks]{Universality for critical heavy-tailed network models: \\ Metric structure of maximal components}
\author[Bhamidi]{Shankar Bhamidi$^1$}
 \author[Dhara]{Souvik Dhara$^{2,3}$}
 \author[Hofstad]{Remco van der Hofstad$^4$}
 \author[Sen]{Sanchayan Sen$^5$}
\blfootnote{\emph{Emails:} 
 \href{mailto:bhamidi@email.unc.edu}{bhamidi@email.unc.edu},
 \href{mailto:sdhara@mit.edu}{s.dhara@mit.edu},
 \href{mailto:r.w.v.d.hofstad@tue.nl}{r.w.v.d.hofstad@tue.nl},
 \href{mailto:sanchayan.sen1@gmail.com}{sanchayan.sen1@gmail.com}} 
 \date{\today}
\blfootnote{$^1$Department of Statistics and Operations Research,  University of North Carolina}
\blfootnote{$^2$Department of Mathematics, Massachusetts Institute of Technology}
\blfootnote{$^3$Microsoft Research Lab -- New England}
\blfootnote{$^4$Department of Mathematics and Computer Science, Eindhoven University of Technology}
\blfootnote{$^5$Department of Mathematics, Indian Institute of Science}
\blfootnote{\emph{Acknowledgment}. 
SB was partially supported by NSF grants DMS-1613072, DMS-1606839 and ARO grant W911NF-17-1-0010. SD,  RvdH, and SS was supported by the Netherlands Organisation for Scientific Research (NWO) through Gravitation Networks grant 024.002.003. 
In addition, RvdH was supported by VICI grant 639.033.806, and SS was supported by a CRM-ISM fellowship. 
SD and RvdH would like to thank UNC Chapel Hill for hospitality where part of this work was done.
A major part of this work was done when SD was a PhD student at Eindhoven University of Technology (TU/e), and SD thanks TU/e for supporting this work.}
\maketitle
\begin{abstract}
We study limits of the largest connected components (viewed as metric spaces) obtained by critical percolation on uniformly chosen graphs and configuration models with heavy-tailed degrees. 
For rank-one inhomogeneous random graphs, such results were derived by Bhamidi, van der Hofstad, Sen (2018)~\cite{BHS15}.
We develop general principles under which the identical scaling limits as in~\cite{BHS15} can be obtained.
Of independent interest,  we derive refined asymptotics for various susceptibility functions and the maximal diameter in the barely subcritical regime. 
\end{abstract}

\tableofcontents

\section{Introduction}

Over the last decades, applications arising from complex systems in different fields
have inspired a host of models for networks as well as models of dynamically evolving networks.
One of the major themes in the study of these models has been in the nature of the emergence of the giant component.
A classical example is the percolation process, where each edge of the network is independently kept with probability $p$, and deleted otherwise.
As $p$ increases from 0 to 1, the graph experiences a transition in the connectivity structure, i.e., there exists a "critical percolation value" $p_c$ such that for any $\varepsilon>0$ and $p< p_c(1-\varepsilon)$, the proportion of vertices in the largest component is asymptotically negligible, while for $p> p_c(1+\varepsilon)$, a unique giant component emerges containing an asymptotically positive proportion of vertices \cite{ABS04,F07,RGCN1,J09,JLR00}. 

Understanding the behavior at criticality is one of the key questions in statistical physics because the components exhibit unique and key features in the critical regime.
In the physics literature, the critical behavior of percolation relates to studying optimal paths in networks in the so-called strong disorder regime. 
A wide array of conjectures and heuristic deductions of the associated critical exponents can be found in \cite{braunstein2003optimal,braunstein2007optimal,CbAH02,Halvin05}. 
In a nutshell, these conjectures can be described as follows: 
\begin{quote} \it 
The intrinsic nature of the critical behavior does not depend on the exact description of the model, but only on moment conditions on the degree distribution. There are two major universality classes corresponding to the critical regime and the nature of emergence of the giant depending on whether the degree distribution has asymptotically finite third moment or infinite third moment.
For example, in case of power-law degree distributions (i.e., $\PR(D\geq x) \approx x^{-(\tau-1)}$ the precise nature of the approximation left implicit), the nature of the critical behavior depends only on the power-law degree exponent $\tau$: (a) For $\tau > 4$, the maximal component sizes are of the order $n^{2/3}$ in the critical regime, whilst typical distances in these maximal connected components scale like $n^{1/3}$; (b) For $\tau \in (3,4)$, the maximal component sizes are of the order $n^{(\tau-2)/(\tau-1)}$, whilst distances scale like $n^{(\tau-3)/(\tau - 1)}$.
\end{quote} 
The above conjectures have inspired a large and beautiful collection of works in probability theory.  
In a seminal work, Aldous \cite{A97} provided a detailed understanding for the vector of rescaled component sizes at criticality for Erd\H{o}s-R\'enyi random graphs, and the scaling limits for component sizes are now well understood under quite general setups in both finite third-moment \cite{BHL10,DHLS15,H13,Jo10,NP10b,NP10a,R12} and infinite third-moment \cite{BHL12,DHLS16,H13,Jo10} settings.
We refer the reader to \cite[Chapter 1]{Dha18}, \cite[Chapter 4]{Hof17} for detailed discussions about this topic.
A recent and emerging direction in this literature aims at understanding the critical component structures, and distances within these components from a very general perspective. 
This line of work was pioneered by Addario-Berry, Broutin and Goldschmidt~\cite{ABG09}, where the largest connected components were shown to converge when viewed as metric spaces (see below for exact definitions). 
Subsequently, \cite{BBSX14,BBW12,BS16} have explored the \emph{universality class} corresponding to \cite{ABG09}, showing that the universality in the finite third-moment setting holds not only with respect to functionals like component sizes, but also the entire metric structure.
On the other hand, in the infinite third-moment setting, a recent result \cite{BHS15} shows that the metric structure turns out to be fundamentally different.
 The results in \cite{BHS15} was obtained for one fundamental random graph model (rank-one model, closely related to the Chung-Lu \cite{CL02b,CL02} and Norros-Reittu model \cite{BDM06}) under the assumption that the weights follow a power-law distribution.
In this paper, we explore the universality class corresponding to the candidate limit law established in \cite{BHS15}.
Informally, the main contributions of this paper are as follows: \vspace{.3cm}

\noindent
$\rhd$ {\bf Universality theorem:} We establish sufficient conditions that imply convergence to the limits established in~\cite{BHS15}. 
This is described later in Theorem~\ref{thm:univesalty}. 
	Since we need to set up a number of constructs, a formal statement is deferred until all of these objects have been defined. 
	We refer to Theorem~\ref{thm:univesalty} as a \emph{universality theorem} because it identifies the domain of attraction of the limit laws in \cite{BHS15}. 
	Informally, the theorem implies that if a sequence of dynamic networks satisfies some entrance boundary conditions in the barely subcritical regime, and evolves approximately according to the multiplicative coalescent dynamics over the critical window, then the metric structure of the critical components are close to those for rank-one inhomogeneous random graphs.
Theorem~\ref{thm:univesalty} is similar in spirit to \cite[Theorem 3.4]{BBSX14}, but our result holds for the infinite third-moment degrees.
Technically, we do not need additional restrictions as in  \cite[Assumption 3.3]{BBSX14}, since we compare the metric structures in the Gromov-weak topology, instead of the Gromov-Hausdorff-Prokhorov topology.
The universality theorem holds under arguably optimal assumptions (see Remark~\ref{rem:about-assumption}). 
	
\vspace{.2cm}	

	\noindent
$\rhd$ {\bf Critical percolation on graphs with given degrees:} Our primary motivation was to analyze the critical regime for percolation on the uniform random graph model (and the closely associated configuration model) with a prescribed degree distribution that converges to a heavy-tailed degree distribution. 
	Limit laws for the metric structure of maximal components in the critical regime are described in Theorems~\ref{thm:main}~and~\ref{thm:main-simple}. 
	These results are proved under Assumption~\ref{assumption1}, which is the most general set of assumptions under which the component sizes were shown to converge in \cite{DHLS16} (see \cite[Section 2 and 3]{DHLS16} for the applicability and necessity of these assumptions).
	
%
%
	
	\vspace{.2cm}
	\noindent
$\rhd$ {\bf Barely subcritical regime:} In order to carry out the above analysis and in particular to apply the universality theorem for percolation on configuration models, we establish refined bounds for component sizes, various susceptibility functionals, and diameters of connected components in the barely subcritical regime of the configuration model which are of independent interest; these are described in Theorems \ref{thm:susceptibility} and \ref{thm:diam-max}.
\subsection{Organization of the paper}
In Section~\ref{sec:res}, we describe the configuration model and critical behavior of percolation, which is the main motivation of this paper, and then describe the main results relevant to this model. 
Section~\ref{sec:discussion} has a detailed discussion about the relevance of the results in this paper, some open problems, and an informal description of the proof ideas. 
We provide a full description of the limit objects and various notions of convergence of metric-space-valued random variables in Section~\ref{sec:definitions-full}. 
Section~\ref{sec:univ-thm} describes and proves the general universality result. 
Section~\ref{sec:entrance-bdd-proofs} proves results about the configuration model in the barely subcritical regime. 
Finally, Section~\ref{sec:proof-metric-mc} combines the above estimates with a coupling of the evolution of the configuration model through the critical percolation scaling window to finish the proof of Theorem \ref{thm:main}.

\section{Critical percolation on the configuration model} 
\label{sec:res}
In this section, we state our main results. 
In Section~\ref{sec:mspace-results}, we state the results about the metric structure of the largest critical percolation clusters of the configuration model. 
We defer full definitions of the limit objects as well as notions of convergence of measured metric spaces to Section~\ref{sec:definitions-full}.  
In Section~\ref{sec:mesoscopic-results}, we state the results about the barely subcritical regime, and we conclude this section with an overview of the proofs in Section~\ref{sec:overview}.

\subsection{Metric structure of the critical components} \label{sec:mspace-results}
\subsubsection*{The configuration model}  Consider $n$ vertices labeled by $[n]:=\{1,2,...,n\}$ and a non-increasing sequence of degrees $\boldsymbol{d} = ( d_i )_{i \in [n]}$ such that $\ell_n = \sum_{i \in [n]}d_i$ is even. For notational convenience, we suppress the dependence of the degree sequence on $n$. The configuration model on $n$ vertices having degree sequence $\boldsymbol{d}$ is constructed as follows \cite{B80,MR95}:
 \begin{itemize}
 \item[] Equip vertex $j$ with $d_{j}$ stubs, or \emph{half-edges}. Two half-edges create an edge once they are paired. Therefore, initially we have $\ell_n=\sum_{i \in [n]}d_i$ half-edges. Pick any one half-edge and pair it with a uniformly chosen half-edge from the remaining unpaired half-edges and keep repeating the above procedure until all the unpaired half-edges are exhausted. 
 \end{itemize}
  Let $\CM$ denote the graph constructed by the above procedure.
  Note that $\CM$ may contain self-loops or multiple edges. 
  Let $\mathrm{UM}_n(\bld{d})$ denote the graph chosen uniformly at random from the collection of all simple graphs with degree sequence $\boldsymbol{d}$.
  It can be shown that the conditional law $\mathrm{CM}_{n}(\boldsymbol{d})$, conditioned being simple, is same as $\mathrm{UM}_n(\bld{d})$ (see \cite[Proposition 7.15]{RGCN1}). 
 It was further shown in \cite{J09c} that, if the degree distribution satisfies a finite second-moment condition (a condition which will hold in the context of this paper), then the asymptotic probability of the graph being simple converges to a positive limit. 
  
Let us now describe the assumptions on the degree sequences. 
For $p>0$, define the metric space
\begin{equation}
\ell^p_{\shortarrow}=\Big\{(x_1,x_2,\dots)\in \R^\N_{+}: x_1\geq  x_2\geq   \dots , \ \sum_i x_i^p<\infty\Big\},  
\end{equation}with metric $d(\bld{x}, \bld{y})= \big( \sum_{i} |x_i-y_i|^p \big)^{1/p}$.
 Fix $\tau\in (3,4)$.  Throughout this paper we use the following functionals of $\tau$:
\begin{equation}\label{eqn:notation-const}
 \alpha= 1/(\tau-1),\qquad \rho=(\tau-2)/(\tau-1),\qquad \eta=(\tau-3)/(\tau-1).
\end{equation}
 \begin{assumption}[Degree sequence]\label{assumption1}
\normalfont  For each $n\geq 1$, let $\bld{d}=\boldsymbol{d}_n=(d_1,\dots,d_n)$ be a degree sequence ($d_i$'s may depend on $n$, but we suppress $n$ in the notation for clarity). 
We assume the following about $(\boldsymbol{d}_n)_{n\geq 1}$ as $n\to\infty$:
\begin{enumerate}[(i)] 
\item \label{assumption1-1} (\emph{High-degree vertices}) For each fixed $i\geq 1$, $ n^{-\alpha}d_i\to \theta_i,$
where $\boldsymbol{\theta}=(\theta_1,\theta_2,\dots)\in \ell^3_{\shortarrow}\setminus \ell^2_{\shortarrow}$. \vspace{.15cm}
\item \label{assumption1-2} (\emph{Moment assumptions}) 
Let $V_n$ be chosen uniformly from $[n]$ (independently of $\mathrm{CM}_n(\boldsymbol{d}))$, and $D_n = d_{V_n}$.
%
Then $D_n$ converges in distribution to some positive integer-valued random variable $D$, and 
\begin{equation}
 \frac{1}{n}\sum_{i\in [n]}d_i\to \mu := \E[D], \quad \frac{1}{n}\sum_{i\in [n]}d_i^2 \to \mu_2:=\E[D^2],\quad  \lim_{K\to\infty}\limsup_{n\to\infty}n^{-3\alpha} \sum_{i=K+1}^{n} d_i^3=0.
\end{equation}
\end{enumerate}
\end{assumption}   
\begin{remark} \normalfont
Assumption~\ref{assumption1} is identical to \cite[Assumption~1]{DHLS16}. 
We refer the reader to \cite[Sections~2~and~3]{DHLS16} for discussions about the relevance and necessity of these assumptions.
It was shown in \cite[Section~2]{DHLS16} that Assumption~\ref{assumption1} is satisfied in two key settings,  when (i) the degrees are taken to be an i.i.d.~sample from a power-law distribution, and (ii) the degrees are chosen according to the quantiles of a power-law distribution. 
The first setting has been considered in \cite{Jo10}, and the latter setting has been considered for the rank-one inhomogeneous random graphs in~\cite{BHS15,BHL12}.
\end{remark}
The component sizes of $\CM$ are known to undergo a phase transition \cite{JL09,MR95} depending on the parameter
\begin{equation}
 \nu_n = \frac{\sum_{i\in [n]}d_i(d_i-1)}{\sum_{i\in [n]}d_i} \to \nu =\frac{\expt{D(D-1)}}{\expt{D}}.
\end{equation} 
When $\nu>1$, $\CM$ is supercritical in the sense that there exists a unique \emph{giant} component with high probability, and when $\nu<1$, all the components have size $o(n)$ with high probability and $\CM$ is subcritical. In this paper, when considering percolation on $\CM$, we will always assume that
\begin{equation}\label{eq:super-crit-start}
 \nu>1, \text{ i.e. } \CM \text{ is supercritical.}
\end{equation}
 Percolation refers to deleting each edge of a graph independently with probability $1-p$.  
In the case of percolation on random graphs, the deletion of edges is also independent from the underlying graph.
 Let $\CMP$ and $\mathrm{UM}_n(\bld{d},p_n)$ denote the graphs obtained from percolation with probability $p_n$ on the graphs $\mathrm{CM}_n(\boldsymbol{d})$ and $\mathrm{UM}_n(\bld{d})$, respectively. 
 For $p_n\to p$, it was shown in \cite{J09} that the critical point for the phase transition of the component sizes is $p=1/\nu$. 
 The critical window for percolation was studied in \cite{DHLS16,DHLS15} to obtain the asymptotics of the largest component sizes and their surplus edges. 
In the infinite third-moment setting, $\CMP$ lies in the critical window when, for some $\lambda\in\R$,
 \begin{equation}\label{eq:critical-window-defn}
  p_n = p_n(\lambda) = \frac{1}{\nu_n}+\frac{\lambda}{n^{\eta}}+o(n^{-\eta}).
 \end{equation}
We now explain the precise meaning of convergence of components as metric spaces.
Let $\mathscr{C}_{\sss (i)}^p(\lambda)$ denote the $i$-th largest component of $\rCM_n(\bld{d},p_n(\lambda))$. 
A measured metric space is a metric space equipped with a measure on the associated Borel sigma-algebra.
Each component $\mathscr{C}$ can be viewed as a measured metric space  with (i) the metric being the graph distance where each edge has length one; (ii) the measure being proportional to the counting measure, i.e., for any $A\subset \mathscr{C}$, the measure of $A$ is given by $\mu_{\sss \mathrm{ct},i}(A) = |A|/|\mathscr{C}_{\sss (i)}^p(\lambda)|$, where $|A|$ denotes the cardinality of $A$. 
For a generic measured metric space $M = (M,\mathrm{d},\mu)$ and $a>0$, $aM$ denotes the measured metric space $(M,a\mathrm{d},\mu)$. 
We write $\sS_*$ for the space of all measured metric spaces equipped with the Gromov-weak topology (see Section~\ref{sec:defn:GHP-weak}) and let $\sS_*^{\N}$ denote the corresponding product space with the accompanying product topology. 
For each $n\geq 1$, view $\big( n^{-\eta}\mathscr{C}_{\sss (i)}^p(\lambda) \big)_{i\geq 1} $ as an object in $\sS_*^{\N}$ by appending an infinite sequence of empty metric spaces after enumerating the components in $\mathrm{CM}_n(\bld{d},p_n(\lambda))$. 
The main results for critical percolation on the configuration model are as follows:
\begin{theorem}\label{thm:main}Consider $\mathrm{CM}_n(\bld{d},p_n(\lambda))$ satisfying {\rm Assumption~\ref{assumption1}}, \eqref{eq:super-crit-start} and \eqref{eq:critical-window-defn} for some $\lambda\in\R$. There exists a sequence of random measured metric spaces $(\mathscr{M}_i(\lambda))_{i\geq 1}$ such that on $\mathscr{S}_*^\N$, as $n\to\infty$,
\begin{equation}\label{eq:thm:main}
 \big( n^{-\eta}\mathscr{C}_{\sss (i)}^p(\lambda) \big)_{i\geq 1} \dto  \big(\mathscr{M}_i(\lambda)\big)_{i\geq 1}.
\end{equation} 
\end{theorem}
\begin{theorem}\label{thm:main-simple}
Under {\rm Assumption~\ref{assumption1}}, \eqref{eq:super-crit-start} and \eqref{eq:critical-window-defn} for some $\lambda\in\R$, the convergence in \eqref{eq:thm:main} also holds for the components of $\mathrm{UM}_n(\bld{d},p_n(\lambda))$, with the identical limiting object.
\end{theorem}

\begin{remark}
The limiting objects are precisely described in Section~\ref{sec:descp-limit}.
\end{remark}

\begin{remark}
The notion of convergence in Theorems~\ref{thm:main}~and~\ref{thm:main-simple} implies weak convergence of a wide array of continuous functionals with respect to the Gromov-weak topology. 
For example, it implies the joint convergence of the distances between an arbitrary (but fixed) number of uniformly (and independently) chosen vertices in the $i$-th largest component of $\mathrm{CM}_n(\bld{d},p_n(\lambda))$ or $\mathrm{UM}_n(\bld{d},p_n(\lambda))$.
\end{remark}
\begin{remark}
The conclusion of Theorem~\ref{thm:main} holds if the measure $\mu_{\sss ct , i }$ on $\mathscr{C}_{\sss (i)}^p$ is replaced by more general measures.
Indeed, define the probability measure $\mu_{w,i}:= \sum_{k\in A} w_k/ \sum_{k\in \mathscr{C}_{\sss (i)}^p} w_k$ for $A\subset \mathscr{C}_{\sss (i)}^p$. 
To prove analogous results as Theorems~\ref{thm:main}~and~\ref{thm:main-simple} with $\mu_{w,i}$'s, we require $w_i$'s to satisfy some regularity conditions (see Assumption~\ref{assumption-w} below).
The reason will be discussed in Remark~\ref{rem:switch-measure}.
\end{remark}

\begin{remark}The results above can be extended to the case $\PR(D_n\geq x)\sim L(x)x^{-(\tau-1)}$, where $L(\cdot)$ is a slowly-varying function. 
The scaling limits would be the same, however the scaling exponents will be different as observed in \cite{DHLS16}.
In particular, the width of the scaling window now turns out to be $n^{-\eta}L_1(n)^{2}$ (for some slowly varying $L_1(\cdot)$) instead of~$n^{-\eta}$, and results identical to Theorem~\ref{thm:main} can be obtained by scaling the distances by~$n^{\eta}L_1(n)^{-2}$.
\end{remark}

\subsection{Mesoscopic properties of the critical clusters: barely subcritical regime} \label{sec:mesoscopic-results}
One of the main ingredients in the proof of Theorem~\ref{thm:main} is a refined analysis of various susceptibility functions in the \emph{barely subcritical} regime (see \eqref{defn:barely-subcrit} below for a definition) for the percolation process.
The barely subcritical and supercritical regimes correspond to regimes that are just below or above the critical window. 
For the percolation process under Assumption~\ref{assumption1}, barely subcritical (supercritical) behavior is observed for $p$ satisfying $n^{\eta}(p-p_n(0))\to-\infty$ ($n^{\eta}(p-p_n(0))\to\infty$), where $p_n(0)$ is defined in \eqref{eq:critical-window-defn} for $\lambda = 0$.
These behaviors are well understood for Erd\H{o}s-R\'enyi random graphs  \cite[Section 23]{JKLP93}, \cite{Bol01,JLR00} and configuration models in the  Erd\H{o}s-R\'enyi universality class~\cite{HM12,KS08,R12}. 
For barely supercritical configuration models in the heavy-tailed setting,  the size of the emerging giant component was obtained in~\cite{HJL16}.
We provide a detailed picture about the component sizes and susceptibility functions in the subcritical regime below.

We will prove general statements about the susceptibility functions applicable not just to percolation on the configuration model, but rather to any barely subcritical configuration model.
Since percolation on a configuration model yields a configuration model \cite{F07,J09}, the above yields susceptibility functions for percolation on configuration model as a special case.
 To set this up we need a little more notation, where each vertex in the network is associated with both degree and weight, satisfying the following assumptions:

\begin{assumption}[Barely subcritical degree sequence]
	\label{assumption-w}
\normalfont 
Let $\boldsymbol{d}'=(d_1',\dots,d_n')$ be a degree sequence and let $w_{\sss (\cdot )}:[n]\mapsto\R$ be a non-negative weight function such that the following conditions hold:
\begin{enumerate}[(i)] 
\item Assumption~\ref{assumption1} holds for $\bld{d}'$ with some $\bld{c}\in \ell^3_{\shortarrow}\setminus \ell^2_{\shortarrow}$, and 
\begin{equation}
 \lim_{n\to\infty}\frac{1}{n}\sum_{i\in [n]}d_i'=\mu_{d},\quad \lim_{n\to\infty}\frac{1}{n}\sum_{i\in [n]}w_i=\mu_{w},\quad \lim_{n\to\infty}\frac{1}{n}\sum_{i\in [n]}d_i'w_i=\mu_{d,w}.
 \end{equation}
 \item $\max_{i\in [n]}w_i = O(n^{\alpha})$, $\sum_{i\in [n]} w_i^2 =O(n)$ and  
  $ \max\Big\{ \sum_{i\in [n]} w_i^3, \sum_{i\in [n]}d_i'^2 w_i,\sum_{i\in [n]}d_i' w_i^2 \Big\} = O(n^{3\alpha}). $ \vspace{.15cm}
 \item (\emph{Barely subcritical regime}) The configuration model is at the barely subcritical regime, i.e., there exists $0<\delta<\eta$ and $\lambda_0>0$ such that
\begin{equation}\label{defn:barely-subcrit}
 \nu_n'=\frac{\sum_{i\in [n]}d_i'(d_i'-1)}{\sum_{i\in [n]}d_i'}=1-\lambda_0 n^{-\delta}+o(n^{-\delta}).
\end{equation}
\end{enumerate}
\end{assumption} 
 Let $\sC'(j)$ denote the connected component of $\rCM_n(\bld{d}')$ containing vertex $j$, and define 
\begin{eq}\label{defn:pCli}
\pCli = 
\begin{cases}
\sC'(i), &\quad \text{ if } i \leq j,\  \forall j \in \sC'(i), \\
\varnothing, &\quad \text{ otherwise},
\end{cases}
\end{eq}
and $\Wli=\sum_{k\in\pCli}w_k$. 
Define the weight-based susceptibility functions as
\begin{equation}\label{defn:susc-w}
 s_r^\star = \frac{1}{n}\sum_{i\geq 1}\Wli^r \quad \text{for } r\geq 1, \qquad
 s_{pr}^\star = \frac{1}{n}\sum_{i\geq 1} \Wli\times |\pCli|.
\end{equation}
The definition in \eqref{defn:pCli} takes care of the double counting in the definition of susceptibility functions.
Also, define the weighted distance-based susceptibility as 
\begin{equation}\label{defn:susc-dist}
\mathcal{D}_n^\star=\frac{1}{n}\sum_{i,j\in [n]}w_iw_j\mathrm{d}(i,j)\ind{i,j \text{ are in the same connected component}},
\end{equation} where $\mathrm{d}$ denotes the graph distance. 
The goal of the next result is to show that the component sizes and the susceptibility functions defined in \eqref{defn:susc-w} and \eqref{defn:susc-dist} satisfy asymptotic conditions such as the entrance boundary conditions for the multiplicative coalescent~\cite{AL98}:
\begin{theorem}[Susceptibility functions]\label{thm:susceptibility} Under {\rm Assumption~\ref{assumption-w}}, as $n\to\infty$, 
 \begin{eq}
  n^{-\delta}s_2^\star\pto \frac{\mu_{d,w}^2}{\mu_d\lambda_0}, &\quad n^{-\delta}s_{pr}^\star \pto \frac{\mu_{d,w}}{\lambda_0}, \quad n^{-(\alpha+\delta)}\Wlj\pto \frac{\mu_{d,w}}{\mu_d\lambda_0}c_j, \\
  n^{-3\alpha-3\delta+1}s_3^\star&\pto \bigg(\frac{\mu_{d,w}}{\mu_d\lambda_0}\bigg)^3\sum_{i=1}^{\infty}c_i^3, \quad n^{-2\delta}\mathcal{D}_n^\star\pto \frac{\mu_{d,w}^2}{\mu_d\lambda_0^2}. 
 \end{eq}
\end{theorem}
 For a connected graph $G$, $\Delta(G)$ denotes the diameter of the graph, and for any arbitrary graph~$G$, $\Delta_{\max}(G):=\max \Delta(\mathscr{C})$, where the maximum is taken over all connected components ${\mathscr{C}\subset G}$. 
 We simply write $\Delta_{\max}$ for $\Delta_{\max}(\rCM_n(\bld{d}'))$. 
The asymptotics of $\Delta_{\max}$ is derived below:

\begin{theorem}[Maximum diameter]\label{thm:diam-max} 
Under {\rm Assumption~\ref{assumption-w}}, as $n\to\infty$, $\PR(\Delta_{\max}>n^{\delta}(\log(n))^2)\to 0$.
\end{theorem}

\begin{remark}\normalfont 
By taking $w_i=1$ for all $i\in [n]$, implies that $\Wli=|\Cli|$, and thus Theorem~\ref{thm:susceptibility} hold also for the usual susceptibility functions defined in terms of the component sizes (cf. \cite{J09b}).  
In the proof of Theorem~\ref{thm:main}, we will require a more general weight function, where $w_i$ is taken to be the number of half-edges deleted from vertex $i$ due to percolation.
\end{remark}

\begin{remark}\normalfont 
Unlike Theorem~\ref{thm:main}, Theorems~\ref{thm:susceptibility}~and~\ref{thm:diam-max} yield  statements about convergence in probability to constants. 
So, under Assumption~\ref{assumption-w}, one can use the fact from~\cite{J09c} that $\liminf_{n\to\infty}\PR(\mathrm{CM}_{n}(\bld{d}')\text{is simple})>0$, and thus it immediately follows that the results in Theorems~\ref{thm:susceptibility}~and~\ref{thm:diam-max} hold for $\mathrm{UM}_n(\bld{d}')$. 
\end{remark}

\subsection{Overview of the proof}\label{sec:overview}
We now summarize the key ideas of the proofs at a heuristic level.
\subsubsection*{Universality theorem} As discussed earlier, we first prove a universality theorem (Theorem~\ref{thm:univesalty}) which roughly states that if one replaces the vertices in a rank-one inhomogeneous random graph by \emph{small} metric spaces (called blobs), then the limiting metric space structure remains identical. 
The characterization of \emph{blobs} leads to some asymptotic negligibility conditions, formally stated in Assumption~\ref{assm:blob-diameter}, which simply says that the diameter of the individual blobs must be negligible compared to the typical distances in the whole graph. 
However, the typical distance can be cumulatively affected by the blobs, hence we get a different scaling factor for distances in Theorem~\ref{thm:univesalty} than in  Theorem~\ref{thm:BHS15}.

\subsubsection*{Mesoscopic or Blob-level analysis}
Percolation on $\CM$ can be viewed as a dynamic process by associating i.i.d.~uniform$[0,1]$ weights $U_e$ to each edge $e$, and keeping $e$ if $U_e\leq p$. 
The parameter $p\in [0,1]$ can be interpreted as time. 
Now for $p_n = p_n(\lambda_n)$, for some $\lambda_n\to-\infty$, $\mathrm{CM}_n(\bld{d},p_n(\lambda_n))$ lies in the barely subcritical regime and the estimates for different functionals can be obtained using Theorem~\ref{thm:susceptibility}.
We regard the components of $\mathrm{CM}_n(\bld{d},p_n(\lambda_n))$ as the blobs. 
Under the current scaling, the blobs shrink to zero, and the edges appearing in the dynamic process between the interval $[p_n(\lambda_n),p_n(\lambda)]$ connecting the blobs give rise to the macroscopic structure of the largest components of $\mathrm{CM}_n(\bld{d},p_n(\lambda))$.
However, the effects of the blobs on the limiting structure are reflected via different functionals, which is the reason for referring to the properties of the blobs as \emph{mesoscopic} properties.
\subsubsection*{Coupling to the multiplicative coalescent}
Finally, the goal is to understand the macroscopic structure formed between blobs within the time interval $[p_n(\lambda_n),p_n(\lambda)]$.
The merging dynamics of the components between $[p_n(\lambda_n),p_n(\lambda)]$ can be heuristically described as follows:
Let $p_0$ be a time when an edge appears. 
Then the two half-edges corresponding to the new edge are chosen uniformly at random from the open half-edges (half-edges deleted due to percolation) of $\mathrm{CM}_n(\bld{d},p_0-)$.
Therefore, if $(\mathcal{O}_i(p))_{i\geq 1}$ denotes the vector of open half-edges in distinct components at time $p$, then the clusters corresponding to $\mathcal{O}_i(p)$ and $\mathcal{O}_j(p)$ merge at rate proportional to $\mathcal{O}_i(p)\times \mathcal{O}_j(p)$  and creates a new cluster with $\mathcal{O}_i(p)+ \mathcal{O}_j(p)-2$ open half-edges.
Thus, the elements of the vector $(\mathcal{O}_i(p))_{i\geq 1}$, seen as masses, merge \emph{approximately} as the multiplicative coalescent (see Definition~\ref{defn:mul-coalescent}), in the sense that the dynamics experience a depletion of half-edges in the components. 
Now, we can run a parallel process where the paired half-edges are replaced with new \emph{dummy} open half-edges to the corresponding vertices \cite{BBSX14,DHLS16}.
The dynamics in the latter process gives rise to an exact multiplicative coalescent and due to this fact, the modified graph $\bar{\mathcal{G}}_n$ can be shown to be distributed as a rank-one inhomogeneous random graph with the blobs being the mesoscopic components at time $p_n(\lambda_n)$. 
Now, the graph $\bar{\mathcal{G}}_n$ becomes the candidate for applying our universality theorem (see Theorem~\ref{thm:mspace-limit-modified}).
\subsubsection*{Structural comparison}
Finally, we perform a structural comparison between $\mathrm{CM}_n(\bld{d},p_n(\lambda))$ and $\bar{\mathcal{G}}_n$ to conclude Theorem~\ref{thm:main}. 
Let us consider the largest component $\mathscr{C}_{\sss (1)}^p$ (respectively $\bar{\mathscr{C}}_{\sss (1)}^p$) of $\mathrm{CM}_n(\bld{d},p_n(\lambda))$ (respectively $\bar{\mathcal{G}}_n$). 
By the above coupling (with dummy half-edges being added), $\mathscr{C}_{\sss (1)}^p\subset\bar{\mathscr{C}}_{\sss (1)}^p$ and we know the asymptotic metric structure of $\bar{\mathscr{C}}_{\sss (1)}^p$.
Now, the idea is to show that (a) $|\bar{\mathscr{C}}_{\sss (1)}^p\setminus \mathscr{C}_{\sss (1)}^p|= o(|\mathscr{C}_{\sss (1)}^p|)$ with high probability implying that the part of $\bar{\mathscr{C}}_{\sss (1)}^p$ outside $\mathscr{C}_{\sss (1)}^p$ is insignificant, (b) for any pair of vertices $u,v\in \mathscr{C}_{\sss (1)}^p$, the shortest path between them in $\mathscr{C}_{\sss (1)}^p$ and $\bar{\mathscr{C}}_{\sss (1)}^p$ are identical. 
These two properties conclude the proof of Theorem~\ref{thm:main} under the Gromov-weak topology.

\section{Discussion}\label{sec:discussion}

\subsubsection*{Optimality of assumptions and Gromov-weak topology}
As mentioned in the introduction, our goal is not only to consider critical percolation, but to explore the universality class for the scaling limits in~\cite{BHS15} in the same spirit as it was done in \cite{BBSX14} for the Erd\H{o}s-R\'enyi universality class. 
Our universality theorem (Theorem~\ref{thm:univesalty}) holds under optimal assumptions, and does not require additional restrictions such as \cite[Assumption 3.3]{BBSX14}.
However, it is worthwhile noting that the universality theorem (and consequently Theorem~\ref{thm:main}) holds with respect to the Gromov-weak topology instead of the stronger  Gromov-Hausdorff-Prokhorov (GHP) topology. 
This is not a restriction that we impose, but in fact there is a conceptual barrier. 
If the convergence in Theorem~\ref{thm:main} would hold in the GHP-topology only under Assumption~\ref{assumption1}, then the limiting metric space would be compact for any $\bld{\theta}\in \ell^3_{\shortarrow}\setminus \ell^2_{\shortarrow}$, but additional restrictions are needed for the compactness of the limiting metric space and simply assuming $\bld{\theta}\in \ell^3_{\shortarrow}\setminus \ell^2_{\shortarrow}$ does not suffice.
See \cite[Section 7]{AMP04}, for an explicit conjecture  about the compactness of such metric spaces by Aldous, Miermont and Pitman. 
In a follow-up work~\cite{BDHS18}, we extend the scaling limit results in the GHP-topology by establishing the so-called global lower mass bound property \cite[Theorem 6.1]{ALW16}, which ensures that the components have sufficient mass everywhere and thus  forbids the existence of long, thin paths, when the total mass of the component converges. 
However, one needs additional technical conditions in~\cite{BDHS18} on top of Assumption~\ref{assumption1} to prove the global lower mass bound property.

\subsubsection*{Extensions, recent developments and open problems}
\begin{enumeratei}
\item The universality theorem is applicable to dynamically evolving random networks with heavy-tailed degrees, which evolve (approximately) as the multiplicative coalescent over the critical window, and satisfy some nice properties such as Theorem~\ref{thm:susceptibility} in the barely subcritical regime. 
For this reason, we believe that the universality theorem and the methods of this paper are applicable to many known inhomogeneous random graph models with suitable kernels \cite{BJR07}, as well as Bohman-Frieze processes which satisfy different initial conditions so that one gets a heavy tailed-degree distribution at criticality. 
We leave these as interesting open problems. 

\item In a recent work, Broutin, Duquesne and Wang~\cite{BDW18} obtained structural limit laws for rank-one inhomogeneous random graphs which evolve as general multiplicative coalescent processes over the critical window. 
This framework unifies the scaling limits for the heavy-tailed and non heavy-tailed cases in terms of a single limit law.
It will be interesting to prove a universality theorem for the limit laws in \cite{BDW18}.
\item Recently, Conchon-Kerjan and Goldschmidt~\cite{CG17} derived the scaling limit of the maximal components at criticality for $\CM$ when the degrees form an  i.i.d.~sample from a power-law distribution with $\tau\in (3,4)$. 
The properties of the corresponding limiting object was investigated in a recent preprint by Goldschmidt,  Haas, and S{\'{e}}nizergues~\cite{GHS18}. 
The scaling limits in the i.i.d.~setting has a completely different description of the limiting object compared to the one in this paper. 
It will be interesting to explore the connections between the results in the above paper and the current work. 

\item It turns out that the study of the component structures corresponding to critical percolation plays a crucial role in the study of the metric structure of the minimal spanning tree (MST) \cite{ABGM13}.
In fact, a detailed understanding of the metric structures in the critical window obtained in \cite{ABG09} played a pivotal role in the proofs of \cite{ABGM13}.
The connections to the MST-problem such as those outlined in \cite{ABGM13} suggest that the scaling limit results in this paper will be useful in the study of metric structures for the MST for graphs with given degrees in the heavy-tailed regime. 
However, the MST problem in this regime is an open question. 
\end{enumeratei}

\section{Convergence of metric spaces, discrete structures and limit objects}
\label{sec:definitions-full}
The aim of this section is to define the proper notion of convergence relevant to this paper (Section~\ref{sec:defn:GHP-weak}),
set up discrete structures required in the statement and in the proof of the universality result in Theorem~\ref{thm:univesalty} (Sections~\ref{defn:super-graph}, \ref{sec:tree-space}, \ref{sec:p-trees}), and describe limit objects that arise in Theorem~\ref{thm:main} (Sections~\ref{sec:inhom-cont-tree} and~\ref{sec:descp-limit}). 

\subsection{Gromov-weak topology}
\label{sec:defn:GHP-weak}
A complete separable measured metric space (denoted by $(X, \mathrm{d} , \mu)$) is a complete, separable metric space $(X,\mathrm{d})$ with an associated probability measure $\mu$ on the Borel sigma algebra $\cB(X)$.
The Gromov-weak topology is defined on $\mathscr{S}_0$, the space of 
all complete and separable measured metric spaces  (see \cite{GPW09,G07}, \cite[Section 2.1.2]{BHS15}).
 The notion is formulated based on the philosophy of finite-dimensional convergence.
 Two measured metric spaces $(X_1,\dst_1,\mu_1)$, $(X_2,\dst_2,\mu_2)$ are considered to be equivalent if there exists an isometry  $\psi:\mathrm{support}(\mu_1)\mapsto \mathrm{support}(\mu_2)$ such that $\mu_2=\mu_1 \circ \psi^{-1}$. 
 Let $\mathscr{S}_*$ be the space of all equivalence classes of $\mathscr{S}_0$. 
 We (slightly) abuse the notation by not distinguishing between a metric space and its corresponding equivalence class. 
 Fix $l\geq 2$ and $(X,\dst,\mu)\in \mathscr{S}_*$. 
 Given any collection of points $\mathbf{x}=(x_1,\dots,x_l)\in X^l$, define $\mathrm{D}(\mathbf{x}):=(\dst(x_i,x_j))_{ i,j\in [l]}$ to be the matrix of pairwise distances of the points in $\mathbf{x}$. A function $\Phi: \mathscr{S}_*\mapsto \R$ is called a polynomial if there exists a bounded continuous function $\phi:\R^{l^2}\mapsto \R$ such that 
\begin{equation}\label{defn:polynomials}
 \Phi((X,\dst,\mu))=\int \phi(\mathrm{D}(\mathbf{x}))d\mu^{\sss\otimes l},
\end{equation}
where $\mu^{\sss\otimes l}$ denotes the $l$-fold product measure. A sequence $\{(X_n,\dst_n,\mu_n)\}_{n\geq 1}\subset \mathscr{S}_*$ is said to converge to $(X,\dst,\mu)\in \mathscr{S}_*$ if and only if $\Phi((X_n,\dst_n,\mu_n))\to \Phi((X,\dst,\mu))$ for all polynomials $\Phi$ on~$\mathscr{S}_*$. By \cite[Theorem 1]{GPW09},  $\mathscr{S}_*$ is a Polish space under the Gromov-weak topology.

\subsection{Super graphs}\label{defn:super-graph} 
Our \emph{super graphs} consist of three main ingredients: 1) A collection of metric spaces called \emph{blobs}; 2) A graphical superstructure determining the connections between the blobs; 3) Connection points or junction points at each blob. In more detail,  super graphs contain the following structures (see Figure~\ref{fig:supergraph}):
\begin{figure}
    \centering
    \includegraphics[scale = 1.4]{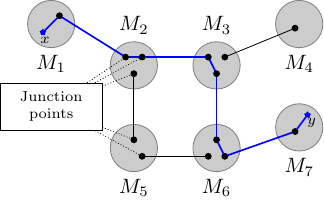}
    \caption{Construction of supergraphs described Section~\ref{defn:super-graph}. The blue line represents the shortest  path between $x$ and $y$.}
    \label{fig:supergraph}
\end{figure}
 \begin{enumerate}[(a)]
 \item \label{defn:blob:supstrct:2} {\bf Blobs}: A collection $\{(M_i,\dst_i,\mu_i)\}_{i\in [m]}$ of  connected, compact measured metric spaces.
  \item {\bf Superstructure}: A (random) graph $\mathcal{G}$ with vertex set $[m]$. The graph has a weight sequence $\mathbf{p}=(p_i)_{i\in [m]}$ associated to the vertex set $[m]$. We regard $M_i$ as the $i$-th vertex of $\mathcal{G}$. 
  \item {\bf Junction points}:  An independent collection of random points $\mathbf{X}:=(X_{i,j}:i,j\in [m])$ such that $X_{i,j}\sim \mu_i$ for all $i,j$. Further, $\mathbf{X}$ is independent of~$\mathcal{G}$.
 \end{enumerate} 
 Using these three ingredients, define a metric space $(\bar{M},\bar{\dst}, \bar{\mu})=\Gamma(\mathcal{G},\mathbf{p}, \mathbf{M},\mathbf{X})$, with $\bar{M}= \sqcup_{i\in [m]}M_i$  (the disjoint union of the $M_i$'s) by putting an edge of \emph{length} one between the pair of points
$\{(X_{i,j},X_{j,i}): (i,j) \text{ is an edge of }\mathcal{G}\}.$
 The distance metric $\bar{\dst}$ is the natural metric obtained from the graph distance and the inter-blob distance on a path. More precisely, for any $x,y\in \bar{M}$ with $x\in M_{j_1}$ and $y\in M_{j_2}$,
 \begin{equation}\label{defn:distance-metric-blob}
  \bar{\mathrm{d}}(x,y)= \inf \Big\{k+\dst_{j_1}(x,X_{j_1,i_1})+\sum_{l=1}^{k-1} \dst_{i_l}(X_{i_l,i_{l-1}}, X_{i_l,i_{l+1}})+\dst_{j_2}(X_{j_2,i_{k-1}},y)\Big\},
 \end{equation}where the infimum is taken over all paths $(i_1,\dots,i_{k-1})$ in $\mathcal{G}$ and all $k\geq 1$, and $i_0= j_1$ and $i_{k} = j_2$.
  The measure $\bar{\mu}$ is given by  $\bar{\mu}(A):=\sum_{i\in [m]}p_i\mu_i(A \cap M_i)$, for any measurable subset $A$ of $\bar{M}$. Note that there is a one-to-one correspondence between the components of $\mathcal{G}$ and $\Gamma(\mathcal{G},\mathbf{p},\mathbf{M},\mathbf{X})$ as the blobs are connected. 
  \subsection{Space of trees with edge lengths, leaf weights, root-to-leaf measures, and blobs}\label{sec:tree-space}
In the proof of the main results we need the following spaces built on top of the space of discrete trees. 
The first space $T_{\sss IJ}$ was formulated in \cite{AP99,AP00a} where it was used to study trees spanning a finite number of random points sampled from an inhomogeneous continuum random tree (as described in the next section).

\subsubsection{The space \texorpdfstring{$T_{\sss IJ}$}{TEXT}} Fix $I\geq 0$ and $J\geq 1$. Let $T_{\sss IJ}$ be the space of trees with each element $\vt\in T_{\sss IJ}$ having the following properties:
\begin{enumeratea}
	\item There are exactly $J$ leaves labeled $1+, \ldots, J+$, and the tree is rooted at the labeled vertex $0+$.
	\item There may be extra labeled vertices (called hubs) with  labels in $\set{1,\ldots, I}$. (It is possible that only some, and not all, labels in $\set{1,\ldots, I}$ are used.)
	\item Every edge $e$ has a strictly positive edge length $l_e$.
\end{enumeratea}
A tree $\vt\in T_{\sss IJ}$ can be viewed as being composed of two parts: 
(1) $\shape(\vt)$ describing the shape of the tree (including the labels of leaves and hubs) but ignoring edge lengths. The set of all possible shapes $T_{\sss IJ}^{\sss \shape}$ is obviously finite for fixed $I, J$.
(2) The edge lengths $\vl(\vt):= (l_e:e\in \vt)$. 
We will consider the product topology on $T_{\sss IJ}$ consisting of the discrete topology on $T_{\sss IJ}^{\sss \shape}$ and the product topology on $\bR^{\sss \mathrm{E}(\vt)}$, where $\mathrm{E}(\vt)$ is the number of edges of~$\vt$.

\subsubsection{The space \texorpdfstring{$T_{\sss IJ}^*$}{TEXT}} \label{sec:space-of-trees}
Along with the three attributes above in $T_{\sss IJ}$, the trees in $T_{\sss IJ}^*$ have two additional properties. Let $\cL(\vt):= \set{1+, \ldots, J+}$ denote the collection of leaves in~$\vt$. 
Then every leaf $v \in \cL(\vt) $ has the following attributes:

\begin{enumeratea}
	\item[(d)] {\bf Leaf weights:} A strictly positive number $A(v)$. 
	\item[(e)] {\bf Root-to-leaf measures:} A probability measure $\nu_{\vt,v}$ on the path $[0+,v]$ connecting the root and the leaf $v$. 
\end{enumeratea}
For each $v \in \cL(\vt) $, the path $[0+,v]$  can be viewed as a compact measured metric space with the measure being  $\nu_{\vt,v}$.
Let $\cX$ denote the space of compact measured metric spaces endowed with the Gromov-Hausdorff-Prokhorov topology (see \cite[Section~2.1.1]{BHS15}).
In addition to the topology on $T_{\sss IJ}$, the space $T_{\sss IJ}^*$ with the additional two attributes inherits the product topology on $\bR^{\sss J}$ due to leaf weights and $\mathcal{X}^{\sss J}$ due to the paths $[0+,v]$ endowed with $\nu_{\vt,v}$ for each $v \in \cL(\vt) $.
For consistency, we add a conventional state $\partial$ to the spaces $T_{\sss IJ}$ and $T_{\sss IJ}^*$. 
Its use will be made clear in Section~\ref{sec:univ-thm}.  
  
   For all instances in this paper, the shape of a tree $\mathrm{shape}(\mathbf{t})$ will be viewed as a subgraph of a graph with $m$ vertices. In that case, the tree will be assumed to inherit the vertex labels from the original graph. We will often write $\mathbf{t}\in T^{*m}_{\sss IJ}$ to emphasize the fact that the vertices of $\mathbf{t}$ are labeled from a subset of $[m]$.

 \subsubsection{The space \texorpdfstring{$\overline{T}^{*m}_{\sss IJ}$}{TEXT}}
 \label{sec:T-space-extended} We enrich the space $T^{*m}_{\sss IJ}$ with some additional elements to accommodate the blobs. 
 Consider $\mathbf{t}\in T_{\sss IJ}^{*m}$ and construct $\bar{\vt}$ as follows:
Let $(M_i,\dst_i,\mu_i)_{i\in [m]}$ be a collection of blobs and $\mathbf{X}=(X_{ij}:i,j\in [m])$ be the collection of junction points as defined in Section~\ref{defn:super-graph}. Construct the metric space $\bar{\vt}$ with elements in $\bar{M}(\mathbf{t})= \sqcup_{i \in \mathbf{t}}M_i$, by putting an edge of \emph{length} one between the pair of vertices $ \{(X_{i,j},X_{j,i}): (i,j) \text{ is an edge of }\mathbf{t}\}.$ 
The distance metric is given by~\eqref{defn:distance-metric-blob}.
 The path from the leaf $v$ to the root $0+$ now contains blobs. 
 Replace the root-to-leaf measure by $\bar{\nu}_{\vt,v}(A):= \sum_{i\in [0+,v]}\nu_{\vt,v}(i)\mu_i(M_i\cap A)$ for $A\subset \sqcup_{i\in [0+,v]} M_i$, where $\nu_{\vt,v}$ is the root-to-leaf measure on $[0+,v]$ for $\mathbf{t}$. 
 Notice that $T_{\sss IJ}^{*m}$ can be viewed as a subset of $\overline{T}_{\sss IJ}^{*m}$. 
 In the proof of the universality theorem in Section~\ref{sec:univ-thm}, the blobs will be a fixed collection and, therefore, any $\vt\in T_{\sss IJ}^{*m}$ corresponds to a unique $\bar{\vt}\in\overline{T}_{\sss IJ}^{*m}$. 
  
%
%
%
%

\subsection{\texorpdfstring{$\mathbf{p}$}{TEXT}-trees}\label{sec:p-trees}
For fixed $m \geq 1$, write $\bT_m$ and $\bT_m^{\sss \ord}$ for the collection of all rooted trees with vertex set $[m]$ and rooted ordered trees with vertex set $[m]$ respectively. 
An ordered rooted tree is a rooted tree where children of each individual are assigned an order.
We define a random tree model called $\vp$-trees \cite{CP99,P01}, and their corresponding limits, the so-called inhomogeneous continuum random trees, which play a key role in describing the limiting metric spaces.   
Fix $m \geq 1$, and a probability mass function $\vp = (p_i)_{i\in [m]}$ with $p_i > 0$ for all $i\in [m]$.
 A $\vp$-tree is a random tree in $\bT_m$, with law as follows: For any fixed $\vt \in \bT_m$ and $v\in \vt$, write $d_v(\vt)$ for the number of children of $v$ in the tree $\vt$. Then the law of the $\vp$-tree, denoted by $\pr_{\text{tree}}$, is defined as
\begin{equation}
\label{eqn:p-tree-def}
	\pr_{\text{tree}}(\vt) = \pr_{\text{tree}}(\vt; \vp) = \prod_{v\in [m]} p_v^{d_v(\vt)}, \quad \vt \in \bT_m.
\end{equation}
Note that a normalizing constant is not required in \eqref{eqn:p-tree-def} to make it a probability distribution (see  \cite[Lemma 1]{CP99}). 
Generating a random $\vp$-tree $\mathscr{T}\sim \pr_{\text{tree}}$ and then assigning a uniform random order on the children of every vertex $v\in \mathscr{T}$ gives a random element with law $\pr_{\ord}(\cdot ; \vp)$ given by
\begin{equation}
\label{eqn:ordered-p-tree-def}
	\pr_{\ord}(\vt) = \pr_{\ord}(\vt; \vp) = \prod_{v\in [m]} \frac{p_v^{d_v(\vt)}}{(d_v(\vt)) !}, \quad \vt \in \bT_m^{\ord}.
\end{equation}

\subsubsection{The birthday construction of \texorpdfstring{$\vp$}{TEXT}-trees} \label{sec:p-tree-birthday}
We now describe a construction of $\vp$-trees, formulated in~\cite{CP99}, that is relevant to this work.  
Let $\vY:=(Y_0, Y_1, \ldots)$ be a sequence of i.i.d.~random variables with distribution~$\vp$. 
Let $R_0=0$ and for $l\geq 1$, let $R_l$ denote the $l$-th repeat time, i.e., $R_l=\min\big\{k>R_{l-1}: Y_k\in\{Y_0,\hdots,Y_{k-1}\}\big\}.$
Now consider the directed graph formed via the edges
$
\cT(\vY):= \set{(Y_{j-1}, Y_j): Y_j\notin \set{Y_0, \ldots, Y_{j-1}}, j\geq 1}.$
This gives a tree which we view as rooted at $Y_0$. 
The following striking result was shown in \cite{CP99}:
\begin{thm}[{\cite[Lemma~1 and Theorem~2]{CP99}}]
	\label{thm:pit-cam-birthday}
	The random tree $\cT(\vY)$, viewed as an element in $\bT_m$, is distributed as a $\vp$-tree with distribution \eqref{eqn:p-tree-def} independently of $Y_{R_1-1}, Y_{R_2-1}, \ldots$ which are i.i.d.~with distribution~$\vp$.
\end{thm}

 \begin{remark}
 	\label{rem:p-tree-dist} The independence between the sequence $Y_{R_1-1}, Y_{R_2-1}, \ldots$ and the constructed $\vp$-tree $\cT(\vY)$ is truly remarkable. 
In particular,  let $\cT_r\subset \cT(\vY)$ denote the subtree with vertex set $\set{Y_0, Y_1, \ldots, Y_{R_r-1}}$, i.e., the tree constructed in the first $R_r$ steps. 
Further take $\tilde{\mathbf{Y}}=(\tilde{Y_1}, \ldots, \tilde{Y_r}) $ to be an i.i.d.~sample from $\vp$ and then construct the subtree $\mathcal{S}_r$ spanned by $\tilde{\mathbf{Y}}$.
Then the above result (formalized as \cite[Corollary 3]{CP99}) implies that 
%
	\begin{equation}
	\label{eqn:p-tree-joint-sample}
		(\tilde{Y_1}, \tilde{Y_2},\ldots, \tilde{Y_r}; \cS_r)\stackrel{d}{=} (Y_{R_1-1}, Y_{R_2-1}, \ldots, Y_{R_r -1}; \cT_r).
	\end{equation}
	We will use this fact in Section~\ref{sec:univ-thm} to complete the proof of the universality theorem.
 \end{remark}

\subsubsection{Tilted \texorpdfstring{$\mathbf{p}$}{TEXT}-trees and connected components of \texorpdfstring{$\mathrm{NR}_n(\bld{x},t)$}{TEXT}} 
Consider the vertex set $[n]$ and assign weight $x_i$ to vertex $i$. 
Now, connect each pair of vertices $i, j$ ($i\neq j$) independently with probability $
		q_{ij}:= 1-\exp(- tx_i x_j).$
The resulting random graph, denoted by $\mathrm{NR}_n(\bld{x},t)$, is known as the Norros-Reittu model or the Poisson graph process \cite{RGCN1}.
For a connected component $\cC\subseteq \mathrm{NR}_n(\bld{x},t)$, let $\mass(\cC):= \sum_{i\in \cC} x_i$ and, for any $t\geq 0$, let  $(\cC_i(t))_{i\geq 1}$ denote the components in decreasing order of their mass sizes. 
In this section, we describe results from \cite{BSW14} that give a method of constructing  connected components of $\mathrm{NR}_n(\bld{x},t)$, conditionally on the vertices of the components. 
This construction involves tilted versions of $\vp$-trees introduced in Section \ref{sec:p-trees}. 
Since these trees are parametrized via a driving probability mass function (pmf) $\vp$, it will be easy to parametrize various random graph constructions in terms of pmfs as opposed to the vertex weights~$\bld{x}$. Proposition~\ref{prop:generate-nr-given-partition} will relate vertex weights to pmfs.

Fix $n\geq 1$ and $\cV\subset [n]$, and write $\bG_{\cV}^{\con}$ for the space of all simple connected graphs with vertex set~$\cV$.
For fixed $a > 0$, and probability mass function $\vp = (p_v)_{ v \in \cV}$, define probability distributions $\pr_{\con}(\cdot; \vp, a, \cV)$ on $\bG_{\cV}^{\con}$ as follows: For $i,j \in \cV$, denote
\begin{equation}
\label{eqn:qij-def-vp}
	q_{ij}:= 1-\exp(-a p_i p_j).
\end{equation}
Then, for $G \in \bG_{\cV}^{\con},$
\begin{equation}
	\label{eqn:pr-con-vp-a-cV-def}
	\pr_{\con}(G; \vp, a, \cV): = \frac{1}{Z(\vp,a)} \prod_{(i,j)\in E(G)} q_{ij} \prod_{(i,j)\notin E(G)} (1-q_{ij}), 
\end{equation}
where $Z(\vp,a)$ is the normalizing constant.
Now let $\cV^{\sss(i)} $ be the vertex set of $\cC_i(t)$ for $i \geq 1$, and note that $(\cV^{\sss(i)})_{i\geq 1}$ denotes a random finite partition of the vertex set $[n]$. 
The next proposition yields a construction of the random (connected) graphs $(\cC_{i}(t))_{i\geq 1}$:

\begin{prop}[{\protect{\cite[Proposition 6.1]{BSW14}}}]
	\label{prop:generate-nr-given-partition}
	Given the partition $(\cV^{\sss(i)})_{i\geq 1}$, define,  for $i\geq 1$,
	\begin{equation}\label{eq:p-n-a-NR}
		\vp_n^{\sss(i)} := \left( \frac{x_v}{\sum_{v \in \cV^{\sss(i)}}x_v } : v \in \cV^{\sss(i)} \right), \quad  a_n^{\sss(i)}:= t\bigg(\sum_{v\in \cV_{\sss(i)}} x_v\bigg)^2.
	\end{equation}
	For each fixed $i \geq 1$, let $G_i \in  \bG_{\cV^{\sss(i)}}^{\con}$ be a connected simple graph with vertex set $\cV^{\sss(i)}$. Then
	\begin{equation}
		\pr\left(\cC_i(t) = G_i, \;\; \forall i \geq 1\ \big|\ (\cV^{\sss(i)})_{i\geq 1} \right) = \prod_{i\geq 1} \pr_{\con}( G_i; \vp_n^{\sss(i)}, a_n^{\sss(i)}, \cV^{\sss(i)}).
	\end{equation}
\end{prop}
\noindent Proposition~\ref{prop:generate-nr-given-partition} yields the following construction of $\mathrm{NR}_n(\bld{x},t)$:
\begin{algo}\normalfont
The random graph $\mathrm{NR}_n(\bld{x},t)$ can be generated in two stages:
\begin{enumeratei}
	\item[(S0)] Generate the random partition $(\cV^{\sss(i)})_{i\geq 1}$ of the vertices into different components.
\item[(S1)] Conditionally on the partition, generate the internal structure of each component following the law of $\pr_{\con}(\cdot ; \vp^{\sss(i)}, a^{\sss(i)}, \cV^{\sss(i)})$, independently across different components.
\end{enumeratei}
\end{algo}
\noindent Let us now describe an algorithm to generate such  connected components using distribution~\eqref{eqn:pr-con-vp-a-cV-def}. To ease notation, let $\cV = [m]$ for some $m\geq 1$ and fix a probability mass function $\vp$ on $[m]$ and a constant $a>0$ and write $\pr_{\con}(\cdot):= \pr_{\con}(\cdot;\vp,a,[m])$ on $\bG_m^{\con}:= \bG_{[m]}^{\con}$. 
To generate a sample $G$ from $\pr_{\con}$, one needs to first generate a $\vp$-tree (with suitable tilt). 
The rest of the edges of $G$ are \emph{surplus edges}, which are generated by connecting the leaves to one of the vertices in their path to the root. Let us now describe this process formally.
As a matter of convention, we view ordered rooted trees via their planar embedding using the associated ordering to determine the relative locations of siblings of an individual. 
We think of the left most sibling as the "oldest". 
Further, in a depth-first exploration, we explore the tree from left to right. 
Now given a planar rooted tree $\vt\in \bT_m$, let $\rho$ denote the root, and for every vertex $v\in [m]$, let $[\rho,v]$ denote the path connecting $\rho$ to $v$ in the tree. Given this path and a vertex $i\in [\rho,v]$, write $\RC(i,[\rho,v])$ for the set of all children of $i$ that fall to the right of $[\rho,v]$. 
Define $\fP(v,\vt):= \cup_{i\in [m]} \RC(i,[\rho,v]).$
	In the terminology of \cite{ABG09,BHS15}, $\fP(v,\vt)$ denotes the set of endpoints of all \emph{permitted edges} emanating from $v$. 
	Define
	\begin{equation}
	\label{eqn:amv-def}
		\dA_{\sss(m)}(v):= \sum_{i\in [\rho,v]} \sum_{j\in [m]} p_j \ind{j\in \RC(i,[\rho,v])}.
	\end{equation}
Let $(v(1), v(2), \ldots, v(m))$ denote the order of the vertices in the depth-first exploration of the tree~$\vt$. 
Let $y^*(0)=0$ and $y^*(i) = y^*(i-1) + p_{v(i)}$ and define
	\begin{equation}
	\label{eqn:ant-def-new}
		A_{\sss(m)}(u) = \dA_{\sss(m)}( v(i)),\ \text{ for } u\in (y^*(i-1), y^*(i)],\quad\text{and}\quad \bar{A}_{\sss(m)}(\cdot):= a A_{\sss(m)}(\cdot),
	\end{equation}where $a$ is defined in \eqref{eqn:qij-def-vp}.
Define the function
		\begin{equation}
		\label{eqn:Lambda-tree}
			\Lambda_{\sss(m)}(\vt) := a\sum_{v\in [m]} p_v \dA_{\sss(m)}(v).
		\end{equation}
Finally, let $ \mathrm{E}(\vt)$ denote the set of edges of $\vt$, $\mathscr{T}_m^{\mathbf{p}}$ the $\vp$-tree defined in \eqref{eqn:ordered-p-tree-def}, $\fP(\vt)=\cup_{v\in [m]}\fP(v,\vt)$, and define the function  $L : \bT_m^{\ord} \to \bR_+$ by
\begin{equation}
\label{eqn:ltpi-def}
	\displaystyle L(\vt)=\displaystyle L_{\sss (m)}(\vt):= \prod_{(k,\ell)\in \mathrm{E}(\vt)} \left[\frac{\exp(a p_k p_{\ell})- 1}{ap_k p_{\ell}} \right] \exp\bigg(\sum_{(k,\ell) \in \sP(\vt)} a p_k p_{\ell}\bigg),
\end{equation} for $\vt \in \bT_m^{\ord}$.
Recall the (ordered) $\vp$-tree distribution from \eqref{eqn:ordered-p-tree-def}. 
Let $\mathscr{T}^{\vp}_m$ be a sample from $\pr_{\ord}$.
Using $L(\cdot)$ to tilt this distribution results in the distribution
\begin{equation}
	\label{eqn:tilt-ord-dist-def}
	\pr_{\ord}^\star( \vt) := \pr_{\ord}(\vt) \cdot \frac{L(\vt)}{\E_{\ord}[ L(\mathscr{T}^{\vp}_m)]}, \qquad \vt \in \bT_m^{\ord}.
\end{equation}

In the algorithm below, all the objects depend on the tree $\vt$, but we often suppress this dependence to ease notation. 


\begin{algo}\label{algo:construction-Pcon}\normalfont
Let $\tilde{\mathcal{G}}_m(\vp,a)$ denote a random graph sampled from $\PR_{\mathrm{con}}(\cdot)$. 
This algorithm gives a construction of $\tilde{\mathcal{G}}_m(\vp,a)$, proved in \cite{BHS15}:
\begin{enumeratei}
	\item[(S1)] {\bf Tilted $\vp$-tree:}
	Generate a tilted ordered $\vp$-tree $\mathscr{T}^{\vp,\star}_m$ with distribution \eqref{eqn:tilt-ord-dist-def}. Now consider the (random) objects $\fP(v,\mathscr{T}^{\vp,\star}_m)$ for $v\in [m]$ and the corresponding (random) functions $\dA_{\sss(m)}(\cdot)$ on $[m]$ and $A_{\sss(m)}(\cdot)$ on $[0,1]$.
		\item[(S2)] {\bf Poisson number of possible surplus edges:} Let $\cP$ denote a rate-one Poisson process on $\bR_+^2$ and define
	\begin{equation}
	\label{eqn:poisson-pp}
		\bar{A}_{\sss(m)}\cap {\cP}:= \set{(s,t)\in \cP: s\in [0,1], t\leq \bar{A}_{\sss(m)}(s)}.
	\end{equation}
	 Write $\bar{A}_{\sss(m)}\cap {\cP}:= \{(s_j,t_j):1\leq j\leq N_{\sss(m)}^\star\}$ where $N_{\sss(m)}^\star = |\bar{A}_{\sss(m)}\cap {\cP}|$.
We next use the set $\{(s_j, t_j):1\leq j\leq N_{\sss(m)}^\star\}$ to generate pairs of points $\set{(\cL_j,\cR_j): 1\leq j\leq N_{\sss(m)}^\star}$ in the tree that will be joined to form the surplus edges.
	 \item[(S3)] {\bf First endpoints:} Fix $j$ and suppose $s_j \in (y^*(i-1), y^*(i)]$ for some $i\geq 1$, where $y^*(i)$ is as given right above \eqref{eqn:ant-def-new}. Then the \emph{first endpoint} of the surplus edge corresponding to $(s_j, t_j)$ is $\cL_j:= v(i)$.
	 \item[(S4)] {\bf Second endpoints:} Note that in the interval $(y^*(i-1), y^*(i)]$, the function $\bar{A}_{\sss(m)}$ is of constant value $a\dA_{\sss(m)}(v(i))$. We will view this value or height as being partitioned into sub-intervals of length $a p_u$ for each $u\in \fP(v(i),\mathscr{T}^{\vp,\star}_m)$, the collection of endpoints of permitted edges emanating from $\cL_k$. (Assume that this partitioning is done according to some preassigned rule, e.g., using the order of the vertices in $\fP(v(i),\mathscr{T}^{\vp,\star}_m)$). 
	 Suppose that $t_j$ belongs to the interval corresponding to $u$. Then the \emph{second endpoint} is $\cR_j = u$. Form an edge between $(\cL_j, \cR_j)$.
\item[(S5)] In this construction, it is possible that one creates more than one surplus edge between two vertices. Remove any multiple surplus edges. 
This has vanishing probability in our applications.
\end{enumeratei}
\end{algo}
\begin{defn}\label{defn:p-tree-graph}
Consider the connected random graph $\tilde{\mathcal{G}}_m(\mathbf{p},a)$, given by {\rm Algorithm~\ref{algo:construction-Pcon}}, viewed as a measured metric space via the graph distance and each vertex $v$ is assigned measure $p_v$.
\end{defn}

\begin{lemma}[{\cite[Lemma 4.10]{BHS15}}]
	\label{lem:lk-rk-equivalent}
	The random graph $\tilde{\cG}_m(\vp,a)$ generated by {\rm Algorithm~\ref{algo:construction-Pcon}} has the same law as $\PR_{\mathrm{con}}(\cdot)$. 
	Further, conditionally on $\mathscr{T}^{\vp,\star}_m$, the following hold:
	\begin{enumeratea}
		\item $N_{\sss(m)}^\star$ has Poisson distribution with mean $\Lambda_{\sss(m)}(\mathscr{T}_m^{\vp,\star})$, where $\Lambda_{\sss(m)}$ is as in \eqref{eqn:Lambda-tree}.
		\item Conditionally on $\mathscr{T}_m^{\vp,\star}$ and $N_{\sss(m)}^\star=k$, the first endpoints $(\cL_j)_{j\in [k]}$ can be generated in an i.i.d.~fashion by sampling from the vertex set $[m]$ with probability distribution
		$
		\cJ^{\sss(m)}(v) \propto p_v \dA_{\sss(m)}(v),$ $v\in [m].
		$
		\item Conditionally on $\mathscr{T}_m^{\vp,\star}$, $N_{\sss(m)}^\star=k$ and the first endpoints $(\cL_j)_{j\in [k]}$, generate the second endpoints in an i.i.d.~fashion where conditionally on $\cL_j = v$, the probability distribution of $\cR_j$ is given by
	\begin{equation}
	\label{eqn:right-end-pt-prob}
		Q_{v}^{\sss(m)}(y):= \begin{cases}
			\sum_{u} p_u \ind{u\in \RC(y,[\rho,v])}/\dA_{\sss(m)}(v) & \text{ if } y\in [\rho,v],\\
			0 & \text{ otherwise}.
		\end{cases}
	\end{equation}	
 Create an  edge between $\cL_j$ and $\cR_j$ for $1\leq j\leq k$.
	\end{enumeratea}
\end{lemma}

\subsection{Inhomogeneous continuum random trees}
\label{sec:inhom-cont-tree}
In a series of papers \cite{AMP04,AP99,AP00a} it was shown that $\vp$-trees, under various assumptions, converge to inhomogeneous continuum random trees (ICRTs) that we now describe. 
Recall from \cite{E06,LG05} that a real tree is a metric space $(\mathscr{T},d)$ that satisfies the following for every pair $a,b\in \mathscr{T}$:
\begin{enumeratea}
	\item There is a {\bf unique} isometric map $f_{a,b}\colon [0,d(a,b)]\to \mathscr{T}$ such that $f_{a,b}(0)=a$ and  $f_{a,b}(d(a,b)) =b$.
	\item For any continuous one-to-one map $g:[0,1]\to \mathscr{T}$ with $g(0)=a$ and $g(1)=b$, we have $g([0,1]) = f_{a,b}([0,d(a,b)])$.
\end{enumeratea}

\noindent {\bf Construction of the ICRT:}
Given $\bld{\beta}\in\ell^2_{\shortarrow}\setminus \ell^1_{\shortarrow}$ with $\sum_i\beta_i^2=1$, we will now define the inhomogeneous continuum random tree $\mathscr{T}^{\bld{\beta}}$. We mainly follow the notation in \cite{AP00a}. Assume that we are working on a probability space $(\Omega, \cF,\pr_{\bld{\beta}})$ rich enough to support the following:
\begin{enumeratea}
	\item For each $i\geq 1$, let $\cP_i:= (\xi_{i,1}, \xi_{i,2}, \ldots)$ be rate $\beta_i$ Poisson processes that are independent for different~$i$. The first point of each process $\xi_{i,1}$ is special and is called a \emph{joinpoint}, while the remaining points $\xi_{i,j}$ with $j\geq 2$ will be called \emph{$i$-cutpoints} \cite{AP00a}.
	\item Independently of the above, let $\mvU=(U_j^{\sss(i)})_{i,j\geq 1}$ be a collection of i.i.d\ uniform $(0,1)$ random variables. These are not required to construct the tree but will be used to define a certain function on the tree.
\end{enumeratea}
  The {\bf random} real tree (with marked vertices) $\icrt$ is then constructed as follows:
\begin{enumeratei}
	\item  Arrange the cutpoints $\set{\xi_{i,j}: i\geq 1, j\geq 2}$ in increasing order as $0< \eta_1 < \eta_2 < \cdots$. The assumption that $\sum_i \beta_i^2 <\infty$ implies that this is possible. For every cutpoint $\eta_k=\xi_{i,j}$, let $\eta_k^*:=\xi_{i,1}$ be the corresponding joinpoint.
	\item Next, build the tree inductively. Start with the branch $[0,\eta_1]$. Inductively assuming that we have completed step $k$, attach the branch $(\eta_k, \eta_{k+1}]$ to the joinpoint $\eta_k^*$ corresponding to $\eta_k$.
\end{enumeratei}
Write $\mathscr{T}_0^{\bld{\beta}}$ for the corresponding tree after one has used up all the branches $[0,\eta_1]$, $\set{(\eta_k, \eta_{k+1}]: k\geq 1}$.
Note that for every $i\geq 1$, the joinpoint $\xi_{i,1}$ corresponds to a vertex with infinite degree. Label this vertex $i$. The ICRT $\icrt$ is the completion of the marked metric tree $\mathscr{T}^{\bld{\beta}}_0$. As argued in \cite[Section~2]{AP00a}, this is a real-tree as defined above which can be viewed as rooted at the vertex corresponding to zero. We call the vertex corresponding to joinpoint $\xi_{i,1}$  {\bf hub} $i$. 
    \begin{figure}
        \centering \includegraphics[scale = .7]{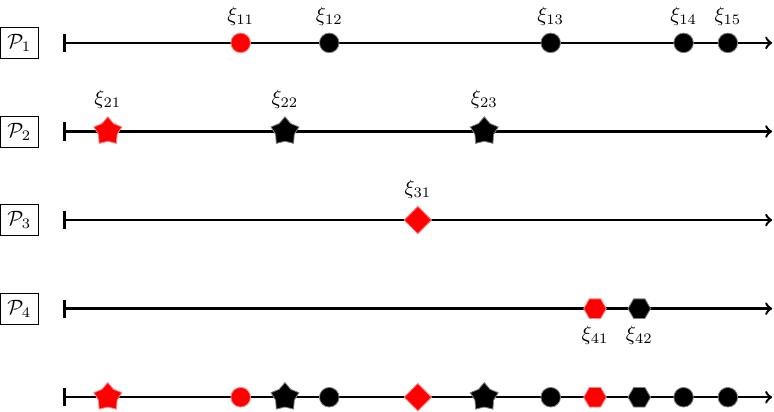}
        \caption{An illustration of the ICRT construction with four point process $\set{\cP_i:1\leq i\leq 4}$. The red points represent the \emph{joinpoint} of the corresponding point process and the blue points the corresponding cutpoints. The last line contains the union of the four point processes. See Figure \ref{fig:tinf-theta} for the corresponding tree.}
           \label{fig:point-process-icrt}
    \end{figure}
    \begin{figure}
        \centering
	\includegraphics[scale= 1]{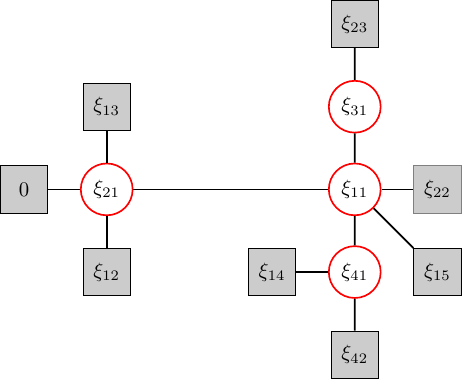}
        \caption{The tree constructed via the stick-breaking construction from Figure \ref{fig:point-process-icrt}.}
        \label{fig:tinf-theta}
    \end{figure}

The uniform random variables $(U_j^{\sss(i)})_{i,j\geq 1}$ give rise to a natural ordering on $\icrt$ (or a planar embedding of $\icrt$) as follows: 
For $i\geq 1$, let $(\mathscr{T}_j^{\sss(i)})_{j\geq 1}$ be the collection of subtrees hanging off the $i$-th hub. Associate $U_j^{\sss(i)}$ with the subtree $\mathscr{T}_j^{\sss(i)}$, and think of $\mathscr{T}_{j_1}^{\sss(i)}$ appearing "to the right of" $\mathscr{T}_{j_2}^{\sss(i)}$ if $U_{j_1}^{\sss(i)}< U_{j_2}^{\sss(i)}$. This is the natural ordering on $\icrt$ when it is being viewed as a limit of ordered $\vp$-trees. We can think of the pair $(\icrt, \mvU)$ as the {\bf ordered ICRT}.

\subsection{Continuum limits of components}
\label{sec:descp-limit}
The aim of this section is to give an explicit description of the limiting (random) metric spaces in Theorem \ref{thm:main}.  We start by constructing a specific metric space using the tilted version of the ICRT in Section~\ref{sec:tilt-icrt}. 
Then we describe the limits of maximal components in Section~\ref{sec:limit-component}.

\subsubsection{Tilted ICRTs and vertex identification}
\label{sec:tilt-icrt}
Let $(\Omega, \cF, \pr_{\beta})$ and $\icrt$ be as in Section \ref{sec:inhom-cont-tree}. 
In \cite{AP00a}, it was shown that one can associate a natural probability measure $\mu$, called the {\bf mass measure}, to $\icrt$, satisfying $\mu(\cL(\icrt))=1$. 
Here we recall that $\mathcal{L}(\cdot)$ denotes the set of leaves. 
Before moving to the desired construction of the random metric space, we will need to define some more quantities that describe the asymptotic analogues of the quantities appearing in Algorithm~\ref{algo:construction-Pcon}.
Similarly to \eqref{eqn:amv-def}, define
	$\dA_{\sss(\infty)}(y)=\sum_{i\geq 1}\beta_{i}\big(\sum_{j\geq 1}U_j^{\sss(i)}\times\mathbbm{1}\{y\in \mathscr{T}_j^{\sss(i)} \}\big).$
It was shown in \cite{BHS15} that $\dA_{\sss(\infty)}(y)$ is finite for almost every realization of $\icrt$ and for $\mu$-almost every $y\in\icrt$. For $y\in \icrt$, let $[\rho,y]$ denote the path from the root $\rho$ to $y$. For every $y$, define a probability measure on $[\rho,y]$ as
\begin{equation}
\label{eqn:right-end-prob-inft}
Q_{y}^{\sss(\infty)}(v):= \frac{\beta_i U_{j}^{\sss(i)}}{ \dA_{\sss(\infty)}(y)}, \quad\mbox{ if } v\mbox{ is the }i\mbox{-th hub and } y\in \mathscr{T}_j^{\sss(i)}\mbox{ for some }j. 	
\end{equation}
Thus, this probability measure is concentrated on the hubs on the path from $y$ to the root.
Let $\gamma >0$ be a constant. 
Informally, the construction goes as follows:
We will first tilt the distribution of the original ICRT $\icrt$ using the exponential functional
	\begin{equation}
	\label{eqn:ltheta-def}
		L_{\sss(\infty)}(\icrt, \mvU):= \exp\bigg(\gamma\int_{y\in \icrt} \dA_{\sss(\infty)}(y)\mu(dy)\bigg)
	\end{equation}
	to get a tilted tree $\tilicrt$. 
	We then generate a random but finite number $N_{\sss(\infty)}^\star$ of pairs of points $\{(x_k, y_k):1\leq k\leq N_{\sss(\infty)}^\star\}$ that will provide the surplus edges. 
	The final metric space is obtained by creating \emph{shortcuts} by identifying the points $x_k$ and $y_k$.
	The construction mimics that of Algorithm~\ref{algo:construction-Pcon}.
	Formally the construction proceeds in four steps:
\begin{enumeratea}
	\item {\bf Tilted ICRT:}  
Define ${\pr}_{\beta}^\star$ on $\Omega$ by
\begin{equation}
\frac{d {{\pr}}_{\beta}^\star}{d{{\pr}}_{\beta}}=\frac{\exp\big(\gamma\int_{y\in\mathscr{T}^{\bld{\beta}}_{(\infty)}}\dA_{(\infty)}(y)\mu(dy) \big)}{\E\Big[\exp\big(\gamma\int_{x\in \mathscr{T}^{\bld{\beta}}_{(\infty)}} \dA_{\sss(\infty)}(x)\mu(dx) \big)\Big]}. 
\end{equation}
The expectation in the denominator is with respect to the original measure ${\pr}_{\beta}$. 
Write $(\tilicrt, \mu^\star)$ and $\mvU^{\star}=(U_j^{(i), \star})_{i,j\geq 1}$ for the tree and the mass measure on it, and the associated random variables under this change of measure.
 \item {\bf Poisson number of identification points:} Conditionally on $((\tilicrt, \mu^\star), \mvU^{\star})$, generate $N_{\sss(\infty)}^\star$ having a $\mathrm{Poisson}(\Lambda_{\sss(\infty)}^\star)$ distribution, where
\begin{equation}
\Lambda_{\sss(\infty)}^\star:= \gamma\int_{y\in \tilicrt}\dA_{\sss(\infty)}(y)\mu^\star(dy)
=\gamma\sum_{i\geq 1}\beta_{i}\bigg[\sum_{j\geq 1}U_j^{(i), \star}\mu^\star(\mathscr{T}_j^{\sss(i), \star})\bigg].
\end{equation}
Here, $(\mathscr{T}_j^{\sss(i), \star})_{j\geq 1}$ denotes the collection of subtrees of hub $i$ in $\tilicrt$. 

\item {\bf First endpoints (of shortcuts): } Conditionally on (a) and (b), sample $x_k$ from $\tilicrt$  with density proportional to $\dA_{\sss(\infty)}(x)\mu^\star(dx)$ for $1\leq k\leq N_{\sss(\infty)}^\star$.
\item {\bf Second endpoints (of shortcuts) and identification:} Having chosen $x_k$, choose $y_k$ from the path $[\rho, x_k]$ joining the root $\rho$ and $x_k$ according to the probability measure $Q_{x_k}^{\sss (\infty)}$ as in \eqref{eqn:right-end-prob-inft} but with $U_j^{(i),\star}$ replacing $U_j^{\sss (i)}$.
(Note that $y_k$ is always a {\bf hub} on $[\rho, x_k]$.)  Identify $x_k$ and $y_k$, i.e., form the quotient space by introducing the equivalence relation $x_k\sim y_k$ for $1\leq k\leq N_{\sss(\infty)}^\star$.
\end{enumeratea}
\begin{defn}\label{def:limiting-space} Fix $\gamma\geq 0$ and $\bld{\beta}\in\ell^2_{\shortarrow}\setminus \ell^1_{\shortarrow}$ with $\sum_i\beta_i^2=1$. Let $\cG_{\sss (\infty)}(\bld{\beta},\gamma)$ be the metric measure space constructed via the four steps above equipped with the measure inherited from the mass measure on $\tilicrt$.	
\end{defn}
\subsubsection{Scaling limit for the component sizes and surplus edges}
\label{sec:comp-size-scaling}
Let us describe the scaling limit results for the component sizes and the surplus edges ($\#\text{edges}-\#\text{vertices}+1$) for the largest components of $\mathrm{CM}_n(\bld{d},p_n(\lambda))$ from~\cite{DHLS16}. 
Although we need to define the limiting object only for describing the limiting metric space, the convergence result will turn out to be crucial in Section~\ref{sec:proof-metric-mc} in the proof of Theorem~\ref{thm:main}, and therefore we state it here as well.
Consider a decreasing sequence $ \boldsymbol{\theta}\in \ell^3_{\shortarrow}\setminus \ell^2_{\shortarrow}$. 
Denote by  $\mathcal{I}_i(s):=\ind{\zeta_i\leq s }$ where $\zeta_i\sim \mathrm{Exp}(\theta_i)$ independently, and $\mathrm{Exp}(r)$ denotes the exponential distribution with rate $r$.  
 Consider the process 
 \begin{equation}\label{defn::limiting::process}
\bar{S}^\lambda_\infty(t) =  \sum_{i=1}^{\infty} \theta_i\left(\mathcal{I}_i(t)- \theta_it\right)+\lambda t,
\end{equation}for some $\lambda\in\mathbbm{R}$. 
Define the reflected version of $\bar{S}_\infty^{\lambda}(t)$ by
$ \mathrm{refl}\big( \bar{S}_\infty^{\lambda}(t)\big)= \bar{S}_\infty^{\lambda}(t) - \inf_{0 \leq u \leq t} \bar{S}_\infty^{\lambda}(u).$
The processes of the form \eqref{defn::limiting::process} were termed \emph{thinned} L\'evy processes in~\cite{BHL12} since the summands are thinned versions of Poisson processes.
Let $(\Xi_i(\bld{\theta},\lambda))_{i \geq 1}$, $(\xi_i(\bld{\theta},\lambda))_{i\geq 1}$, respectively, denote the vector of excursions and excursion-lengths of $\big(\mathrm{refl}( \bar{S}_\infty^{\lambda}(t))\big)_{t\geq 0}$, ordered according to the excursion lengths in a decreasing manner. 
Using \cite[Fact 1]{DHLS16}, there are no ties among the excursion lengths almost surely.
Denote the vector $(\xi_i(\bld{\theta},\lambda))_{i\geq 1}$ by $\bld{\xi}(\bld{\theta},\lambda)$.
The fact that $\bld{\xi}(\bld{\theta},\lambda)$ is always well defined follows from \cite[Lemma 1]{AL98}.
Also, define the counting process of marks $\mathbf{N}$ to be a Poisson process that has intensity $\refl{ \bar{S}_\infty^{\lambda}(t)}$ at time $t$ conditional on $( \refl{ \bar{S}_\infty^{\lambda}(u)} )_{u \leq t}$. 
We use  the notation $\mathscr{N}_i(\bld{\theta},\lambda)$ to denote the number of marks within $\Xi_i(\bld{\theta},\lambda)$.

For a connected graph $G$, let $\mathrm{SP}(G) = \#\text{edges}-\#\text{vertices}+1$ denote its surplus edges. 
In the context of this paper, we simply write $\xi_i$, $\bld{\xi}$ and $\mathscr{N}_i$ respectively for $\xi_i(\bld{\theta}/(\mu\nu),\lambda/\mu)$, $\bld{\xi}(\bld{\theta}/(\mu\nu),\lambda/\mu)$ and $\mathscr{N}_i(\bld{\theta}/(\mu\nu),\lambda/\mu)$.
\begin{prop}[{\cite[Theorem~4]{DHLS16}}]\label{prop:comp-size}
Under {\rm Assumption~\ref{assumption1}}, as $n\to\infty$,
\begin{equation}\label{eq:limit-objetct-comp-size}
\big(n^{-\rho}|\mathscr{C}_{\sss (i)}^p(\lambda)|,\surp{\mathscr{C}_{\sss (i)}^p(\lambda)}\big)_{i\geq 1} \dto \Big(\frac{1}{\nu}\xi_i,\mathscr{N}_i\Big)_{i\geq 1},
 \end{equation}with respect to the topology on the product space $\ell^2_\shortarrow\times \N^\N$.
\end{prop}
The limiting object in  \cite[Theorem 4]{DHLS16} is stated in a slightly different form compared to the right hand side of \eqref{eq:limit-objetct-comp-size}. 
However, the limiting objects are identical in distribution with suitable rescaling of time and space, and  by observing that $r\mathrm{Exp}(r) \stackrel{\sss d}{=} \mathrm{Exp}(1)$, where $\mathrm{Exp}(r)$ denotes an exponential random variable with rate $r$ (see Appendix~\ref{sec:appendix-rescaling}). In fact, the arguments in Appendix~\ref{sec:appendix-rescaling} establish the following lemma which will be used extensively in Section~\ref{sec:proof-metric-mc}:
\begin{lemma} \label{lem:rescale}  For $\eta_1,\eta_2>0$, $\bld{\theta}\in \ell^3_{\shortarrow}\setminus \ell^2_{\shortarrow}$ and $\lambda\in\R$,
$ \bld{\xi}(\eta_1\bld{\theta},\eta_2\lambda)\eqd \frac{1}{\eta_1}\bld{\xi}\big(\bld{\theta},\frac{\eta_2}{\eta_1^2}\lambda\big).
$
\end{lemma}

\subsubsection{Limiting component structures}
\label{sec:limit-component}
We are now all set to describe the metric space $M_i$ appearing in Theorem~\ref{thm:main}.
Recall the graph $\mathcal{G}_{\infty}(\bld{\beta},\gamma)$ from Definition~\ref{def:limiting-space}.
Using the notation of Section~\ref{sec:comp-size-scaling}, write $\xi_i^*$ for $\xi_i((\mu(\nu-1))^{-1}\bld{\theta}, (\mu(\nu-1)^2)^{-1}\nu^2\lambda)$ and $\Xi_i^*$ for the excursion corresponding to $\xi_i^*$.
Note that $\xi_i^*$ has the same distribution as $(\nu-1) \xi_i/\nu$, where $\xi_i$ is as in Proposition~\ref{prop:comp-size}. 
Then the limiting space $M_i$ is distributed as 
\begin{equation}
 M_i \eqd \frac{\nu}{\nu-1}\frac{\xi_i^*}{\big(\sum_{v\in \Xi_i^*}\theta_v^2\big)^{1/2}} \mathcal{G}_\infty ( \bld{\theta}^{\sss (i)}, \gamma^{\sss (i)}  ),
\end{equation}where $\bld{\theta}^{\sss (i)} = \big(\frac{\theta_j}{\sum_{v\in \Xi_i^*}\theta_v^2}: j\in \Xi_i^*\big)$ and $\gamma^{\sss (i)}=\frac{\xi_i^*}{\mu(\nu-1)}\big(\sum_{v\in \Xi_i^*}\theta_v^2\big)^{1/2}$.

\section{Universality theorem} \label{sec:univ-thm}
In this section, we develop universality principles that enable us to derive the scaling limits of the components for graphs that can be compared with the critical rank-one inhomogeneous random graph in a suitable sense. 
For the scaling limits in the basin of attraction of the Erd\H{o}s-R\'enyi random graphs, such a universality theorem was proved in \cite[Theorem~6.4]{BBSX14}, which was applied to deduce the scaling limits of the components for general inhomogeneous random graphs with a finite number of types and the configuration model with an exponential moment condition on the degrees. 
Here we focus on the universality class of the scaling limits in the heavy-tailed case.
We first state the relevant result from~\cite{BHS15} that was used in the context of rank-one inhomogeneous random graphs and then state our main result below. 
The convergence of metric spaces is with respect to the Gromov-weak topology, unless stated otherwise.
Recall the measured metric spaces $\tilde{\mathcal{G}}_m(\vp,a)$ and $\mathcal{G}_{\infty}(\bld{\beta},\gamma)$ defined in {\rm Definitions~\ref{defn:p-tree-graph}~and~\ref{def:limiting-space}}.
\begin{assumption} \label{assm:BHS15}  \normalfont \begin{enumerate}[(i)]
 \item  Let $\sigmap := \big(\sum_i p_i^2\big)^{1/2}$. As $m\to\infty$, $\sigmap\to 0 $, and $p_i/\sigmap \to \beta_i$ for each fixed $i\geq 1$, where $\bld{\beta}=(\beta_i)_{i\geq 1}\in\ell^2_{\shortarrow}\setminus\ell^1_{\shortarrow}$, $\sum_i\beta_i^2=1$.
 \item Recall $a$ from \eqref{eqn:qij-def-vp}. There exists a constant $\gamma >0$ such that $a\sigmap\to\gamma$. 
\end{enumerate}
\end{assumption}
 \begin{theorem}[{\cite[Theorem 4.5]{BHS15}}] \label{thm:BHS15}
  Under {\rm Assumption~\ref{assm:BHS15}},
$ \sigmap \tilde{\mathcal{G}}_m(\mathbf{p},a) \dto \mathcal{G}_{\sss (\infty)}(\bld{\beta},\gamma)$, as $m\to\infty$.
\end{theorem}

For each $m\geq 1$, fix a collection of blobs $\mathbf{M}_m:= \{(M_i, \dst_i, \mu_i):{i\in [m]} \}$. Recall the definition of super graphs from Section~\ref{defn:super-graph} and denote 
\begin{equation}\label{defn:blob-Gm}
\tilde{\mathcal{G}}_m^{\sss \mathrm{bl}}(\mathbf{p},a) = \Gamma(\tilde{\mathcal{G}}_m(\mathbf{p},a),\mathbf{p},\mathbf{M}_m,\mathbf{X}),
\end{equation}where $\mathbf{X} = (X_{ij})_{i,j\in [m]}$, $X_{ij}\sim \mu_i$ independently for each $i$. 
Moreover, $\mathbf{X}$ is independent of the graph $\tilde{\mathcal{G}}_m(\mathbf{p},a)$.
Let $u_i:=\E[\dst_i(X_i,X_i')]$ where $X_i,X_i'\sim \mu_i$ independently and $B_m:= \sum_{i\in [m]}p_iu_i$. 
Let $\Delta_i:=\diam(M_i) $ and  $\Delta_{\max}:=\max_{i\in [m]}\Delta_i$.

\begin{assumption}[Maximum inter-blob-distance] \label{assm:blob-diameter}
 \normalfont 
 $\lim_{m\to\infty}\frac{\sigma(\mathbf{p})\Delta_{\max}}{B_m+1}=0.$ 
\end{assumption}
\begin{theorem}[Universality theorem]\label{thm:univesalty}
Under {\rm Assumptions~\ref{assm:BHS15}}~and~{\rm \ref{assm:blob-diameter}}, 
as $m\to\infty$,
\begin{equation}
 \frac{\sigma(\mathbf{p})}{B_m+1}  \tilde{\mathcal{G}}_m^{\sss \mathrm{bl}}(\mathbf{p}, a) \dto \mathcal{G}_{\sss (\infty)}(\bld{\beta},\gamma).
 \end{equation}
\end{theorem} 

\begin{remark}\label{rem:about-assumption} \normalfont Assumption~\ref{assm:blob-diameter} only assumes that the diameter of the blobs are negligible compared to the graph distances in $ \tilde{\mathcal{G}}_m^{\sss \mathrm{bl}}(\mathbf{p}, a)$. 
This, in a way, is a necessary condition to ensure that the inherent structure of the blobs does not affect the limit. 
Theorem~\ref{thm:univesalty} shows that only Assumption~\ref{assm:blob-diameter} is also sufficient and additional assumptions as in \cite[Assumption 3.3]{BBSX14} are not required to prove universality in the Gromov-weak topology.
\end{remark}
The rest of this section is devoted to the proof of Theorem~\ref{thm:univesalty}. 

\subsection{Completing the proof of the universality theorem in Theorem~\ref{thm:univesalty}}
To simplify notation, we write $\tilde{\mathcal{G}}_m$, $\tilde{\mathcal{G}}_m^{\sss \mathrm{bl}}$, respectively, instead of $\tilde{\mathcal{G}}_m(\mathbf{p}, a)$ and $\tilde{\mathcal{G}}_m^{\sss \mathrm{bl}}(\mathbf{p}, a)$.
\begin{lemma}[{\cite[Lemma 4.11]{BHS15}}]\label{lem:spls-tight-p} Recall the definition of $N_{\sss(m)}^\star$ from {\rm Algorithm~\ref{algo:construction-Pcon}}.
The sequence of random variables $(N_{\sss(m)}^\star)_{m\geq 1}$ is tight. 
\end{lemma}
Recall the definition of Gromov-weak topology from Section~\ref{sec:defn:GHP-weak}. 
Fix some $l\geq 1$ and take any bounded continuous function $\phi: \R^{\sss l^2}\mapsto\R$. 
We simply write $\Phi(X)$ for $\Phi((X,\dst,\mu))$. 
\paragraph*{{\bf Key step 1}}
Let us write the scaled metric spaces as $\tilde{\cG}_m^{\sss \mathrm{s}} = \sigmap \tilde{\mathcal{G}}_m$ and $\tilde{\cG}_m^{\sss \mathrm{bl},\mathrm{s}} = \frac{\sigmap}{B_m+1} \tilde{\mathcal{G}}_m^{\sss \mathrm{bl}}$.
Using Theorem~\ref{thm:BHS15} it is enough to show that
\begin{equation}\label{eq:uni-th:red1}
 \lim_{m\to\infty} \big|\E\big[\Phi\big(\tilde{\cG}_m^{\sss\mathrm{bl}, \mathrm{s}}\big)\big] - \E\big[\Phi(\tilde{\cG}_m^{\sss \mathrm{s}})\big] \big|=0.
\end{equation}
The above step, together with Theorem~\ref{thm:BHS15}, completes the proof of  Theorem~\ref{thm:univesalty}. 
\paragraph*{{\bf Key step 2}}For any $K\geq 1$, the difference 
$ \big|\E\big[\Phi(\tilde{\cG}_m^{\sss \mathrm{s}})\big]- \sum_{k=0}^{K}\E\big[\Phi(\tilde{\cG}_m^{\sss \mathrm{s}})\1\{N_{\sss(m)}^\star = k\}\big]\big|$ is  at most $\|\phi\|_{\infty}\prob{N_{\sss(m)}^\star \geq K+1},$ and the same inequality also holds for $\tilde{\cG}_m^{\sss\mathrm{bl}, \mathrm{s}}$.
Thus, using Lemma~\ref{lem:spls-tight-p}, 
the proof of \eqref{eq:uni-th:red1} reduces to showing that, for each fixed $k\geq 1$, 
\begin{equation}\label{eq:blobvsG-reduction1}
 \lim_{m\to\infty}\Big|\E\big[\Phi\big(\tilde{\cG}_m^{\sss \mathrm{bl},\mathrm{s}}\big)\1\{N_{\sss(m)}^\star = k\}\big]-\E\big[\Phi(\tilde{\cG}_m^{\sss \mathrm{s}})\1\{N_{\sss(m)}^\star = k\}\big]\Big|=0.
\end{equation}
\paragraph*{{\bf Main aim of this section.}} 
Below, we define a function $\mathsf{g}_{\phi}^k(\cdot)$ on the space $\overline{T}_{IJ}^*$ which captures the behavior of pairwise distances after creating $k$ surplus edges. 
Under Assumption~\ref{assm:blob-diameter}, we show that the introduction of blobs changes the distances within the tilted $\vp$-trees and the $\mathsf{g}_\phi^k$ values negligibly. This completes the proof of \eqref{eq:blobvsG-reduction1}. 

For any fixed $k\geq 0$, consider $\mathbf{t}\in T^*_{\sss I,(k+l)}$ with root $0+$, leaves $\bld{i}=(1+,\dots,(k+l)+)$ and root-to-leaf measures $\nu_{\vt,i}$ on the path $[0+,i+]$ for all $1\leq i\leq k+l$. 
We create a graph $G(\mathbf{t})$ by sampling, for each $1\leq i\leq k$, points $i_s$ on $[0+,i+]$ according $\nu_{\vt,i}$ and connecting $i+$ with $i_s$. 
Let $\dst_{\sss G(\mathbf{t})}$ denote the distance on $G(\mathbf{t})$ given by the sum of edge lengths in the shortest path. 
Then, the function $\mathsf{g}_{\phi}^k:T^*_{\sss I,(k+l)}\mapsto\R$ is defined as
\begin{subequations}
\begin{equation}\label{defn:g-phi-1}
 \mathsf{g}_{\phi}^k(\mathbf{t}) = \E\big[\phi\big(\dst_{\sss G(\mathbf{t})}(i+,j+):k+1\leq i,j\leq k+l\big)\big]\ind{\mathbf{t}\neq \partial},
\end{equation}where $\partial$ is a forbidden state defined as follows: 
Given any $\mathbf{t}\in T_{\sss IJ}^*$,  and a set of vertices $\bld{v}=(v_1,\dots,v_r)$, we denote by $\mathbf{t}(\bld{v})$, the subtree of $\mathbf{t}$ spanned by $\bld{v}$, i.e., $\mathbf{t}(\bld{v})$ is the subtree of $\mathbf{t}$ containing all vertices in $\bld{v}$ with minimal number of edges.
We declare $\mathbf{t}(\bld{v})=\partial$ if 
$[\rho,v_i] \subset [\rho,v_j]$ for some $i\neq j$, where $\rho$ is the root of $\mathbf{t}$.
Thus, if $\mathbf{t}(\bld{v})\neq \partial$, the tree $\mathbf{t}(\bld{v})$ necessarily has $r$ leaves.
Notice that the expectation in \eqref{defn:g-phi-1} is over the choices of $i_s$-values only.
In our context, $\vt$ is always considered as a subgraph of the graph on the vertex set $[m]$ and thus we assume that $\vt$ has inherited the labels from the corresponding graph.
Thus $\vt\in T^{*m}_{\sss I,(k+l)}$.
There is a natural way to extend $\mathsf{g}_{\phi}^k(\cdot)$ to $\overline{T}_{\sss I,(k+l)}^{*m}$ as follows:
Consider $\bar{\mathbf{t}}\in \overline{T}_{\sss I,(k+l)}^{*m}$ and the corresponding $\mathbf{t}\in T_{\sss I,(k+l)}^{*m}$ (see Section~\ref{sec:T-space-extended}). 
Let $0+$, $\bld{i}$, $(\nu_{\vt,i})_{i\in [k+l]}$  and $(i_s)_{i\in [k+l]}$ be as defined above. 
Let $\bar{G}(\bar{\mathbf{t}})$ denote the metric space obtained by introducing an edge of length one between $X_{i+i_s}$ and $X_{i_si+}$, where $X_{ij}$ has distribution  $\mu_i$ for all $j\geq 1$, independently of each other and other shortcuts. 
For $k+1\leq i\leq k+l$, $X_i\in M_{x_i}$ have distribution $\mu_{x_i}$ independently for all $i\geq 1$. 
Let $\bar{\dst}_{\sss \bar{G}(\mathbf{\bar{t}})}$ denote the distance on $\bar{G}(\mathbf{\bar{t}})$. Then, let
\begin{equation}\label{defn:g-phi-2}
 \mathsf{g}_{\phi}^k(\bar{\mathbf{t}}) = \E\big[\phi\big(\bar{\dst}_{\sss \bar{G}(\mathbf{\bar{t}})}(X_i,X_j):k+1\leq i,j\leq k+l\big)\big]\ind{\mathbf{t}\neq \partial},
\end{equation}
\end{subequations}where the expectation is taken over the collection of random variables $X_{i+i_s}$ and $X_{i_si+}$.  
At this moment, we urge the reader to recall the construction in Algorithm~\ref{algo:construction-Pcon}, Lemma~\ref{lem:lk-rk-equivalent} and all the associated notations.
Now, conditionally on $\mathscr{T}_m^{ \mathbf{p},\star}$, we can construct the tree $\mathscr{T}_m^{ \mathbf{p},\star}(\tilde{\mathbf{V}}_m^{k,k+l})$, where 
\begin{enumerate}[(a)]
 \item $\tilde{\mathbf{V}}_m^{k,k+l}= (\tilde{V}_1,\dots,\tilde{V}_k, V_{k+1},\dots , V_{k+l})$ is an independent collection of vertices from the vertex set of $\mathscr{T}_m^{ \mathbf{p},\star}$;
 \item $\tilde{V}_i$ is distributed as $\cJ^{\sss (m)}(\cdot)$, for $1\leq i\leq k$ and $V_i$ is distributed as $\mathbf{p}$, for $k+1 \leq i \leq k+l$.
\end{enumerate} Note that, by \cite[(4.25)]{BHS15}, $ \lim_{m\to\infty} \PR(\mathscr{T}_m^{ \mathbf{p},\star}(\tilde{\mathbf{V}}_m^{k,k+l})= \partial)=0.$
Whenever $\mathscr{T}_m^{ \mathbf{p},\star}(\tilde{\mathbf{V}}_m^{k,k+l})\neq \partial$, $\mathscr{T}_m^{ \mathbf{p},\star}(\tilde{\mathbf{V}}_m^{k,k+l})$ can be considered as an element of $T^{*m}_{\sss I,k+l}$  using the leaf-weights $(\mathfrak{G}_{\sss (m)}(\tilde{V}_i))_{i=1}^k$, $(\mathfrak{G}_{\sss (m)}(V_i))_{i=k+1}^{k+l}$ and root-to-leaf measures given by $(Q_{\tilde{V}_i}^m(\cdot))_{i=1}^k$, $(Q_{V_i}^m(\cdot))_{i=k+1}^{k+l}$.
 Let $\bar{\mathscr{T}}_m^{ \mathbf{p},\star}(\tilde{\mathbf{V}}_m^{k,k+l})$ denote the element corresponding to $\mathscr{T}_m^{ \mathbf{p},\star}(\tilde{\mathbf{V}}_m^{k,k+l})$ with blobs. 
 Thus, $\bar{\mathscr{T}}_m^{ \mathbf{p},\star}(\tilde{\mathbf{V}}_m^{k,k+l})$ is viewed as an element of $\overline{T}_{\sss I,(k+l)}^{*m}$. 
Let $\mathbf{V}_m=(V_1,\dots,V_{k+l})$ be an i.i.d.~collection of random variables with distribution $\mathbf{p}$. Let $\E_{\mathbf{p},\star}$ denote the expectation conditionally on $\mathscr{T}_m^{\mathbf{p},\star}$ and~$N_{\sss(m)}^\star$. The proof of \eqref{eq:blobvsG-reduction1} now reduces to
\begin{eq} \label{eq:blobvsG-reduction2}
&\Big|\E\big[\Phi\big(\tilde{\cG}_m^{\sss \mathrm{bl},\mathrm{s}}\big)\1\{N_{\sss(m)}^\star = k\}\big]-\E\big[\Phi(\tilde{\cG}_m^{\sss\mathrm{s}})\1\{N_{\sss(m)}^\star = k\}\big]\Big|\\
&\hspace{1cm}= \bigg| \E\bigg[\E_{\mathbf{p},\star}\Big[\mathsf{g}_{\phi}^{k}\Big(\frac{\sigma(\mathbf{p})}{B_m+1}\bar{\mathscr{T}}_m^{\mathbf{p},\star}(\tilde{\mathbf{V}}_m^{k,k+l}) \Big)\Big]\ind{N_{\sss (m)}^{\star}=k}\bigg]\\
&\hspace{2cm}- \E\Big[\E_{\mathbf{p},\star}\big[\mathsf{g}_{\phi}^{k}\big(\sigma(\mathbf{p})\mathscr{T}_m^{\mathbf{p},\star}(\tilde{\mathbf{V}}_m^{k,k+l})\big)\big]\ind{N_{\sss (m)}^{\star}=k}\Big]\bigg| +o(1).
\end{eq}
Notice that the tilting does not affect the blobs themselves but only the superstructure. Recall also the definition of the tilting function $L(\cdot)$ from \eqref{eqn:ltpi-def}. 
Using the fact that $\cJ^{\sss (m)}(v) \propto p_v \mathfrak{G}_{\sss (m)}(v)$, 
\begin{equation}
\begin{split}
\E_{\mathbf{p},\star}\big[\mathsf{g}_{\phi}^{k}\big(\sigma(\mathbf{p})\mathscr{T}_m^{\mathbf{p},\star}(\tilde{\mathbf{V}}_m^{k,k+l})\big) \big]= \frac{\E_{\mathbf{p},\star}\big[\prod_{i=1}^k\mathfrak{G}_{\sss (m)}(V_i)\mathsf{g}_{\phi}^{k}\big(\sigma(\mathbf{p})\mathscr{T}_m^{\mathbf{p},\star}(\mathbf{V}_m)\big) \big]}{\big(\E_{\mathbf{p},\star}[\mathfrak{G}_{\sss (m)}(V_1)]\big)^k}.
\end{split}
\end{equation} and an identical expression holds by replacing $\sigma(\mathbf{p})\mathscr{T}_m^{\mathbf{p},\star}$ by $\frac{\sigma(\mathbf{p})}{B_m+1}\bar{\mathscr{T}}_m^{\mathbf{p}}$.
Denote the expectation conditionally on $\mathscr{T}_m^{\mathbf{p}}$ and $N_{\sss (m)}$ by $\E_{\mathbf{p}}$ and simply write $\bar{\mathscr{T}}_m^{\mathbf{p},{\sss \mathrm{s}}}$, $\mathscr{T}_m^{\mathbf{p},{\sss \mathrm{s}}}$ for $\frac{\sigma(\mathbf{p})}{B_m+1}\bar{\mathscr{T}}_m^{\mathbf{p}}(\mathbf{V}_m)$,  $\sigma(\mathbf{p})\mathscr{T}_m^{\mathbf{p}}(\mathbf{V}_m)$ respectively. Now, \eqref{eq:blobvsG-reduction2} simplifies to 
\begin{eq}\label{eq:blobvsG-reduction3}
 &\Big|\E\big[\Phi\big(\tilde{\cG}_m^{\sss \mathrm{bl},\mathrm{s}}\big)\1\{N_{\sss(m)}^\star = k\}\big]-\E\big[\Phi(\tilde{\cG}_m^{\sss\mathrm{s}})\1\{N_{\sss(m)}^\star = k\}\big]\Big|\\
 &\hspace{.6cm}\leq \frac{1}{\expt{L(\mathscr{T}_m^{\mathbf{p}})}}\bigg|\E\bigg[\frac{\E_{\mathbf{p}}\big[\prod_{i=1}^k\mathfrak{G}_{\sss (m)}(V_i)\mathsf{g}_{\phi}^{k}\big(\bar{\mathscr{T}}_m^{\mathbf{p},{\sss \mathrm{s}}} \big)\big]}{\big(\E_{\mathbf{p}}[\mathfrak{G}_{\sss (m)}(V_1)]\big)^k} L(\mathscr{T}_m^{\mathbf{p}})\ind{N_{\sss (m)}=k}\bigg] \\
 & \hspace{1cm} - \E\bigg[\frac{\E_{\mathbf{p}}\big[\prod_{i=1}^k\mathfrak{G}_{\sss (m)}(V_i)\mathsf{g}_{\phi}^{k}\big(\mathscr{T}_m^{\mathbf{p},{\sss \mathrm{s}}}\big) \big]}{\big(\E_{\mathbf{p}}[\mathfrak{G}_{\sss (m)}(V_1)]\big)^k}L(\mathscr{T}_m^{\mathbf{p}})\ind{N_{\sss (m)}=k}\bigg] \bigg|.
\end{eq}
\begin{proposition}\label{prop:g-blob-close} As $m\to\infty$,
 $\big|\mathsf{g}_{\phi}^{k}\big(\bar{\mathscr{T}}_m^{\mathbf{p},{\sss \mathrm{s}}}\big)-\mathsf{g}_{\phi}^{k}\big(\mathscr{T}_m^{\mathbf{p},{\sss \mathrm{s}}} \big) \big|\pto 0.$
\end{proposition} 
We first show that it is enough to prove Proposition~\ref{prop:g-blob-close} to complete the proof of~\eqref{eq:blobvsG-reduction3}, but before that we first need to state some results. 
The proofs of Facts~\ref{fact:unif-int} and \ref{fact:2} below  are elementary and we omit the proof here.
The proof of Proposition~\ref{prop:g-blob-close} is deferred to Section~\ref{sec:proof-prop-blob-G}.
\begin{lemma}[{\cite[Proposition 4.8, Theorem 4.15]{BHS15}}] \label{lem:uni-int-tilt} $(L(\mathscr{T}_m^{\mathbf{p}}))_{m\geq 1}$ is uniformly integrable. Also, for each $k\geq 0$, the quantity
	\begin{align}\label{eqn:jt-convg}
&\bigg(\Ep\bigg[\frac{\dA_{\sss(m)}(V_1)}{\sigma(\vp)}\bigg],  ~\Ep\bigg[\bigg(\prod_{i=1}^k \frac{\dA_{\sss(m)}(V_i)}{\sigma(\vp)}\bigg) \mathsf{g}_\phi^{k}\big(\mathscr{T}_m^{\mathbf{p},{\sss \mathrm{s}}}\big)  \bigg]\bigg)
	\end{align} converges in distribution to some random variable.
\end{lemma}
\begin{fact}\label{fact:unif-int} Consider three sequences of random variables  $(X_{m})_{m\geq 1}$, $(Y_m)_{m\geq 1}$ and $(Y_m')_{m\geq 1}$ such that (i) $(X_m)_{m\geq 1}$ is uniformly integrable, (ii) $(Y_m)_{m\geq 1}$ and $(Y_m')_{m\geq 1}$ are almost surely bounded and (iii) $Y_m-Y_m'\xrightarrow{\sss \PR} 0$. Then, as $m\to\infty$, $\expt{|X_mY_m-X_mY_m'|}\to 0$.
\end{fact} 

\begin{fact}\label{fact:2} Suppose that $(X_m)_{m\geq 1}$ is a sequence of random variables such that for every $m\geq 1$, there exists a further sequence $(X_{m,r})_{r\geq 1}$ satisfying (i) for each fixed $r\geq 1$, $X_{m,r}\xrightarrow{\sss \PR} 0$ as $m\to\infty$, and (ii) $\lim_{r\to\infty}\limsup_{m\to\infty}\PR(|X_m-X_{m,r}|>\varepsilon)=0$ for any $\varepsilon>0$. Then $X_m \xrightarrow{\sss \PR} 0$ as $m\to\infty$.
\end{fact}

\begin{proof}[Proof of \eqref{eq:blobvsG-reduction3} from Proposition~\ref{prop:g-blob-close}]
We apply Fact~\ref{fact:unif-int} with $X_m =L(\mathscr{T}_m^{\mathbf{p}}) \ind{N_{\sss (m)} = k}$, which is uniformly integrable by  Lemma~\ref{lem:uni-int-tilt}.
Thus it is enough to show that 
\begin{eq}
\bigg| \frac{\E_{\mathbf{p}}\big[\prod_{i=1}^k\mathfrak{G}_{\sss (m)}(V_i)\mathsf{g}_{\phi}^{k}\big(\bar{\mathscr{T}}_m^{\mathbf{p},{\sss \mathrm{s}}} \big)\big]}{\big(\E_{\mathbf{p}}[\mathfrak{G}_{\sss (m)}(V_1)]\big)^k}-\frac{\E_{\mathbf{p}}\big[\prod_{i=1}^k\mathfrak{G}_{\sss (m)}(V_i)\mathsf{g}_{\phi}^{k}\big(\mathscr{T}_m^{\mathbf{p},{\sss \mathrm{s}}}\big) \big]}{\big(\E_{\mathbf{p}}[\mathfrak{G}_{\sss (m)}(V_1)]\big)^k}\bigg| \pto 0.
\end{eq}
Applying Lemma~\ref{lem:uni-int-tilt} again, the above reduces to showing
\begin{equation}\label{eq:blobvsG-reduction4}
\Ep\bigg[\bigg(\prod_{i=1}^k \frac{\dA_{\sss(m)}(V_i)}{\sigma(\vp)}\bigg) \big(\mathsf{g}_{\phi}^{k}\big(\bar{\mathscr{T}}_m^{\mathbf{p},{\sss \mathrm{s}}}\big)-\mathsf{g}_{\phi}^{k}\big(\mathscr{T}_m^{\mathbf{p},{\sss \mathrm{s}}} \big) \big)  \bigg] \pto 0.
\end{equation}
We now apply Fact~\ref{fact:2}. Let $Y_m$ denote the term inside the expectation in \eqref{eq:blobvsG-reduction4}.
Further, sample the set of leaves $\mathbf{V}_m$ independently $r$ times on the same tree $\mathscr{T}_m^\vp$ and let $Y_{m}^i$ denote the observed value of $Y_m$ in the $i$-th sample. 
Now, let $X_m = \E_{\mathbf{p}} [Y_m]$ and $X_{m,r} = r^{-1} \sum_{i=1}^r Y_m^i $. 
First, to verify condition (ii), note that $\Ep(X_{m,r}) = X_m$ and therefore Chebyshev's inequality yields
\begin{eq}\label{chebyshev-X-m}
\prob{|X_m-X_{m,r}|>\varepsilon}&\leq \frac{\E[X_m^2]}{\varepsilon^2r} \leq \frac{4\|\phi\|_{\infty}^2}{\varepsilon^2r}\frac{\expt{\dA_{\sss (m)}(V_1)^{2k}}}{\sigmap^{2k}}.
\end{eq}
By \eqref{eqn:ant-def-new}, $\|\dA_{\sss (m)}\|_{\infty}\leq \|A_{\sss (m)}\|_{\infty}$, and thus an application of  \cite[Lemma 4.9, (4.12)]{BHS15} yields that for any $x\geq \e$ and $m\geq 1$,
\begin{eq}\label{g-m-infty-bound}
\PR(\|\dA_{\sss (m)}\|_{\infty}\geq x \sigmap) \leq \e^{-Cx \log(\log x)},
\end{eq}
 where $C>0$ is a constant.
Combining \eqref{chebyshev-X-m} and \eqref{g-m-infty-bound}, the condition (ii) is verified. 
Next, condition (i) in Fact~\ref{fact:2} is satisfied by Proposition~\ref{prop:g-blob-close} and \eqref{g-m-infty-bound}. 
An application of Fact~\ref{fact:2} 
concludes the proof of \eqref{eq:blobvsG-reduction4}, and hence the proof of \eqref{eq:blobvsG-reduction3} follows. 
\end{proof}

\subsection{Comparing distances with and without blobs: Proof of Proposition~\ref{prop:g-blob-close}}\label{sec:proof-prop-blob-G}
In this section, we will use the notion of Gromov-Hausdorff-Prokhorov topology on the collection of measured metric spaces $(X,\dst,\mu)$, where $(X,\dst)$ is a compact metric space and $\mu$ is a probability measure on corresponding Borel sigma algebra. 
Without re-defining all the required notions, we refer the reader to \cite[Section~2.1.1]{BHS15}.
Let $\mathrm{d}_{\sss \mathrm{GHP}}$ denote the distances in this topology.
 We further recall the notation $\mathrm{dis}$ for distortion and $D(\mu;\mu_1,\mu_2)$ for discrepancy of measures as defined in \cite[Section~2.1.1]{BHS15}.
Denote the root of $\mathscr{T}_m^{\vp}(\mathbf{V}_m)$ by $0+$ and the $j$th leaf by $j+$.
Let $\mathscr{M}_j^m:= \big([0+,j+],\dst,\nu_j\big)$ be the (random) measured metric space, where $\nu_j$ is any probability measure on the Borel sigma-algebra of $\{[0+,j+],\dst\}$.
 In particular, we can take $\nu_j$'s to be the corresponding  root-to-leaf measures. 
 Let $\vmbar{12}{\mathscr{M}}_j^m:= \{ \bar{M}_j, \bar{\dst}, \bar{\nu}_j \}$ be the measured metric space   with $\bar{M}_j:=\sqcup_{i\in [0+,j+]}M_i$ and the induced root-to-leaf measure $\bar{\nu}_j(A)=\sum_{i\in [0+,j+]}\nu_j(\{i\})\mu_i(A\cap M_i)$. 
For convenience, we have suppressed the dependence on $\mathscr{T}_m^{\vp}(\mathbf{V}_m)$ in the notation. Note that 
 $\mathscr{M}_j^m$ is coupled to $\vmbar{12}{\mathscr{M}}_j^m$ in the obvious way that the superstructure of $\vmbar{12}{\mathscr{M}}_j^m$ is given by $\mathscr{M}_j^m$.
 We need the following lemma to prove Proposition~\ref{prop:g-blob-close}:
\begin{lemma}\label{lem:path:GHP-conv}For $j\geq 1$, as $m\to\infty$, $\dGHP \big(\sigma(\mathbf{p})\mathscr{M}_j^m, \frac{\sigma(\mathbf{p})}{B_m+1}\vmbar{12}{\mathscr{M}}_j^m \big) \pto 0.$
\end{lemma}

\begin{proof}
 We prove this for $j=1$ only. The proof for $j\geq 2$ is identical. 
 For $x\in \bar{M}_1$, we denote its corresponding vertex label by $i(x)$, i.e., $i(x)=k$ if and only if $x\in M_k$. Consider the correspondence $C_m$ and the measure $\mathfrak{m}$ on the product space $[0+,1+]\times \bar{M}_1$ defined as
 \begin{equation}\label{corr-blob-vs-path}
 C_m:=\{(i,x): i\in [0+,1+], x\in M_i\}, \quad \mathfrak{m}(\{i\}\times A)=\nu_1(\{i\})\mu_i(A\cap M_i).
 \end{equation} 
%
  Note that the discrepancy of $\mathfrak{m}$ satisfies $D(\mathfrak{m}; \nu_1,\bar{\nu}_1)=0$, since the marginals are exactly equal to $\nu_1$ and $\bar{\nu}_1$. Further, $\mathfrak{m}(C_m^c)=0.$ Therefore, Lemma~\ref{lem:path:GHP-conv} follows if we can prove that
 \begin{equation}\label{distortion:expression}
  \dis(C_m):= \sup_{x,y\in \bar{M}_1} \Big\{ \sigma(\mathbf{p})\dst(i(x),i(y))-\frac{\sigma(\mathbf{p})}{B_m+1}\bar{\dst}(x,y) \Big\} \pto 0.
 \end{equation}
To simplify the expression for $\dis(C_m)$, suppose that $i(x)$ is an ancestor of $i(y)$ on the path from $0+$ to $1+$. Then, 
\begin{gather*}
 \dst(i(x),i(y))=\dst(0+,i(y))-\dst(0+,i(x)),\\
 \bar{\dst}(x_0,y)-\bar{\dst}(x_0,x) \leq \bar{\dst}(x,y)\leq \bar{\dst}(x_0,y)-\bar{\dst}(x_0,x)+2\Delta_{\max},
\end{gather*}for any $x_0\in M_{\sss 0+}$. 
This implies that
\begin{equation}\label{eq:distortion-reduction}
\begin{split}
 &\sup_{x,y\in \bar{M}_1} \Big\{ \sigma(\mathbf{p})\dst(i(x),i(y))-\frac{\sigma(\mathbf{p})}{B_m+1}\bar{\dst}(x,y) \Big\}\\
 &\hspace{.8cm}\leq 2 \sup_{y \in \bar{M}_1}\Big\{ \sigma(\mathbf{p})\dst(0+,i(y))-\frac{\sigma(\mathbf{p})}{B_m+1}\bar{\dst}(x_0,y) \Big\}+\frac{2 \sigmap \Delta_{\max}}{B_m+1}.
 \end{split}
\end{equation}
Further, replacing $y$ by any other point $y'$ in the right hand side in \eqref{eq:distortion-reduction} incurs an error of at most $\sigmap \Delta_{\max}/(B_m+1)$. Now, write the path $[0+,1+]$ as $0+=i_0\to i_1\to \dots \to i_{R^*-2} \to i_{R^*-1}=1+.$ Then
\begin{equation}\label{eqn:estimate-dis}
 \dis(C_m)\leq 2\sup_{k\leq R^*-1} \Big| \sigmap \dst(i_0,i_k)-\frac{\sigmap}{B_m+1}\bar{\dst}(X_{\sss i_0,i_1}, X_{\sss i_k,i_{k+1}}) \Big|+\frac{6\sigmap \Delta_{\max}}{B_m+1},
\end{equation}where $(X_{i,j})_{i,j\in [m]}$ are the junction-points.
Using Assumption~\ref{assm:blob-diameter} and \eqref{eqn:estimate-dis}, it is now enough to show that for any $\varepsilon >0$, 
\begin{equation}\label{prob:dst:paths-mg}
 \lim_{m\to\infty} \PR\bigg(\sup_{k\leq R^*-1}\Big| \sigmap \dst(i_0,i_k)-\frac{\sigmap}{B_m+1}\bar{\dst}(X_{\sss i_0,i_1}, X_{\sss i_k,i_{k+1}}) \Big| >\varepsilon\bigg) = 0.
\end{equation}
Denote the term inside $\sup$ above by $Q_k$. Then,
\begin{equation}\label{Qk-simplif}
\begin{split}
 Q_k&:=\bigg[ \sigmap \dst(i_0,i_k)-\frac{\sigmap}{B_m+1}\bar{\dst}\big(X_{\sss i_0,i_1}, X_{\sss i_k,i_{k+1}}\big) \bigg]\\
 &=\bigg[ \sigmap k - \frac{ \sigmap}{B_m+1}\bigg(k+\sum_{j=1}^k\dst_{i_j}\big(X_{\sss i_j,i_{j-1}}, X_{\sss i_j, i_{j+1}}\big)\bigg)\bigg]\\
 &=\frac{\sigmap}{B_m+1} \bigg[\sum_{j=1}^k\Big( B_m-\dst_{i_j}\big(X_{\sss i_j,i_{j-1}}, X_{\sss i_j, i_{j+1}}\big)  \Big) \bigg].
 \end{split}
\end{equation}
 Recall the construction of the path $[0+,1+]$ via the birthday problem from Section~\ref{sec:p-tree-birthday}. Take $\mathbf{J}:=(J_i)_{i\geq 1}$ such that $J_i$ are an i.i.d.~sample from $\mathbf{p}$. 
 Further let $\boldsymbol{\xi}:= (\xi_i)_{i\in [m]}$ be an independent sequence such that $\xi_i$ is the distance between two points, chosen randomly from $M_i$ according to $\mu_i$. 
 Further, let $\mathbf{J}$ and $\boldsymbol{\xi}$ be independent. Then $R^*$ can be thought of as the first repeat time of the sequence $\mathbf{J}$. Thus, $(Q_k)_{k=1}^{R^*-1}$ in \eqref{Qk-simplif} has the same distribution as $(\hat{Q}_k)_{k=1}^{R^*-1}$, where 
 \begin{equation}
  \begin{split}
    \hat{Q}_k:= \frac{\sigmap}{B_m+1} \sum_{i=1}^k\big(   B_m -\xi_{J_i} \big).
  \end{split}
 \end{equation}
From the birthday construction $\expt{\xi_{J_1}}= \sum_{i\in [m]}p_iu_i=B_m$ and $(\xi_{J_i})_{i \geq 1}$ is an independent sequence. Therefore, $(\hat{Q}_k)_{k\geq 0}$ is a martingale. Further, 
\begin{equation}
\mathrm{Var}(\hat{Q}_k)\leq \bigg(\frac{\sigmap}{B_m+1}\bigg)^2k\Delta_{\max}\sum_{i\in [m]}p_iu_i=\frac{\sigmap^2k\Delta_{\max}B_m}{(B_m+1)^2}.
\end{equation} Thus, by Doob's inequality (for example applying \cite[Chapter II, Lemma 54.5]{RW94} to $\hat{Q}_k^2$), it follows that, for any $\varepsilon>0$ and $T>0$,
\begin{equation}\label{eqn:supbound:Q}
 \PR\bigg(\sup_{k\leq T}|\hat{Q}_k|>\varepsilon \bigg)\leq \frac{T\sigmap^2\Delta_{\max}B_m}{(B_m+1)^2\varepsilon^2}.
\end{equation}Recall from \cite[Theorem 4]{CP99} that $(\sigmap R^*)_{m\geq 1}$ is a tight sequence of random variables. 
The proof now follows using Assumption~\ref{assm:blob-diameter}.
%
%
%
%
%
\end{proof}
\begin{proof}[Proof of Proposition~\ref{prop:g-blob-close} using Lemma~\ref{lem:path:GHP-conv}] 
We use the objects defined in \eqref{corr-blob-vs-path}, \eqref{distortion:expression} in the proof of Lemma~\ref{lem:path:GHP-conv} for all the path metric spaces with $j\leq k$. We assume that we are working on a probability space such that the convergence \eqref{distortion:expression} holds almost surely for all $j\leq k$. 
To summarize, for fixed $\varepsilon>0$ and for each $j\leq k$, we can choose a correspondence $C_m^j$ and a measure $\mathfrak{m}_j$ of $[0+,j+]\times \bar{M}_j$ satisfying (i) $(i,X_{ik})\in C_m^j$, for all $i,k\in [0+,j+]$, (ii) $\dis(C_m^j)<\varepsilon/2k$ almost surely, and (iii) $D(\mathfrak{m}_j;\nu_j,\bar{\nu}_j)=0$ and $\mathfrak{m}_j((C^j_m)^c)=0$. 
Recall the definitions of the function $\mathsf{g}_{\phi}^k$ from \eqref{defn:g-phi-1}, \eqref{defn:g-phi-2} and the associated graphs $G(\cdot)$, $\bar{G}(\cdot)$. We simply write $G$ and $\bar{G}$ for $G(\sigma(\mathbf{p})\mathscr{T}_m^{\mathbf{p}}(\mathbf{V}_m))$ and $\bar{G}(\frac{\sigma(\mathbf{p})}{B_m+1}\bar{\mathscr{T}}_m^{\mathbf{p}}(\mathbf{V}_m))$, respectively. 
Let $\mathfrak{m}^{\sss \otimes k}$ denote the $k$-fold product measure of $\mathfrak{m}_j$ for $j\leq k$. 
We denote the graph distance on a graph $H$ by $\dst_{\sss H}$.
Note that 
\begin{equation}\label{g-phi-compare}
\begin{split}
&\Big|\mathsf{g}_{\phi}^{k}\big(\sigma(\mathbf{p})\mathscr{T}_m^{\mathbf{p}}(\mathbf{V}_m)\big)-\mathsf{g}_{\phi}^{k}\Big(\frac{\sigma(\mathbf{p})}{B_m+1}\bar{\mathscr{T}}_m^{\mathbf{p}}(\mathbf{V}_m) \Big) \Big|\\
&\leq \E\big[\big|\phi\big((\dst_{\sss G}(i+,j+))_{i,j=k+1}^{k+l}\big)- \phi\big((\dst_{\sss \bar{G}}(X_i,X_j))_{i,j=k+1}^{k+l}\big)\big|\big],
\end{split}
\end{equation} where $X_i\sim \mu_i$ independently for $i\in [m]$, and the above expectation is with respect to the measure $\mathfrak{m}^{\sss \otimes k}$.
Recall the notation while defining $\mathsf{g}_\phi^k(\cdot)$ in \eqref{defn:g-phi-1}, \eqref{defn:g-phi-2}. Notice that for any point $k\in [0+,i+]$ and $x_k\in M_k$ and $x_{i_s}\in M_{i_s}$,
\begin{equation}\label{shortcut-contribute}
 | \dst_{\sss \mathbf{t}}(k,i_s)- \dst_{\sss\bar{\mathbf{t}}}(x_k,x_{i_s})|\leq \frac{\varepsilon}{2k}.
\end{equation}
Now, for any path $i+$ to $j+$ in $G$, we can essentially take the same path from $X_i$ to $X_j$ in $\bar{G}$ and take the corresponding inter-blob paths on the way. 
The distance traversed in $\bar{G}$ in this way gives an upper bound on $\dst_{\sss \bar{G}}(X_i,X_j)$. Notice that, by \eqref{shortcut-contribute}, taking a shortcut contributes at most $\varepsilon / 2k$ to the difference of the distance traveled in $G$ and $\bar{G}$. Also, traversing a shortcut edge contributes $\sigmap B_m/(B_m+1)$ and there are at most $k$ shortcuts on the path. 
Furthermore, it may be required to reach the relevant junction points from $X_i$ and $X_j$ and that contributes at most $2 \sigmap \Delta_{\max}/(B_m+1)$. Thus, for $k+1\leq i,j\leq k+l$, and sufficiently large $m$,
\begin{equation}
 \dst_{\sss \bar{G}}(X_i,X_j)\leq \dst_{\sss G}(i+,j+)+\frac{\varepsilon}{2}+\frac{k\sigmap B_m}{B_m+1}+\frac{2\sigmap \Delta_{\max}}{B_m+1}\leq \dst_{\sss G}(i+,j+)+\varepsilon. 
\end{equation}
By symmetry we can conclude the lower bound also, and the continuity of $\phi(\cdot)$ (see \cite[Theorem 4.18]{BHS15}) along with \eqref{g-phi-compare} completes the proof of Proposition~\ref{prop:g-blob-close}.
\end{proof}

\section{Mesoscopic properties: Proofs of Theorems~\ref{thm:susceptibility}~and~\ref{thm:diam-max}} 
\label{sec:entrance-bdd-proofs}
At this moment, we urge the reader to recall the definitions from \eqref{defn:barely-subcrit}, \eqref{defn:pCli}, \eqref{defn:susc-w}, and \eqref{defn:susc-dist}. 
The configuration model graphs considered in this section will be assumed to have degree sequence $\bld{d}'$ and the vertices have an associated weight sequence $\bld{w}$ such that  Assumption~\ref{assumption-w} is satisfied. 
We use the notation $C, C'$ to denote generic positive constants, whose values can be different in different lines. 
The rest of the section is organized as follows. 
In Section~\ref{sec:diameter}, we start by proving the required bound on the diameter in Theorem~\ref{thm:diam-max}.
In order to deal with  the different terms in Theorem~\ref{thm:susceptibility}, we first obtain some moment estimates in Section~\ref{sec:moment-esimates-V}, and these estimates are then used to prove asymptotics of $s_2^\star$ in Section~\ref{sec:s2}.
The individual component weights are estimated in Section~\ref{sec:barely-subcritical-mass}. 
In Section~\ref{sec:s3}, we prove asymptotics of $s_3^\star$, and finally the mesoscopic typical distance is computed in Section~\ref{sec:meso-dist}.



\subsection{Maximum diameter: Proof of Theorem~\ref{thm:diam-max}} \label{sec:diameter}
Let $\ell_n' = \sum_{i\in [n]} d_i'$.
We will use path-counting estimates for the configuration model from \cite[Lemma 5.1]{J09b}. 
Let $P_l$ denote the number of paths of length $l$ in $\mathrm{CM}_n(\bld{d}')$. 
Then \cite[Lemma 5.1]{J09b} shows that for any $l\geq 1$
\begin{eq}\label{eq:janson-estimate}
\E[P_l] \leq \ell_n' (\nu_n')^{l-1}.
\end{eq}
If the maximum diameter is at least $n^{\delta} (\log n)^2$, then there exists a path of length $n^{\delta} (\log n)^2$, and therefore 
\begin{eq}\label{eq:quasi-poly-bound}
\PR(\Delta_{\max} > n^{\delta} (\log n)^2)\leq \sum_{l\geq n^{\delta} (\log n)^2}\expt{P_l}\leq \frac{\ell_n' (\nu_n')^{n^{\delta}(\log(n))^2}}{1-\nu_n'}\leq C n^{1+\delta}\e^{-C'(\log n)^2},
\end{eq}where the second step follows using \eqref{eq:janson-estimate}. Thus the proof of Theorem~\ref{thm:diam-max} follows. \qed

\subsection{Moment bounds for total weights}\label{sec:moment-esimates-V}
Consider the size-biased distribution on the vertex set $[n]$ with sizes $(w_i)_{i\in [n]}$.
 Let $V_n$ and $V_n^{*}$, respectively, denote a vertex chosen  uniformly at random and according to the size-biased distribution with respect to the sizes $(w_i)_{i\in [n]}$, independently of the underlying graph $\rCM_n(\bld{d}')$. 
Let $D_n'$, $W_n$ (respectively $D_n^*$, $W_n^*$) denote the degree and weight of $V_n$ (respectively $V_n^*$). 
For a vertex $v\in [n]$, let $\mathscr{W}(v):=\sum_{k\in\mathscr{C}'(v)}w_k$, where $\mathscr{C}'(v)$ denotes the component of $\rCM_n(\bld{d}')$ containing $v$.
The reader should note the difference in notation that terms such as $\Wli$, $\pCli$ with $i$ in the subscript refer to the quantities defined in \eqref{defn:pCli}.


In this section, we prove the following moment bounds for $\mathscr{W}(V_n^*)$, which will help us compute the expectation and variance of $s_2^\star$: 
\begin{lemma}\label{lem:expt-weight-random} Under {\rm Assumption~\ref{assumption-w}}, the following holds: 
\begin{enumerate}[(i)]
    \item $
 \expt{\mathscr{W}(V_n^*)}= \frac{\expt{D_n^*}\expt{D_n' W_n}}{\expt{D_n'}(1-\nu_n')}(1+o(1))$,
 \item $\E\big[\big(\mathscr{W}(V_n^*)\big)^2\big] \leq   \frac{\expt{D_n^*}(\expt{D_n'W_n})^2\sigma_3(n)}{(\expt{D_n'})^2(1-\nu_n')^3}(1+o(1))$, where $\sigma_3(n) = \frac{1}{\ell_n'} \sum_{k\in [n]} d_k'(d_k'-1)(d_k'-2)$,
 \item $\E\big[\big(\mathscr{W}(V_n^*)\big)^3\big]= o(n^{1+2\delta})$.
\end{enumerate}
\end{lemma}

\begin{proof}[Proof of Lemma~\ref{lem:expt-weight-random}~(i)]  
We use path-counting techniques for configuration models from \cite[Lemma 5.1]{J09b}. 
Let $\fI_l(v,k)$ denote the collection of $\xx = (x_i)_{0\leq i\leq l}$ such that $x_0 = v$, $x_l = k$ and $x_i$'s are distinct.
Then, an identical argument to the proof of \cite[Lemma 5.1]{J09b} shows that, for any $l\geq 1$, the expected number of paths of length exactly $l$ starting from vertex $v$ and ending at $k$ is given by
\begin{equation}\label{eq:ineq-path}
 \sum_{\xx \in \fI_l(v,k)}\frac{d'_{x_0}d'_{x_l} \prod_{i=1}^{l-1} d_{x_i}'(d_{x_i}'-1) }{(\ell_n'-1)\cdots (\ell_n'-2l+1)}\leq \frac{d'_v \ell'_n}{\ell'_n-2l+3}\nu_n'^{l-1} = \Big(1+O\Big(\frac{l}{n}\Big)\Big)d'_v\nu_n'^{l-1},
\end{equation}where the last step holds for $l = o(n)$.
  Let $\mathcal{A}_l(v,k)$ denote the event that there exists a path of length $l$ from $v$ to $k$ and let $\mathcal{A}'_l(v,k)$ denote the event that there exist two different paths from $v$ to $k$, one of length $l$ and another one of length at most $l-1$. Notice that 
\begin{subequations}
\begin{equation}\label{sus1-expr-1}
 \begin{split}
  \expt{\mathscr{W}(V_n^*)\vert V_n^* = v} &= \E\bigg[\sum_{k\in [n]}w_k\ind{v\leadsto k}\bigg]\leq w_v +  \sum_{l\geq 1}\sum_{k\in [n]}w_{k}\prob{\mathcal{A}_l(v,k)},
 \end{split}
 \end{equation}
\begin{equation}\label{sus1-expr-2}
 \expt{\mathscr{W}(V_n^*)\vert V_n^* = v} \geq \sum_{l\geq 1}\sum_{k\in [n]}w_{k}\prob{\mathcal{A}_l(v,k)}-\sum_{l\geq 1}\sum_{k\in [n]}w_{k}\prob{\mathcal{A}'_l(v,k)}.
\end{equation}
\end{subequations}
Now, using \eqref{eq:ineq-path} and  Assumption~\ref{assumption-w}, \eqref{sus1-expr-1} yields
\begin{eq}\label{sus-simpl-1}
  &\expt{\mathscr{W}(V_n^*)\vert V_n^* = v} \\
  &\leq w_v+ \sum_{l = 1}^{n^{\delta} (\log n)^2}\sum_{k\in [n]}w_k\sum_{\xx \in \fI_l(v,k)}\frac{d'_{x_{0}}d_{x_k}' \prod_{i=1}^{l-1}d_{x_i}'(d_{x_i}'-1)}{(\ell_n'-1)\cdots (\ell_n'-2l+1)}+C\ell_n^w n^{1+\delta} \e^{-C'(\log n)^2} \\
  &\leq w_v+(1+o(1))\frac{d'_v\expt{D_n' W_n}}{\expt{D_n'}}\sum_{l=1}^{\infty}\nu_n'^{l-1} +o(1),
 \end{eq}
 where in the second step we have used \eqref{eq:quasi-poly-bound} and \eqref{eq:ineq-path}, and in the last step, we have used the facts that  $\ell_n' = n \E[D_n']$ and $\sum_{k\in [n]} d_k' w_k = n\E[D_n'W_n]$.
 Thus, 
 \begin{eq} \label{sus-simpl-12}
 \expt{\mathscr{W}(V_n^*)} \leq \frac{\expt{D_n^*}\expt{D_n' W_n}}{\expt{D_n'}(1-\nu_n')}(1+o(1)),
 \end{eq}
 where the multiplicative $(1+o(1))$ term in the final expression comes observing that $(1-\nu_n')^{-1} = \Theta(n^{\delta})$, and $\lim_{n\to\infty}\E[W_n^*]<\infty$. 
 For the computation of the lower bound in \eqref{sus1-expr-2}, we note that 
\begin{eq}
\PR(\cA_{l}(v,k)) &\geq \PR(\exists \text{ a unique path of length } l  \text{ from }v \text{ to }k ) \\
&= \sum_{\xx \in \fI_l(v,k)} \PR( \xx \text{ is the unique path between }v \text{ and } k)\\
&\geq \sum_{\xx \in \fI_l(v,k)}\PR( \xx \text{ is a path from }v \text{ and } k) \\
& - \sum_{\xx \in \fI_l(v,k)}\PR( \exists \yy\in\fI_l(v,k)\setminus\{\xx\}:  \xx, \yy \text{ both create paths from }v \text{ and } k),  
\end{eq} 
where the final step follows from using the inclusion-exclusion principle.
Also, the first term inside the sum in \eqref{eq:ineq-path} is the probability that $\xx = (x_0,\dots,x_l)$ creates a path for some $\xx$. 
Thus,
 \begin{eq} \label{eq:path-inclusion-excusion}
 \PR(\cA_l(v,k)) &\geq \sum_{\xx \in \fI_l(v,k)}\frac{d'_{x_{0}}d_{x_k}' \prod_{i=1}^{l-1}d_{x_i}'(d_{x_i}'-1)}{(\ell_n'-1)\cdots (\ell_n'-2l+1)} \\
 &- \sum_{\xx\in \fI_l(v,k)}\PR( \exists \yy\in\fI_l(v,k)\setminus\{\xx\} :\xx, \yy \text{ both create paths from }v \text{ and } k). 
 \end{eq}
If we have two distinct paths, one being $\xx$, and another of length at most $l$ (say $\yy$), then there must be two distinct vertices $a,b$ in $\xx$ such that the path between $a,b$ on $\xx$ is disjoint from that of $\yy$. 
 Depending on the choices of $a,b$, one of the structures in Figure~\ref{fig:multi-paths} occurs.
\begin{figure}
\centering
\begin{subfigure}{.2\textwidth}
\centering
	\includegraphics[width=.5\textwidth]{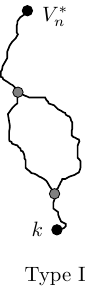}
\end{subfigure}
\begin{subfigure}{.2\textwidth}
\centering
	\includegraphics[width=.45\textwidth]{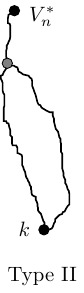}
\end{subfigure}
\begin{subfigure}{.2\textwidth}
\centering
	\includegraphics[width=.5\textwidth]{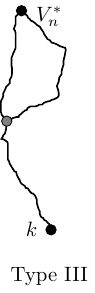}
\end{subfigure}
\begin{subfigure}{.2\textwidth}
\centering
	\includegraphics[width=.5\textwidth]{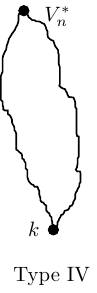}
\end{subfigure}
\caption{Possible structures for two distinct paths from $V_n^*$ to $k$.}\label{fig:multi-paths}
\end{figure}
 Let $r$ be the length of the path between $a,b$ that is disjoint of $\xx$. 
 Denote by $\mathcal{A}_l'(v,k,\xx,i)$ the event that the structure of type $i$ ($i$=I, II, III, IV) in Figure~\ref{fig:multi-paths} appears, where $\xx \in \fI_l(v,k)$.
Using an argument identical to \eqref{sus-simpl-1}, and applying Assumption~\ref{assumption-w}, it follows that
\begin{subequations}
\begin{equation}
\label{eqn:paths1}
\begin{split}
  &\sum_{l= 1}^{n^{\delta} (\log n)^2}\sum_{k\in [n]}w_{k}\sum_{\xx \in \fI_l(v,k)}(l-1)(l-2)\prob{\mathcal{A}'_l(v,k,\xx,\mathrm{I})}\\
  &\leq Cd_v'\expt{D_n'W_n}\sum_{l\geq 1}\sum_{r\geq 1} \frac{\sigma_3(n)^2}{\ell_n'}(l-1)(l-2)\nu_n'^{l+r-4}\\
&\leq C \frac{d_v'n^{6\alpha-3}}{1-\nu_n'} \sum_{l\geq 3}(l-1)(l-2)\nu_n'^{l-3}\leq C\frac{d_v'n^{6\alpha-3}}{(1-\nu_n')^4} =o(d_v' n^{\delta}),
\end{split}
\end{equation}where the $(l-1)(l-2)$ factor is due to the possible choices of $a,b$, $\sigma_3(n)^2$ is due to the two branch points for which three half-edges needs to be paired, 
and in the last step we have used the fact $6\alpha - 3 +3\delta <6\alpha - 3 +3\eta =0$ since $\delta<\eta$. 
Similarly, with $b = k$, we get the Type-II structures in Figure~\ref{fig:multi-paths}, and thus 
\begin{eq}\label{eqn:paths2}
&\sum_{l= 1}^{n^{\delta} (\log n)^2}\sum_{k\in [n]}w_{k}\sum_{\xx \in \fI_l(v,k)}(l-1)\prob{\mathcal{A}'_l(v,k,\xx,\mathrm{II})}\\
&\leq Cd_v'\sigma_3(n) \bigg(\frac{1}{\ell_n'^2}\sum_{k\in [n]} w_kd_k'(d_k'-1) \bigg)\sum_{l\geq 1} \sum_{r\geq 1} (l-1) (\nu_n')^{l+r -3} \\
&\leq C\frac{d'_vn^{6\alpha -3} }{(1-\nu_n')^3} = o(d_v'n^{\delta}),
\end{eq}
Again, 
\begin{eq}\label{eqn:paths3}
&\sum_{l= 1}^{n^{\delta} (\log n)^2}\sum_{k\in [n]}w_{k}\sum_{\xx \in \fI_l(v,k)}(l-1)\prob{\mathcal{A}'_l(v,k,\xx,\mathrm{III})}\\
&\leq Cd_v'(d_v'-1)\sigma_3(n) \bigg(\frac{1}{\ell_n'^2}\sum_{k\in [n]} w_kd_k' \bigg)\sum_{l\geq 1} \sum_{r\geq 1} (l-1) (\nu_n')^{l+r -3} \\
&\leq C\frac{d_v'^2 n^{6\alpha -3} }{n^{3\alpha-1}(1-\nu_n')^3} = o(d_v'^2 n^{\delta+1-3\alpha}),
\end{eq}
and 
\begin{eq}\label{eqn:paths4}
&\sum_{l= 1}^{n^{\delta} (\log n)^2}\sum_{k\in [n]}w_{k}\sum_{\xx \in \fI_l(v,k)}\prob{\mathcal{A}'_l(v,k,\xx,\mathrm{IV})}\\
&\leq Cd_v'(d_v'-1) \bigg(\frac{1}{\ell_n'^2}\sum_{k\in [n]} w_kd_k' (d_k'-1)\bigg)\sum_{l\geq 1} \sum_{r\geq 1}  (\nu_n')^{l+r -2} \\
&\leq C\frac{d_v'^2 n^{6\alpha -3} }{n^{3\alpha-1}(1-\nu_n')^2} = o(d_v'^2 n^{\delta+1-3\alpha}).
\end{eq}
\end{subequations}
Taking expectations with respect to $V_n^*$, all the terms in \eqref{eqn:paths1}, \eqref{eqn:paths2}, \eqref{eqn:paths3} and \eqref{eqn:paths4} are $o(n^{\delta})$, where we use that $\E[(D_n^*)^2] = O(n^{3\alpha -1})$.
To compute the leading contribution to \eqref{sus1-expr-2}, using \eqref{eq:quasi-poly-bound} and \eqref{eq:path-inclusion-excusion}, we lower bound
 \begin{equation}\label{sus-lb}
  \begin{split}
   &d_v' \sum_{l= 1}^{n^{\delta}(\log n)^2}\sum_{k\in [n]} w_k\sum_{\xx \in \fI_l(v,k)}\frac{1}{\ell_n'^{l-1}}\prod_{i=1}^{l-1}d_{x_i}'(d_{x_i}'-1)d_{k}' +o(1)\\
    &\geq \frac{d_v'\expt{D_n'W_n}}{\expt{D_n'}}\sum_{l = 1}^{n^{\delta}(\log n)^2}\bigg(\nu_n'^{l-1}-\frac{d_1'n^{\delta} (\log n)^2}{\ell_n'^{l-1}}\bigg(\sum_{i\in [n]}d_i'(d_i'-1)\bigg)^{l-2}\bigg)+o(1)  \\
    & = \frac{d_v'\expt{D_n' W_n}(1-(\nu_n')^{n^{\delta} (\log n)^2})}{\expt{D_n'}(1-\nu_n')}(1+o(1))=\frac{d_v'\expt{D_n' W_n}}{\expt{D_n'}(1-\nu_n')}(1+o(1)),
  \end{split}
 \end{equation}where we have used the fact that $d_1'l\leq d_1'n^{\delta} (\log n)^2$ and inclusion-exclusion to obtain the third step, and 
 \eqref{defn:barely-subcrit}, $d_1'n^{\eta}/\ell_n'= c_1/\mu_d(1+o(1))$ and the fact $(\nu_n')^{n^{\delta} (\log n)^2} \leq  \e^{-C(\log n)^2} = o(1)$ in the last step.
Thus, it follows that
\begin{equation}\label{susc-lb2}
 \expt{\mathscr{W}(V_n^*)}\geq  \frac{\expt{D_n^*}\expt{D_n' W_n}}{\expt{D_n'}(1-\nu_n')}(1+o(1)),
\end{equation}and the proof of Lemma~\ref{lem:expt-weight-random}~(i) is now complete using \eqref{sus-simpl-12}.
\end{proof}
\begin{remark}\label{rem:ub-susc} \normalfont It may be worthwhile to point out that the  upper bound  \eqref{sus-simpl-12} holds for any configuration model satisfying $\nu_n'<1-n^{-\varepsilon_0}$ for some $\varepsilon_0>0$ and $\sum_{i\in [n]} w_i^2 = O(n)$. The rest of Assumption~\ref{assumption-w} is not required in the proof of this upper bound. 
\end{remark}

\begin{proof}[Proof of Lemma~\ref{lem:expt-weight-random}~(ii)]
Note that 
\begin{eq}\label{moment-identity-path}
\big(\mathscr{W}(V_n^*)\big)^r &= \sum_{k_1,\dots,k_r\in [n]}w_{k_1}\cdots w_{k_r}\ind{V_n^*\leadsto k_1,\dots, V_n^*\leadsto k_r} = D_n(V_n^*,r)+E_n(V_n^*,r),
\end{eq}
where 
\begin{eq}
D_n(V_n^*,r) := \sum_{\substack{k_1,\dots,k_r\in [n] \\V_n^*, k_1,\dots,k_r \text{ distinct}}}w_{k_1}\cdots w_{k_r}\ind{V_n^*\leadsto k_1,\dots, V_n^*\leadsto k_r}.
\end{eq}
Let us formulate a general upper bound on $\E[D_n(V_n^*,r) ]$ using similar computations as in \eqref{sus-simpl-12}. 
Note that if $V_n^* \leadsto k_i$ for all $i\in [r]$, then we must have a tree $T$ with $V_n^*$ as root and $(k_i)_{i\in [r]}$ as leaves. 
Let us ``collapse'' all the degree-two vertices in $T$ except $V_n^*$. 
More precisely, we sequentially take a degree-two vertex (except $V_n^*$), delete it, and create an edge between its neighbors.
Denote the obtained tree by $T' = (\rV(T'), \rE(T'))$.  
Thus, $T'$ can be thought of as a rooted tree with $V_n^*$ being its root, and $(k_i)_{i\in [r]}$ being its leaves.
Also, $T'$ does not have any degree two vertices except possibly $V_n^*$.
Further, note that $r+1\leq |\rV(T')| \leq 2r$, and thus $r\leq |\rE(T')| \leq  2r - 1$. 
Let $m_i(T')$ denote the number of degree-$i$ vertices in $T'$ and let $d_0(T')$ be the degree of $V_n^*$ in $T'$.

Let $l_e$ be the number of edges that are  collapsed to create $e\in \rE(T')$. In that case, exactly $l_e-1$ degree-two vertices get collapsed in $T'$.
Using~\eqref{eq:quasi-poly-bound}, we can restrict ourselves to the case  $l_e\leq n^{\delta} (\log n)^2$, and the error due to such a restriction is given by $(\ell_n^w)^r n^{1+\delta} \e^{-C'(\log n)^2}$.
For each $T'$ described above, $i$ half-edges of $V_n^* = v$ are being paired, for which there are $ \prod_{i=1}^{d_0(T')}(d_v' - i+1)$ possible ways.
Moreover, the number of choices of the $r$ distinct leaves $(k_i)_{i\in [r]}$ gives rise to the factor $\big(\sum_{k\in [n]} d_k'w_k\big)^r$ in the path counting. 
Define 
\begin{eq}\label{defn-gen-R-path}
 Q_n(T'):= \prod_{j\geq 3}(\sigma_j(n))^{m_j(T') - \1\{j=d_0(T')\}} , \quad R_n(T') := \sum_{\substack{(l_1,\dots,l_{\sss |\rE(T')|}): \\ 1\leq l_e \leq n^{\delta}(\log n)^2 }} (\nu_n') ^{\sum_{e= 1}^{|\rE(T')|}(l_e-1)}.
\end{eq}
Note that $Q_n(T')$ gives the contribution due to the pairing of the half-edges of the vertices in  $\rV(T') \setminus \{V_n^*,k_1,\dots,k_r\}$, and  $R_n(T')$ is the total contribution due to degree-two vertices of possible trees $T$ that could give rise to $T'$ after collapsing. 
Thus, 
\begin{eq}\label{susc-general-ub}
\E\big[D_n(V_n^*,r)\vert V_n^* = v\big]& \leq  (1+o(1))\sum_{T'} \prod_{i=1}^{d_0(T')} (d_v' - i+1)   \\
&\  \times\bigg(\frac{1}{\ell_n'}\sum_{k\in [n]} d_k'w_k\bigg)^r 
Q_n(T') R_{n}(T') + C(\ell_n^w)^r n^{1+\delta} \e^{-C'(\log n)^2},
\end{eq}
for constants $C, C'>0$.

Let us now apply \eqref{moment-identity-path} and \eqref{susc-general-ub} for the special case $r=2$.
\begin{figure}\centering
\begin{subfigure}{.4\textwidth}
\centering
	\includegraphics[width=.5\textwidth]{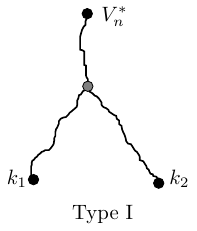}
\end{subfigure}
\begin{subfigure}{.4\textwidth}
\centering
	\includegraphics[width=.5\textwidth]{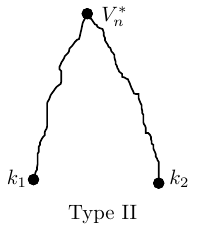}
\end{subfigure}
\caption{Possible paths when $V_n^*\leadsto k_1$, and $V_n^*\leadsto k_2$.}\label{fig:paths2}
\end{figure}
Figure~\ref{fig:paths2} describes the possible structures $T'$.
Application of  \eqref{susc-general-ub} yields  
\begin{equation}\label{W-moment2-ub}
  \begin{split}
  \E\big[D_n(V_n^*,2)\big]&\leq 
 \frac{\expt{D_n^*}(\sum_{k\in [n]} d_k'w_k)^2\sigma_3(n)}{\ell_n'^2(1-\nu_n')^3}+\frac{\expt{D_n^*(D_n^*-1)}(\sum_{k\in [n]} d_k'w_k)^2}{\ell_n'^2(1-\nu_n')^2}.
  \end{split}
\end{equation}
The two terms are $O(n^{3\alpha+3\delta - 1})$ and $O(n^{3\alpha+2\delta - 1})$ respectively.
Also, 
\begin{eq}\label{E-n-estimates}
\E\big[E_n(V_n^*,2)\big] \leq \expt{(W_n^*)^2}+ 2 \E[W_n^*\sW(V_n^*)] + \E\bigg[\sum_{k\in [n]} w_k^2 \ind{V_n^* \leadsto k}\bigg],
\end{eq}
where the first term is due to $V_n^* = k_1 = k_2$, the second term is due to $V_n^* = k_1$ but $V_n^* \neq k_2$ or $V_n^* = k_2$ but $V_n^* \neq k_1$, while the third term due to $k_1= k_2$ but $V_n^* \neq k_1$.
The three terms are respectively $O(n^{3\alpha-1})$, $O(n^{3\alpha+\delta - 1})$ and $O(n^{3\alpha+\delta - 1})$, where we have used \eqref{sus-simpl-1} to compute the second term, and an analogous computation as in \eqref{sus-simpl-1} to compute the final term (replacing $w_k$ by $w_k^2$ in \eqref{sus-simpl-1}). 
This proves our required upper bound that 
\begin{eq}\label{upper-second-moment}
\E\big[\big(\mathscr{W}(V_n^*)\big)^2\big] \leq (1+o(1))\frac{\expt{D_n^*}(\expt{D_n'W_n})^2\sigma_3(n)}{(\expt{D_n'})^2(1-\nu_n')^3}.
\end{eq}

\end{proof}

\begin{proof}[Proof of Lemma~\ref{lem:expt-weight-random}~(iii)]
We again use \eqref{moment-identity-path} and  \eqref{susc-general-ub}. 
For the third moment, the leading contributions to $\E[D_n(V_n^*,3)]$ arise from one of the structures given in Figure~\ref{fig:paths3}. 
\begin{figure} \centering
\begin{subfigure}{.2\textwidth}
 \centering
	\includegraphics[width=.8\textwidth]{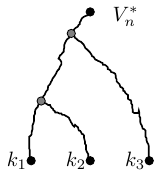}
\end{subfigure}
\begin{subfigure}{.2\textwidth}
\centering
	\includegraphics[width=.8\textwidth]{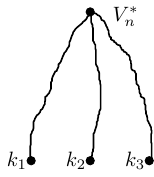}
\end{subfigure}
\begin{subfigure}{.2\textwidth}
\centering
	\includegraphics[width=.8\textwidth]{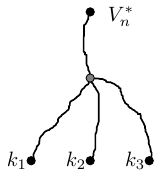}
\end{subfigure}
\begin{subfigure}{.2\textwidth}
\centering
	\includegraphics[width=.8\textwidth]{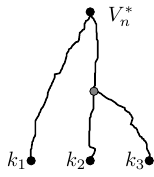}
\end{subfigure}
\caption{Possible paths when $V_n^*\leadsto k_1$, $V_n^*\leadsto k_2$, and $V_n^*\leadsto k_3$.}\label{fig:paths3}
\end{figure}
We will use the fact that $\E[(D_n^*)^{r-1}], \sigma_r(n) = O(n^{r\alpha -1})$ for $r\geq 3$.
The contributions on $\E[D_n(V_n^*,3)]$ due to the first type of tree in Figure~\ref{fig:paths3} are upper bounded  by 
\begin{eq}
 &\frac{C\sigma_3(n)^2 }{(1-\nu_n')^5}  = O(n^{6\alpha - 2+5\delta}) = O(n^{1+2\delta + (6\alpha -3+3\delta)}) = o(n^{1+2\delta}),
\end{eq}
where in the last step we have used the fact that $6\alpha - 3+\delta < 6\alpha - 3+1 - 2\alpha = 2(2\alpha - 1)<0$. 
The contributions due to the other three types of trees are respectively upper bounded  by 
\begin{gather*}
    \frac{C \E[(D_n^*)^3]}{(1-\nu_n')^3} = O(n^{4\alpha - 1 +3\delta}), \quad \frac{C\sigma_4(n)}{(1-\nu_n')^4} = O(n^{4\alpha -1+4\delta}), \quad \frac{C \E[(D_n^*)^2]\sigma_3(n)}{(1-\nu_n')^4} = O(n^{6\alpha - 2+4\delta}),
\end{gather*}
all of which are $o(n^{1+2\delta})$. 
Also, 
\begin{eq}
\E[E_n(V_n^*,3)] &\leq \E[(W_n^*)^3] \\
&+ \E\bigg[\sum_{k\in [n]}w_k^3 \ind{V_n^* \leadsto k}\bigg]+ 3 \E\bigg[W_n^* \sum_{k\in [n]}w_k^2 \ind{V_n^*\leadsto k}\bigg]+3 \E\big[(W_n^*)^2 \sW(V_n^*)\big] \\
&+ 3 \E[W_n^* (\sW(V_n^*))^2] + 3 \E\bigg[\sum_{k_1,k_2\in [n]}w_{k_1}^2 w_{k_2} \ind{V_n^*\leadsto k_1, V_n^* \leadsto k_2}\bigg],
\end{eq}
where the first term is due to $|\{V_n^*,k_1,k_2,k_3\}| = 1$, the three cases in the second line are due to $|\{V_n^*,k_1,k_2,k_3\}| = 2$ 
and the final two cases are due to $|\{V_n^*,k_1,k_2,k_3\}| = 3$.
Using the fact that $\max_{i\in [n]} w_i = O(n^{\alpha})$, we can use the estimates in \eqref{W-moment2-ub} and \eqref{E-n-estimates} to show that the first term is $O(n^{4\alpha - 1})$, the next three terms are $O(n^{4\alpha + \delta - 1})$, and the last two terms are  $O(n^{4\alpha + 3\delta - 1})$.
All these contributions are $o(n^{1+2\delta})$ and hence we conclude that $\E[(\sW(V_n^*))^3] = o(n^{1+2\delta})$. 
\end{proof}

\subsection{Analysis of the susceptibility function \texorpdfstring{$s_2^\star$}{TEXT}}\label{sec:s2} 

 \begin{proof}[Asymptotics of $s_2^\star$] The asymptotics of $s_2^\star$ is a   consequence of the Chebyshev inequality. Denote $\ell_n^w=\sum_{i\in [n]}w_i$. First, if $\E_{\bld{d}'}$ denotes the conditional expectation given $\mathrm{CM}_n(\bld{d}')$, then for any $r\geq 1$,
\begin{equation}\label{eq:identity-moment}
\begin{split}
\E_{\bld{d}'}\big[\big(\mathscr{W}(V_n^*)\big)^{r-1}\big] 
= \sum_{i\geq 1}\sum_{k\in \pCli}\frac{w_k}{\ell_n^w}\bigg(\sum_{l\in \pCli}w_l\bigg)^{r-1} = \frac{1}{\ell_n^w}\sum_{i\geq 1}(\Wli)^r.
 \end{split}
\end{equation} Therefore, using Lemma~\ref{lem:expt-weight-random} and \eqref{defn:barely-subcrit}, it follows from Assumption~\ref{assumption-w} that
\begin{equation} \label{expt:s2}
 n^{-\delta}\expt{s_2^\star}= \frac{\ell_n^w}{n} n^{-\delta}\expt{\mathscr{W}(V_n^*)}\to \frac{\mu_{d,w}^2}{\mu_d\lambda_0},
\end{equation}where we have used the fact that $\expt{D_n^*}\to \mu_{d,w}/\mu_w$. It remains to compute the variance. Let $U_n^*$ denote another vertex chosen in a size-biased way with the sizes being $(w_i)_{i\in [n]}$, independently of the graph $\mathrm{CM}_n(\bld{d}')$ and $V_n^*$.  
Then \eqref{eq:identity-moment} yields
 \begin{equation}\label{eq:moment2-s2}
 \begin{split}
  &\expt{(s_2^{\star})^2}
  =\frac{1}{n^2}\E\bigg[\sum_{i,j\geq 1}\Wli^2\Wlj^2\bigg]
  =\frac{1}{n^2}\E\bigg[\sum_{i\geq 1}\Wli^4\bigg]+\frac{1}{n^2}\E\bigg[\sum_{i\neq j}\Wli^2\Wlj^2\bigg]\\
  & = \frac{\ell_n^w}{n}\frac{1}{n}\expt{\big(\mathscr{W}(V_n^*)\big)^3}+\bigg(\frac{\ell_n^w}{n}\bigg)^2\expt{\mathscr{W}(U_n^*)\mathscr{W}(V_n^*)\ind{U_n^*\notin \mathscr{C}'(V_n^*)}},
 \end{split}
\end{equation}
where the equality of the second term in the third equality follows using similar arguments as in \eqref{eq:identity-moment}. 
Denote the last two terms of \eqref{eq:moment2-s2} by $(\mathbf{I})$ and $(\mathbf{II})$ respectively.
To estimate~$(\mathbf{II})$, observe that, conditionally on the graph $\mathscr{C}'(V_n^*)$, the graph obtained by removing $\mathscr{C}'(V_n^*)$ from $\rCM_n(\bld{d}')$ is again a configuration model with the induced degree sequence $\Mtilde{\boldsymbol{d}}$ and number of vertices $\tilde{n}$. Let $\tilde{\nu}_n$ denote the corresponding criticality parameter. 
 In the proof of Lemma~\ref{lem:expt-weight-random}~(i), we observed  that the upper bound holds whenever $\tilde{\nu}_n<1 - n^{-\varepsilon}$ with $\varepsilon \in (0,1)$ (see Remark~\ref{rem:ub-susc}). To this end, let us show that there exists $\varepsilon_0\in (0,1)$ and $c_1>0$ such that  for all sufficiently large $n$,
\begin{equation}\label{fact:graphs}
 \prob{\tilde{\nu}_n<1-n^{-\varepsilon_0}\mid \mathscr{C}'(V_n^*)}=1, \text{ with probability at least }1-\e^{-n^{c_1}}.
\end{equation}
Denote $\ell_n' = \sum_{i\in [n]}d_i'$. To see \eqref{fact:graphs}, first notice that 
 \begin{equation}
\begin{split}
 &\tilde{\nu}_n-1 = \nu_n'-1-\frac{\sum_{j \in \mathscr{C}'(V_n^*)} d_{j}'(d_{j}'-2)}{\sum_{j \in [n]} d_{j}'}+(\tilde{\nu}_n-1)\frac{\sum_{j \in \mathscr{C}'(V_n^*)} d_{j}'}{\ell_n'}.
%
 \end{split}
\end{equation}
Moreover, for any connected graph $\mathcal{G}$, $\sum_{i\in \mathcal{G}}d_i'(d_i'-2)\geq -2$ (this can be proved by induction) 
so that
\begin{eq}\label{nu-tilde-prime-compare}
(\tilde{\nu}_n-1)\bigg(1-\sum_{j \in \mathscr{C}'(V_n^*)} \frac{d_{j}'}{\ell_n'}\bigg)\leq \nu_n'-1+\frac{2}{\ell_n'}. 
\end{eq}
Next we use the following: 
\begin{fact}\label{fact:total-edges-random-comp}
There exists $c_0,c_1>0$ (sufficiently small), and $n_0\geq 1$ such that for all $n\geq n_0$, $\PR(\sum_{j \in \mathscr{C}'(V_n^*)} d_{j}' \geq n^{\alpha+\delta+c_0})\leq \e^{-n^{c_1}}$. 
\end{fact}
The proof of Fact~\ref{fact:total-edges-random-comp} follows using the exploration process in  Section~\ref{sec:barely-subcritical-mass}, and martingale concentration inequalities such as \cite{Fre75} (see Appendix~\ref{sec:appendix-barely-subcrit-large-deviation} for a detailed proof). 
The proof of \eqref{fact:graphs} now follows using $\ell_n' = \Theta(n)$, \eqref{defn:barely-subcrit},  and Fact~\ref{fact:total-edges-random-comp}.

As mentioned in Remark~\ref{rem:ub-susc}, now we can apply the upper bound from  \eqref{sus-simpl-1}. 
Therefore,
 \begin{equation}\label{weight-conditioned-comp}
\begin{split}
 &\expt{\mathscr{W}(U_n^*)\ind{U_n^* \notin \mathscr{C}'(V_n^*)} \big\vert \mathscr{C}'(V_n^*)}\\
 &= \frac{\sum_{i\notin \mathscr{C}'(V_n^*)}w_i}{\ell_n^w}\expt{\mathscr{W}(U_n^*) \big\vert \mathscr{C}'(V_n^*), U_n^* \notin \mathscr{C}'(V_n^*)}\\
 & \leq \frac{\sum_{i\in [n]}w_i^2}{\ell_n^w}+ \frac{\sum_{i\notin \mathscr{C}'(V_n^*)}w_i}{\ell_n^w} \times \frac{\big(\sum_{i\notin \mathscr{C}'(V_n^*)}d_i'w_i\big)^2\sum_{i\notin\mathscr{C}'(V_n^*)}d_i'}{\sum_{i\notin \mathscr{C}'(V_n^*)} w_i\sum_{i\notin \mathscr{C}'(V_n^*)} d_i' (-\sum_{i\notin\mathscr{C}'(V_n^*)}d_i'(d_i'-2)) }\\
 &\leq \frac{\big(\sum_{i\in [n]}d_i'w_i\big)^2}{\ell_n^w\ell_n'(1-\nu_n'+\oP(1))}+O(1)= \expt{\mathscr{W}(V_n^*)}\big(1+\oP(1)\big),
 \end{split}
\end{equation}where the penultimate step again follows from \eqref{nu-tilde-prime-compare}.
Thus,
\begin{equation}\label{cross-prod-var-s2}
 \expt{\mathscr{W}(U_n^*)\mathscr{W}(V_n^*)\ind{U_n^* \notin \mathscr{C}'(V_n^*)} }\leq \big(\expt{\mathscr{W}(V_n^*)}\big)^2\big(1+o(1)\big).
\end{equation} 
Now, \eqref{eq:moment2-s2}, \eqref{cross-prod-var-s2} together with Lemma~\ref{lem:expt-weight-random} implies that $\var{s_2^\star}=o(n^{2\delta})$.
We can use the Chebyshev inequality and \eqref{expt:s2} to conclude that $n^{-\delta}s_2^\star \xrightarrow{\sss \PR} \mu_{d,w}^2/(\mu_d\lambda_0).$ 
 \end{proof} 

\begin{remark}\label{rem:spr-asymp} \normalfont The method used to obtain the asymptotics of $s_2^\star$ can also be followed verbatim to obtain the asymptotics of $s_{pr}^\star$. Indeed, notice that 
\begin{equation}
 \expt{s_{pr}^\star} = \frac{1}{n}\E\bigg[\sum_{i\geq 1}\Wli|\pCli|\bigg] = \expt{\mathscr{W}(V_n)}.
\end{equation} 
A similar identity for the second moment of $s_{pr}^\star$ also holds.
\end{remark}

\subsection{Barely subcritical masses}\label{sec:barely-subcritical-mass}
We now prove the asymptotics of $\Wlj$ in Theorem~\ref{thm:susceptibility}.
The idea is to obtain the asymptotics for $\mathscr{W}(j)$ for each fixed $j$, and then show that 
$\Wlj=\mathscr{W}(j)$ with high probability. Consider the breadth-first exploration of the graph  starting from vertex $j$ as follows:
 \begin{algo} \label{algo:explor-barely-sub}\normalfont
 The algorithm carries along three disjoint sets of half-edges: \emph{active, neutral, dead}. 
\begin{itemize}
\item[(S0)] At stage $i=0$, the half-edges incident to $j$ are active and all the other half-edges are neutral. Order the initially active half-edges arbitrarily.
\item[(S1)]  At each stage, take the \emph{smallest} half-edge $e$ and pair it with another half-edge $f$, chosen uniformly at random from the set of half-edges that are either active or neutral. If $f$ is neutral, then the vertex $v$ to which $f$ is incident, is not discovered yet. 
Declare the half-edges incident to $v$ to be active and \emph{larger} than all other active vertices (choose any order between the half-edges incident to $v$). 
Declare $e,f$ to be dead.
\item[(S2)] Repeat from (S1) until the set of active half-edges is empty.
\end{itemize}
\end{algo}
\noindent Define  the process $\mathbf{S}_n^j$ by
$S_n^j(l)= S_n^j(l-1)+d_{\sss (l)}'J_l-2,$ and $S_n^j(0)= d_j'$,
where $J_l$ is the indicator that a new vertex is discovered at time $l$ and $d_{\sss (l)}'$ is the degree of the discovered vertex, if any. 
Thus, when the exploration starts from vertex $j$, then $S_n^j$ tracks the number of active half-edges.
Let $L:=\inf\{l\geq 1:S_n^j(l)=0\}$. 
 By convention, we assume that $S_n^j(l)=0$ for $l>L$. 
 Let $\mathscr{V}_l$ denote the vertex set discovered up to time $l$ excluding $j$ and $\mathcal{I}_i^n(l):=\ind{i\in\mathscr{V}_l}$. Define $\mathcal{I}_j^n(l)\equiv 0$. 
 Also, let $\mathscr{F}_l$ denote the sigma-field containing all the information upto time $l$ in Algorithm~\ref{algo:explor-barely-sub}. 
Note that 
   \begin{equation}
   \begin{split}
    S_n^j(l)&= d_j'+\sum_{i\in [n]} d_i' \mathcal{I}_i^n(l)-2l=d_j'+\sum_{i\in [n]} d_i' \bigg( \mathcal{I}_i^n(l)-\frac{d_i'}{\ell_n'}l\bigg)+\left( \nu_n'-1\right)l.
    \end{split}
   \end{equation}  
    Consider the re-scaled process $\bar{\mathbf{S}}^j_n$ defined as $\bar{S}^j_n(t)= n^{-\alpha}S_n^j(\floor{tn^{\alpha+\delta} })$. Then, using Assumption~\ref{assumption-w},
   \begin{equation} \label{eqn::scaled_process_j}
    \bar{S}_n^j(t)= c_j+n^{-\alpha} \sum_{i\in [n]}d_i'\bigg( \mathcal{I}_i^n(tn^{\alpha+\delta})-\frac{d_i'}{\ell_n'}tn^{\alpha+\delta} \bigg)- \lambda_0 t +o(1).
   \end{equation}The following three lemmas determine the asymptotics of $\Wli$ and $s_3^\star$: 
\begin{lemma} \label{lem:exploration::subcritical-j}
 Let $L_j$ be the function with  $L_j(t)=c_j-\lambda_0t$ for $t\in [0,c_j\lambda_0^{-1}]$ and $L_j(t)=0$ for $t>c_j\lambda_0^{-1}$. Then, under {\rm Assumption~\ref{assumption-w}}, as $n\to\infty$,
  \begin{eq}
  \sup_{t\leq c_j\lambda_0^{-1}}\big| \bar{S}^j_n(t) - L_j(t)\big| \pto 0.
  \end{eq}
\end{lemma}
\begin{lemma}\label{lem:weight-prop-l} For any $T>0$,
$\sup_{l\leq Tn^{\alpha+\delta}}\Big|\sum_{i\in [n]}w_i\mathcal{I}_i^n(l)- \frac{\sum_{i\in [n]}d_i'w_i}{\sum_{i\in [n]}d_i'}l\Big|=\oP(n^{\alpha+\delta}).$
\end{lemma}
\begin{lemma}\label{lem:hdeg-hweight}
 Fix any $j\geq 1$. Then with high probability $\mathscr{W}(j) = \Wlj$.
\end{lemma}
\begin{proof}[Asymptotics of $\Wlj$]Note that, since the exploration process explores one edge at each time, Lemma~\ref{lem:exploration::subcritical-j} implies that (see e.g. \cite[Theorem 13.6.4]{W02})
\begin{equation}\label{eq:weight-j-subcrit-1}
 \frac{1}{2n^{\alpha+\delta}}\sum_{k\in \mathscr{C}'(j)}d_k'\pto \frac{c_j}{\lambda_0}.
\end{equation}Moreover, Lemma~\ref{lem:weight-prop-l} yields that 
\begin{equation}\label{eq:weight-j-subcrit-2}
\frac{1}{n^{\alpha+\delta}}\mathscr{W}(j)= \frac{1}{n^{\alpha+\delta}}\sum_{k\in \mathscr{C}'(j)}w_k = \frac{\sum_{i\in [n]}d_i'w_i}{\ell_n'n^{\alpha+\delta}}\frac{1}{2}\sum_{k\in \mathscr{C}'(j)}d_k'+\oP(1) \pto \frac{ \mu_{d,w}}{\mu_d \lambda_0}c_j.
\end{equation}Now the asymptotics of $\Wlj$ in Theorem~\ref{thm:susceptibility} follows by an application of Lemma~\ref{lem:hdeg-hweight}.
\end{proof}
Next we provide a proof for Lemma~\ref{lem:hdeg-hweight}. 
The proofs of Lemmas~\ref{lem:exploration::subcritical-j}~and~\ref{lem:weight-prop-l} follow using similar techniques as in \cite{DHLS16}, and thus are provided in Appendix~\ref{sec:appendix-barely-subcrit}.
\begin{proof}[Proof of Lemma~\ref{lem:hdeg-hweight}]
Recall the definition of $\pClj,\Wlj$  from  \eqref{defn:pCli}.
For fixed $K\geq 1$, if all the components $(\pClj)_{j\in [K]}$ are disjoint, then $j=\min\{k:k\in\pClj\}$, i.e., $j$ is the minimum index among the vertices in $\pClj$.
In that case, $\sW(j) = \Wlj$.
Thus, it is enough to show that, for each fixed $i,j\geq 1$, 
\begin{eq}\label{eq:i-j-diff-comp-whp}
\PR(i \text{ and } j \text{ are in same connected component}) \to 0.
\end{eq}
If $i,j$ are in the same connected component, then the $\bar{\mathbf{S}}^j_n$ will have a jump of size $d_i' = (1+o(1)) c_i n^{\alpha}$.
By Lemma~\ref{lem:exploration::subcritical-j}, and the fact that $L_j$ is continuous, it follows that the probability of $\bar{\mathbf{S}}^j_n$ having a jump of size at least $\varepsilon n^{\alpha}$ tends to zero for any fixed $\varepsilon >0$. 
Thus we conclude \eqref{eq:i-j-diff-comp-whp} and the proof follows. 
\end{proof}

\subsection{Analysis of the susceptibility function \texorpdfstring{$s_3^{\star}$}{TEXT}} \label{sec:s3}The  aim of this section is to prove the following proposition which estimates the contribution on $s_3^\star$ due to components $(\pCli)_{i>K}$:
\begin{proposition}\label{prop:tail3-bare-subcrit}
Suppose that {\rm Assumption~\ref{assumption-w}} holds. 
For any $\varepsilon > 0$,
\begin{equation}
 \lim_{K\to\infty}\limsup_{n\to\infty}\PR\bigg(\sum_{i>K}\big(\Wli\big)^3>\varepsilon n^{3(\alpha+\delta)}\bigg)=0.
\end{equation}
\end{proposition}
\begin{proof}
 Let $\mathcal{G}^{\sss K}$ denote the graph obtained by deleting all the edges incident to the vertices in $[K]$. In this proof, a superscript $K$ to any previously defined object will correspond to the object in $\mathcal{G}^K$. Note that $\mathcal{G}^{\sss K}$ is again distributed as a configuration model conditioned on the new degree sequence $\bld{d}^{\sss K} = (d_i^{\sss K})_{i\in [n]}$. 
 We also augment a previously defined notation with $K$ in the superscript to denote the corresponding quantity for $\mathcal{G}^K$.
 Note that $\sum_{i\in [n]} d_i^{\sss K} \geq  \ell_n' - 2\sum_{i\in [K]}d_i' = \ell_n' (1+ O(n^{\alpha - 1}))$. 
 Also, $d_i^{\sss K} = 0$ for all $i\in [K]$, and $d_i^{\sss K} \leq d_i'$ for all $i\in [n]\setminus [K]$.
Recall the definition of $c_i$'s from Assumption~\ref{assumption-w}.
 First, for each fixed $K\geq 1$, 
  \begin{equation}\label{eqn:nu-K}
  \begin{split}
  \nu^{\sss K}_n&:= \frac{\sum_{i\in [n]} d_i^{\sss K}(d^{\sss K}_i-1)}{\sum_{i\in [n]} d_i^{\sss K}}\leq \frac{\sum_{i>K}d_i'(d_i'-1)}{\ell_n' (1+ O(n^{\alpha - 1}))}\\
  &=\frac{\sum_{i\in [n]} d_i'(d_i'-1)}{\ell_n'} - \frac{\sum_{i\in [K]} d_i'(d_i'-1)}{\ell_n'} + \frac{\sum_{i>K} d_i'(d_i'-1) \sum_{i\in [K]}d_i'}{\ell_n'\sum_{i>K}d_i} + O(n^{\alpha-1})\\
  & = \nu_n' +O(n^{2\alpha - 1} ) + O(n^{\alpha - 1}) = 1- \lambda_0 n^{-\delta} +o(n^{-\delta}),
  \end{split}
  \end{equation}
  where we have used the fact that $\delta<\eta = 1-2\alpha<1-\alpha$ in the last step. 
  We aim to apply the upper bound \eqref{upper-second-moment}. 
  Since we have only deleted $K =O(1)$ many vertices and $\Theta(n^{\alpha})$ many half-edges to obtain $\cG^{\sss K}$, it follows that $(d_i^{\sss K})_{i\in [n]}$ also satisfies Assumption~\ref{assumption-w}.
  We can apply the upper bound in \eqref{upper-second-moment}, and thus
  \begin{equation}\label{est:3rd-susc-order}
  \begin{split}
  \frac{1}{n}\E\bigg[\sum_{i}\big(\Wli^{\sss K}\big)^3\bigg]&= \frac{\ell_n^w}{n}\expt{\big(\mathscr{W}^{\sss K}(V_n^*)\big)^2}
  \leq C\frac{\sum_{i>K} d_i'(d_i'-1)(d_i'-2)}{\sum_{i>K} d_i'(1-\nu_n^{\sss >K})^3} \\
  &\leq C n^{3\alpha +3\delta - 1} \bigg(n^{-3\alpha} \sum_{i>K} d_i'^3\bigg),
  \end{split}
  \end{equation}
  which tends to zero in the iterated limit $\lim_{K\to\infty}\limsup_{n\to\infty}$.
  Therefore, using the Markov inequality and the fact that $\bld{c}\in \ell^3_{\shortarrow}\setminus \ell^2_{\shortarrow}$, it follows that, for any $\varepsilon>0$,
\begin{equation}\label{eq:delete-K-susc}
\lim_{K\to\infty}\limsup_{n\to\infty}\PR\bigg(\sum_{i\geq 1}\big(\Wli^{\sss K}\big)^3>\varepsilon n^{3(\alpha+\delta)}\bigg) = 0.
\end{equation}  
    Now, the proof is complete by  observing that $
  \sum_{i>K}\Wli^3\leq \sum_{i\geq 1}(\Wli^{\sss K})^3.$
  \end{proof}
\begin{remark}\label{rem:3rd-suscep-K}\normalfont Notice that the proof of Proposition~\ref{prop:tail3-bare-subcrit} can be modified to conclude the similar results for  $\sum_{i>K}\big(\Wli\big)^2|\pCli|$ and  $\sum_{i>K}\Wli|\pCli|^2$.
%
Indeed, an analogue of \eqref{est:3rd-susc-order} can be computed by observing that
$ \E[\sum_{i\geq 1}(\Wli^{\sss K})^2|\pCli^{\sss K}|] = n  \E[\mathscr{W}^{\sss K}(V_n)],$ and $\E[\sum_{i\geq 1}\Wli^{\sss K}(|\pCli^{\sss K}|)^2] = \ell_n^w \E[|\mathscr{C}'^{\sss K}(V_n^*)|^2].$
\end{remark}
Finally we prove the asymptotics of $s_3^\star$ stated in Theorem~\ref{thm:susceptibility}:
\begin{proof}[Asymptotics of $s_3^\star$]
The proof follows by combining the asymptotics of $\Wlj$ and Proposition~\ref{prop:tail3-bare-subcrit}.
\end{proof}

\begin{remark}\label{rem:3rd-suscep}\normalfont 
The argument for $s_3^\star$ can be followed verbatim to also conclude that 
\begin{eq}
n^{-3\alpha-3\delta}\sum_{i=1}^\infty\big(\Wli\big)^2|\pCli| \pto \frac{\mu_{d,w}^2}{\mu_d^2\lambda_0^3} \sum_{i=1}^\infty c_i^3, \quad  n^{-3\alpha-3\delta}\sum_{i=1}^\infty\Wli|\pCli|^2 \pto \frac{\mu_{d,w}}{\mu_d\lambda_0^3} \sum_{i=1}^\infty c_i^3.
\end{eq}
\end{remark}

\subsection{Mesoscopic typical distances}\label{sec:meso-dist}
In this section, we obtain the asymptotics of $\mathcal{D}_n^\star$ in Theorem~\ref{thm:susceptibility} using a similar analysis as in Section~\ref{sec:s2}. 
Again the proof involves the Chebyshev inequality where the moments are estimated using path counting. 
We sketch the computation of $\E[\mathcal{D}_n^\star]$.
%
Recall the notations $U_n^*$, $V_n^*$, 
$\mathcal{A}_l(v,k)$ 
and $\mathcal{A}'_l(v,k)$ 
from Section~\ref{sec:s2}. 
Note that
\begin{align*}
 \expt{\mathcal{D}_n^\star} &= \frac{1}{n}\E\bigg[\sum_{i,k\in [n]}w_iw_k\dst(i,k)\ind{k\in\mathscr{C}'(i)}\bigg] = \frac{\ell_n^w}{n}\E\bigg[\sum_{k\in [n]}w_k\dst(V_n^*,k)\ind{k\in\mathscr{C}'(V_n^*)}\bigg]
 \\
 &\leq  \frac{\ell_n^w}{n}\sum_{k\in [n]}w_k\sum_{l\geq 1}l\prob{\mathcal{A}_l(V_n^*,k)}=   \frac{\ell_n^w}{n} \sum_{l\geq 1} l\sum_{k\in [n]}w_k\prob{\mathcal{A}_l(V_n^*,k)},
\end{align*}
and also
\begin{equation}
\expt{\mathcal{D}_n^\star} \geq \frac{\ell_n^w}{n} \sum_{l\geq 1} l \bigg(\sum_{k\in [n]}w_k\big(\prob{\mathcal{A}_l(V_n^*,k)}-\prob{\mathcal{A}'_l(V_n^*,k)}\big)\bigg).
\end{equation}
Now compare the terms above  to \eqref{sus1-expr-1}, \eqref{sus1-expr-2}. The only difference is that there is an extra multiplicative $l$ here. 
Thus we can follows identical arguments as in the proof of  \eqref{sus-simpl-1}, \eqref{susc-lb2}, and at the final step, we can use that $\sum_{l\geq 1} l (\nu_n')^{l-1} = (1-\nu_n')^{-2}$. Thus, 
\begin{equation}
 \expt{\mathcal{D}^\star_n} = \frac{\expt{W_n}\expt{D_n^*}\expt{D_n'W_n}}{\expt{D_n'}(1-\nu_n')^2}(1+o(1))= \frac{(\expt{D_n'W_n})^2}{\expt{D_n'}(1-\nu_n')^2}(1+o(1)).
\end{equation}
The variance terms can also be computed similarly. Due to the presence of $l^2$ in the second moment, we can use $\sum_{l\geq 1} l(l-1) (\nu_n')^{l-2} = (1-\nu_n')^3$.
This gives rise to an additional fact $1/(1-\nu_n')^2 = O(n^{2\delta})$. Again, the identical arguments as \eqref{weight-conditioned-comp} can be applied to show that $\var{\mathcal{D}_n^{\star}}=o(n^{4\delta})$. 
Thus the proof of the asymptotics of $\mathcal{D}_n^{\star}$ follows. \qed

\section{Metric space limit for critical  percolation clusters}
\label{sec:proof-metric-mc}
The aim of this section is to complete the proof of Theorem~\ref{thm:main}. 
We start by defining the multiplicative coalescent process \cite{A97,AL98} that will play a pivotal role in this section.
\begin{defn}[Multiplicative coalescent]\label{defn:mul-coalescent} \normalfont
Consider a (possibly infinite) collection  of particles and let $\mathbf{X}(s)=(X_i(s))_{i\geq 1}$ denote the collection of masses of those particles at time $s$. 
Thus the $i$-th particle has mass $X_i(s)$ at time~$s$. 
The evolution of the system takes place according to the following rule at time $s$: At rate $X_i(s)X_j(s)$,  particles $i$ and $j$ merge into a new particle of mass $X_i(s)+X_j(s)$.
\end{defn}
Before going into the details, let us describe the general idea and the organization of this section. 
The proof combines many ingredients and ideas from \cite{BBSX14} and \cite{DHLS16}.
In Section~\ref{sec:dyn-cons} we  consider a dynamically growing process of graphs that approximates the percolation clusters in the critical window. 
Now, the graphs generated by this dynamic evolution satisfy the following properties: 
(i) In the critical window, the components merge \emph{approximately} as the multiplicative coalescent where the \emph{mass} of each component is approximately proportional to the component size; 
(ii) The masses of the barely  subcritical clusters satisfy \emph{nice} properties due to Theorem~\ref{thm:susceptibility}. 
In Section~\ref{sec:entrance-boundary-open-he}, we derive the required properties in the barely subcritical regime for the dynamically growing graph process using Theorems~\ref{thm:susceptibility}~and~\ref{thm:diam-max}.
Section~\ref{sec:struct-compare} is devoted to deriving scaling limits of functionals of $\mathcal{G}_n(t_c(\lambda))$.
In Section~\ref{sec:coupling-mul-coal}, we modify the dynamic process in such a way  that the components merge exactly as multiplicative coalescent.
Since the exact multiplicative coalescent corresponds to the rank-one inhomogeneous random graphs, thinking of these barely subcritical clusters as blobs, we use the universality theorem (Theorem~\ref{thm:univesalty}) in Section~\ref{sec:modified-graph} to determine the metric space limits of the largest components of the modified graph (Theorem~\ref{thm:mspace-limit-modified}). 
We finally complete the proof of Theorem~\ref{thm:main}
 in Section~\ref{sec:proof-thm1}.
  The proof of Theorem~\ref{thm:main-simple} is given in Section~\ref{sec:simple}.

\subsection{The dynamic construction and its properties}\label{sec:dyn-cons}
\begin{algo}[The dynamic construction]\label{algo:dyn-cons} \normalfont Let $\mathcal{G}_n(t)$ be the graph obtained up to  time $t$ by the following dynamic construction:
\begin{itemize}
 \item[(S0)]  Initially, any vertex $i$ has $d_i$ incident half-edges and all the half-edges are alive. During the construction, a half-edge can be in one of the following two sets: alive or dead. All the half-edges have an independent unit rate  exponential clock attached to them.
 \item[(S1)] Whenever a clock rings, we take the corresponding half-edge, kill it and pair it with a half-edge chosen uniformly at random among the alive half-edges. The paired half-edge is also killed and the exponential clocks associated with killed half-edges are discarded.
\end{itemize} 
\end{algo}  Since a half-edge is paired with another unpaired half-edge, chosen uniformly at random from the set of all unpaired half-edges, the final graph $\mathcal{G}_n(\infty)$ is distributed as $\mathrm{CM}_n(\boldsymbol{d})$. 
Define
\begin{equation}\label{eq:def-crit-time}
 t_c(\lambda)=\frac{1}{2}\log\bigg(\frac{\nu_n}{\nu_n-1}\bigg)+\frac{\nu_n}{2(\nu_n-1)}\frac{\lambda}{n^\eta}.
\end{equation}
We denote the $i$-th largest component of $\mathcal{G}_n(t)$ by $\mathscr{C}_{\sss (i)}(t)$.
In the subsequent part of this paper, we will derive the metric space limit of $(\mathscr{C}_{\sss (i)}(t_c(\lambda)))_{i\geq 1}$. 
The following lemma enables us to switch to the conclusions for the largest clusters of $\rCM_n(\bld{d},p_n(\lambda))$:
\begin{lemma}[{\cite[Proposition 24]{DHLS16}}]\label{lem:coupling-whp} There exists $\varepsilon_{n}=o(n^{-\eta})$ and a coupling such that, with high probability,
\begin{gather}
 \mathcal{G}_n(t_c(\lambda)-\varepsilon_{n})\subset \mathrm{CM}_n(\bld{d},p_n(\lambda)) \subset\mathcal{G}_n(t_c(\lambda)+\varepsilon_n),\\
  \mathrm{CM}_n(\bld{d},p_n(\lambda)-\varepsilon_n) \subset \mathcal{G}_n(t_c(\lambda)\subset \mathrm{CM}_n(\bld{d},p_n(\lambda)+\varepsilon_n).
  \end{gather}
\end{lemma}
Let $\omega_i(t)$ denote the number of unpaired/open half-edges incident to  vertex~$i$ at time $t$ in Algorithm~\ref{algo:dyn-cons}. 
We end this section by understanding the evolution of some functionals of the degrees and the open half-edges in the graph $\mathcal{G}_n(t)$.
Let $s_1(t)$ denote the total number of unpaired half-edges at time $t$. Denote also $s_2(t)=\sum_{i\in [n]} \omega_i(t)^2$, $s_{d,\omega}(t)=\sum_{i\in [n]}d_i\omega_i(t)$. 
Further, we write $\mu_n=\ell_n/n$.
\begin{lemma}\label{lem:total-open-he}  Under {\rm Assumption~\ref{assumption1}}, the quantities 
 $\sup_{t\leq T}|\frac{1}{n}s_1(t)-\mu_n\e^{-2t}|$, $\sup_{t\leq T}|\frac{1}{n} s_2(t) - \mu_n\e^{-4t}(\nu_n+\e^{2t})|,$ $\sup_{t\leq T}|\frac{1}{n}s_{d,\omega}(t) - \mu_n(1+\nu_n)\e^{-2t}|$  are all  $\OP(n^{-1/2})$, for any $T>0$.
\end{lemma}
\begin{proof}
 The proof uses the differential equation method \cite{wormald1995differential}. 
 Notice that, after each ring of an exponential clock in Algorithm~\ref{algo:dyn-cons}, $s_1(t)$ decreases by two. Let $Y$ denote a unit-rate Poisson process. Using the random time change representation \cite{EK86},
 \begin{equation}\label{eq:diff-eqn}
  s_1(t) = \ell_n - 2 Y\bigg(\int_{0}^t s_1(u)\mathrm{d} u\bigg) = \ell_n +M_n(t)-2\int_{0}^t s_1(u)\mathrm{d} u,
 \end{equation}where $\bld{M}_n$ is a martingale. 
 Now, the quadratic variation of $\bld{M}_n$ satisfies $ \langle M_n \rangle (t)\leq 4t\ell_n = O(n),$ which implies that $\sup_{t\leq T}|M_n(t)|= \OP(\sqrt{n}).$ 
 Moreover, notice that the function $f(t)=\mu_n\e^{-2t}$ satisfies $f(t)=\mu_n-2\int_0^tf(u)\mathrm{d}u$. 
  Therefore,
 \begin{equation}
 \begin{split}
  \sup_{t\leq T}\bigg|\frac{1}{n}s_1(t)- \mu_n\e^{-2t}\bigg| &\leq \sup_{t\leq T}\frac{|M_n(t)|}{n}+2\int_0^T\sup_{t\leq u}\bigg|\frac{1}{n}s_1(t)- \mu_n\e^{-2t}\bigg| \mathrm{d}u.
 \end{split}
 \end{equation} Using Gr\H{o}nwall's inequality \cite[Proposition 1.4, page 204]{M86}, it follows that
 \begin{equation}\label{eq:diff-eqn-gronwall}
   \sup_{t\leq T}\bigg|\frac{1}{n}s_1(t)- \mu_n\e^{-2t}\bigg|\leq \e^{2T} \sup_{t\leq T}\frac{|M_n(t)|}{n} =\OP(n^{-1/2}),
 \end{equation}as required. 
For $s_2(t)$, note that if half-edges corresponding to vertices $i$ and $j$ are paired, $s_2$ changes by $-2\omega_i-2\omega_j+2$ and if two half-edges corresponding to $i$ are paired, $s_2$ changes by $-4\omega_i+4$. Thus,
\begin{equation}
\begin{split}
\sum_{i\in [n]}\omega_i(t)^2&=\sum_{i\in [n]}d_i^2+M_n'(t)+\int_0^t \sum_{i\neq j}\frac{\omega_i(u)\omega_j(u)(-2\omega_i(u)-2\omega_j(u)+2)}{s_1(u)-1} \dif u\\
 &\hspace{4cm}+\int_0^t\sum_{i\in [n]}\frac{\omega_i(u)(\omega_i(u)-1)(-4\omega_i(u)+4)}{s_1(u)-1}\dif u\\
 & = n\mu_n(1+\nu_n)+M_n'(t)+\int_0^t(-4s_2(u)+2s_1(u))\mathrm{d}u+O(1),
 \end{split}
\end{equation}where $\bld{M}_n'$ is a martingale with quadratic variation given by $\langle M_n'\rangle (t) =O(n)$. 
Again, an estimate equivalent to \eqref{eq:diff-eqn-gronwall} follows using Gr\H{o}nwall's inequality.
Notice also that when a clock corresponding to vertex $i$ rings and it is paired to vertex $j$, then $s_{d,\omega}$ decreases by $d_i+d_j$. 
Thus,
 \begin{equation}
  \begin{split}
   s_{d,\omega}(t)&=\sum_{i\in [n]}d_i^2+M_n''(t)-\int_0^t \sum_{i\neq j}\frac{\omega_i(u)\omega_j(u)(d_i+d_j)}{s_1(u)-1} \mathrm{d}u-\int_0^t \sum_{i\in [n]}\frac{\omega_i(u)(\omega_i(u)-1)2d_i}{s_1(u)-1}\mathrm{d}u\\
   & = n\mu_n(1+\nu_n)+M_n''(t)- 2 \int_0^ts_{d,\omega}(u)\mathrm{d}u,
  \end{split}
 \end{equation}where $\bld{M}_n''$ is a martingale with quadratic variation given by $\langle M_n''\rangle (t) \leq 2t \sum_{i\in [n]}d_i^2=O(n)$. 
 We can now apply Gr\H{o}nwall's inequality as before. 
 The proof of Lemma~\ref{lem:total-open-he} is  complete. 
\end{proof}

\subsection{Entrance boundary for open half-edges}\label{sec:entrance-boundary-open-he}
Define
\begin{equation}\label{eq:def-subcrit-time}
 t_n = \frac{1}{2}\log\bigg(\frac{\nu_n}{\nu_n-1}\bigg)-\frac{\nu_n}{2(\nu_n-1)}\frac{1}{n^{\delta}}, \quad 0< \delta < \eta.
\end{equation} The goal is to show that the open half-edges satisfy the entrance boundary conditions. 
Let $\bld{d}(t)=(d_i(t))_{i\in [n]}$ denote the degree sequence of $\mathcal{G}_n(t)$ constructed by Algorithm~\ref{algo:dyn-cons}. Recall that $\mathcal{G}_n(t)$ is a configuration model conditionally on~$\bld{d}(t)$. 
Let us first derive the asymptotics of $\nu_n(t_n)$.
Recall that $\omega_i(t)$ denotes the number of open half-edges adjacent to vertex $i$ in $\mathcal{G}_n(t)$. Notice that 
\begin{equation}
 \nu_n(t_n)=\frac{\sum_{i\in [n]}(d_i-\omega_i(t_n))^2}{\ell_n-s_1(t_n)} -1 = \frac{\sum_{i\in [n]}d_i^2-2s_{d,\omega}(t_n)+s_2(t_n)}{\ell_n-s_1(t_n)}-1.
\end{equation}
Using Lemma~\ref{lem:total-open-he} and Assumption~\ref{assumption1}, 
\begin{equation}\label{nu-calc-1}
 \begin{split}
  &\frac{1}{n}(\ell_n-s_1(t_n))=\mu_n (1-\e^{-2t_n}) +\oP(n^{-\delta})
    = \frac{\mu_n}{\nu_n}\Big(1-\frac{\nu_n}{n^{\delta}}\Big)+\oP(n^{-\delta}),
 \end{split}
\end{equation}
\begin{equation}\label{nu-calc-2}
\begin{split}
 \frac{1}{n}\bigg(\sum_{i\in [n]}d_i^2-2s_{d,\omega}(t_n)+s_2(t_n)\bigg) 
 & = \frac{\mu_n}{\nu_n}\Big(2-\frac{3\nu_n}{n^{\delta}}\Big)+\oP(n^{-\delta}).
\end{split}
\end{equation}
Thus, \eqref{nu-calc-1} and \eqref{nu-calc-2} yields that 
$\nu_n(t_n) 
= 1- \nu_n n^{-\delta} +\oP(n^{-\delta}).$
Further, using the differential equation method again, the evolution of $(\omega_i(t))_{t\geq 0}$ is given by 
\begin{eq}
\omega_i(t) &= d_i + M_n(t) + \int_0^t \bigg[\frac{2\omega_i(u) (s_1(u) -\omega_i(u))}{s_1(u)-1} + \frac{2 \omega_i(u)(\omega_i(u)-1)}{s_1(u)-1}\bigg] \dif u \\ & = d_i + M_n(t) + 2\int_0^t \omega_i(u) \dif u,
\end{eq}
and Assumption~\ref{assumption1} yields that, for all $T>0$,
\begin{eq}\label{open-he-high-degree-asymp}
\sup_{t\leq T} |n^{-\alpha} \omega_i(t) - \theta_i \e^{-2t}| \pto  0.
\end{eq}

We aim to apply the results for the barely subcritical regime in Theorem~\ref{thm:susceptibility} to the number of open half-edges $\bld{\omega}(t_n)=(\omega_i(t_n))_{i\in [n]}$. 
Notice that, by  Lemma~\ref{lem:total-open-he}, \eqref{open-he-high-degree-asymp} and Assumption~\ref{assumption1}, $\bld{\omega}(t_n)$ and $\bld{d}(t_n)$ satisfy Assumption~\ref{assumption-w} with
\begin{equation}\label{eq:mu-mudw-open-he}
 \mu_{\omega} = \frac{\mu(\nu-1)}{\nu},\quad \mu_d = \frac{\mu}{\nu}, \quad \mu_{d,\omega}  = \frac{\mu(\nu-1)}{\nu} ,\quad c_i = \frac{\theta_i}{\nu}.
\end{equation} 
Let $\Cli (t)$ be defined analogously as \eqref{defn:pCli} for the graph $\cG_n(t)$.
Denote $f_i(t)= \sum_{k\in \Cli(t)}\omega_k(t)$ and $\bld{f}(t) = (f_i(t))_{i\geq 1}$.
The following theorem summarizes the entrance boundary conditions for $\bld{f}(t)$. 
Let $s_2^\omega$, $s_3^\omega$, $\mathcal{D}_n^\omega$ respectively denote the quantities $s_2^\star$, $s_3^\star$, $\mathcal{D}_n^\star$ respectively with the weights being the number of open half-edges, and the underlying graph being $\mathcal{G}_n(t_n)$.
\begin{theorem}\label{th:open-he-entrance}Under {\rm Assumption~\ref{assumption1}},  as $n\to\infty$,
\begin{gather}
 n^{-\delta}s_2^\omega \xrightarrow{\sss \PR} \frac{\mu (\nu-1)^2}{\nu^2},\quad n^{-\delta}s_{pr}^\omega \xrightarrow{\sss \PR}\frac{\mu (\nu-1)}{\nu^2}, \quad  n^{-(\alpha+\delta)}f_i(t_n) \xrightarrow{\sss \PR} \bigg(\frac{\nu-1}{\nu^2}\bigg)\theta_i\\
 n^{-3\alpha-3\delta+1}s_3^\omega \xrightarrow{\sss \PR} \bigg(\frac{\nu-1}{\nu^2}\bigg)^3 \sum_{i=1}^\infty \theta_i^3,\qquad n^{-2\delta} \mathcal{D}_n^\omega \xrightarrow{\sss \PR} \frac{\mu (\nu-1)^2}{\nu^3}.
 \end{gather}
\end{theorem}
\begin{remark}\label{rem:entrance-comp-size} \normalfont Setting $w_i =1$ for all $i$, we get the entrance boundary conditions for the component sizes also. In this case $\mu_d=\mu_{d,w}=\mu/\nu$. Replacing $\omega$ by $c$ in the above notation to denote the component susceptibilities, it follows that 
\begin{align*}
 n^{-\delta}s_2^c \pto \frac{\mu}{\nu^2},\quad n^{-(\alpha+\delta)}|\Cli(t_n)|\pto \frac{\theta_i}{\nu^2},\quad n^{-3\alpha-3\delta+1}s_3^c\pto \frac{1}{\nu^6}\sum_{i=1}^\infty \theta_i^3.
\end{align*}
\end{remark}
\subsection{Components of the dynamically constructed graph}
\label{sec:struct-compare}
The idea is to regard $(\Cli(t_n))_{i\geq 1}$, the connected components at time $t_n$, as blobs.
For $t\geq t_n$, the graph $\mathcal{G}_n(t)$ should be viewed as a super-graph with the superstructure being determined by the edges appearing after time $t_n$.  
  Thus, the components of $\mathcal{G}_n(t)$ can be regarded as a union of the blobs. 
  For a component $\mathscr{C}$, we use the notation $\bl(\mathscr{C})$ to denote the collection of indices corresponding to the blobs within $\mathscr{C}$ given by $\{b: \Clb(t_n)\subset\mathscr{C}, \Clb(t_n) \neq \varnothing\}$.
  Denote 
 $$ \mathcal{F}_{i}(t) = \sum_{b\in \bl(\mathscr{C}_{\sss (i)}(t))}f_{b}(t_n).$$
Let us denote the ordered components and the  $\mathcal{F}$-values of $\mathcal{G}_n(t_c(\lambda))$ simply by $(\mathscr{C}_{\sss (i)}(\lambda))_{i\geq 1}$ and $(\mathcal{F}_{i}(\lambda))_{i\geq 1}$ respectively. 
The goal in this section is to obtain the scaling of these component functionals, and also understand structural properties related to the surplus edges.
Recall that $\mathrm{SP}(\sC)$ denotes the number of surplus edges in the component $\sC$, i.e. $\mathrm{SP}(\sC) = \#\text{edges in }\sC - |\sC|+1$.
The following result gives the scaling limits of the rescaled component sizes and surplus edges of $\mathcal{G}_n(t_c(\lambda))$:
\begin{proposition}\label{thm:comp-functionals-original} Let $(\mathscr{C}_{\sss (i)}(\lambda))_{i\geq 1}$ denote the ordered vector of components sizes of the graph $\mathcal{G}_n(t_c(\lambda))$. Then, 
$\big(n^{-\rho}|\mathscr{C}_{\sss (i)}(\lambda)|,\surp{\mathscr{C}_{\sss (i)}(\lambda)}\big)_{i\geq 1} \xrightarrow{\sss d} (\frac{1}{\nu}\xi_i,\mathscr{N}_i)_{i\geq 1}$
as $n\to\infty$,
with respect to the topology on $\ell^2_\shortarrow\times \N^\N$, where the limiting objects are defined in Proposition~\ref{prop:comp-size}.
\end{proposition}
The proof is a direct consequence of Lemma~\ref{lem:coupling-whp} and Proposition~\ref{prop:comp-size}. See for example \cite[Proposition 25]{DHLS16}.
 The components consist of surplus edges within the blobs and the surplus edges in the superstructure. 
Next, let $\mathrm{SP}'(\mathscr{C}_{\sss (i)}(\lambda))$ denote the number of surplus edges in the superstructure of $\mathscr{C}_{\sss (i)}(\lambda)$. 
Thus $\mathrm{SP}'(\mathscr{C}_{\sss (i)}(\lambda))$ denotes the macroscopic surplus edges which are not inside some blob. 
The next result proves that all the surplus edges in the critical components are macroscopic. Further, it relates the component sizes and the  $\mathcal{F}$-values of $\mathcal{G}_n(t_c(\lambda))$:
\begin{proposition}\label{thm:comp-functional-original}
Assume that $\eta/2<\delta<\eta$. 
Then, for each $1\leq i\leq K$, the following hold:
\begin{enumerate}[(a)]
\item With high probability, $\mathrm{SP}'(\mathscr{C}_{\sss (i)}(\lambda)) = \mathrm{SP}(\mathscr{C}_{\sss (i)}(\lambda))$. Consequently, there are no surplus edges within blobs in $\mathscr{C}_{\sss (i)}(\lambda)$ with high probability;
\item $\mathcal{F}_i(\lambda)/|\mathscr{C}_{\sss (i)}(\lambda)|\pto \nu-1$. Consequently, $\big(n^{-\rho}\mathcal{F}_i(\lambda)\big)_{i\geq 1} \dto \frac{\nu-1}{\nu}\bld{\xi}$ with respect to the product topology.
\end{enumerate}
\end{proposition}
\noindent 
Since $\mathrm{SP}'(\mathscr{C}_{\sss (i)}(\lambda))\leq \mathrm{SP}(\mathscr{C}_{\sss (i)}(\lambda))$ almost surely, for Part (a) it suffices to show that 
 \begin{equation}\label{eq:surp-sam-dist-limit}
\mathrm{SP}'(\mathscr{C}_{\sss (i)}(\lambda))\text{ and }\mathrm{SP}(\mathscr{C}_{\sss (i)}(\lambda))\text{ have the same distributional limit}.
 \end{equation}
 Let $\mathcal{G}_n'$ denote the graph obtained from $\mathcal{G}_n(t_c(\lambda))$ by shrinking each blob to a single node. 
 Then, $\mathrm{SP}'(\cdot)$ can be viewed as the surplus edges in the components of $\mathcal{G}_n'$. 
The graph $\mathcal{G}_n'$ can also be viewed to be constructed dynamically as in Algorithm~\ref{algo:dyn-cons} with the degree sequence being $(f_i(t_n))_{i\geq 1}$.
In the following, we investigate the relations between $\mathcal{G}_n(t_n)$ and $\mathcal{G}_n'$. 
Lemma~\ref{lem:total-open-he} implies that the number of unpaired half-edges in $\mathcal{G}_n(t_n)$ that are paired in $\mathcal{G}_n(t_c(\lambda))$ is given by
\begin{equation}\label{eq:estimate-he-modi-supst}
s_1(t_n)-s_1(t_c(\lambda))  = n \mu_n( n^{-\delta} + \lambda n^{-\eta} )+\oP(n^{1-\gamma}), \quad \text{for some }\eta<\gamma.
\end{equation} 
Note that the we have used $\delta> \eta/2$ in \eqref{eq:estimate-he-modi-supst}.
 \begin{algo}\label{algo:perc-for-blob}\normalfont
  Define
 $ \pi_n = \frac{\nu_n}{\nu_n-1}( n^{-\delta} + \lambda n^{-\eta} )$
   and associate $f_i(t_n)$ half-edges to the vertex~$i$ of $\mathcal{G}_n'$. Construct the graph $\mathcal{G}_n'(\pi_n)$ as follows:
  \begin{enumerate}
  \item[(S1)] Retain each half-edge independently with probability $\pi_n$. 
  \item[(S2)] Create a uniform perfect matching between the retained half-edges and obtain $\mathcal{G}_n'(\pi_n)$ by creating edges corresponding to any two pair of matched half-edges.
  \end{enumerate}
 \end{algo}
 In (S1), if the total number of retained half-edges is odd, then add an extra half-edge to vertex~1. However, this possible addition of 1 extra half-edge will be ignored since it does not make any difference in the asymptotic computations.
Notice that $a_i$, the number of half-edges attached to $i$ that are retained by Algorithm~\ref{algo:perc-for-blob}~(S1), is distributed as $\mathrm{Bin}(f_i(t_n),\pi_n)$, independently for each $i$. 
Thus the number of half-edges in the graph $\mathcal{G}_n'(\pi_n)$ is distributed as a $\mathrm{Bin}(s_1(t_n),\pi_n)$ random variable.
We claim that there exists $\varepsilon_n = o(n^{-\eta})$ and a coupling such that, with high probability 
\begin{equation}\label{eq:blob-graphs}
 \mathcal{G}_n'(\pi_n-\varepsilon_n)\subset \mathcal{G}_n'\subset \mathcal{G}_n'(\pi_n+\varepsilon_n).
\end{equation}
The proof follows using an identical argument as \cite[Proposition 24]{DHLS16} using the estimate~\eqref{eq:estimate-he-modi-supst} and standard concentration inequalities for binomial random variables. 
We skip the proof here. 
We now continue to analyze $\mathcal{G}_n'(\pi_n)$, keeping in mind that the relation~\eqref{eq:blob-graphs} allows us make conclusions for $\mathcal{G}_n'$.
%
%
%
To analyze the component sizes and the surplus edges of the components of $\mathcal{G}'(\pi_n)$ we first need some regularity conditions on $\bld{a}$, the degree sequence of $\mathcal{G}_n'(\pi_n)$, as summarized in the following lemma:
\begin{lemma}\label{lem:a-asymp}
For any $0<\delta<\eta$, as $n\to\infty$,
 $n^{-\alpha}a_i\pto\frac{\theta_i}{\nu}$,  $\frac{a_i}{\sum_i a_i}n^{\rho-\delta} \pto \frac{\theta_i}{\mu\nu},$ $\nu_n(\bld{a}) = \frac{\sum_ia_i(a_i-1)}{\sum_ia_i} = 1+ \lambda n^{-\eta+\delta} + \oP(n^{-\eta+\delta}),$ and for any $\varepsilon>0$,
 \begin{equation}\label{eq:a-third-moment}
 \lim_{K\to\infty}\limsup_{n\to\infty}\PR\bigg(\sum_{i>K}a_i^3>\varepsilon n^{3\alpha}\bigg) = 0.  
 \end{equation} 
\end{lemma}  
\begin{proof} 
Using Theorem~\ref{th:open-he-entrance} and the fact that $a_i\sim \mathrm{Bin}(f_i(t_n),\pi_n)$, one gets $n^{-\alpha}a_i = (1+\oP(1))\frac{\theta_i}{\nu}$.
Moreover, $\sum_ia_i\sim\mathrm{Bin}(\sum_if_i(t_n),\pi_n)$
and $\sum_ia_i = (1+\oP(1))\pi_n \sum_if_i(t_n)$ yield the required asymptotics for  $a_i/\sum_ia_i$.
Next note that if $X\sim \mathrm{Bin} (r,\pi)$, then $\mathrm{Var}(X(X-1)) = 2r(r-1)\pi^2(1-\pi) (1+(2r-3) \pi)$. 
Thus,
\begin{eq}
    &\mathrm{Var} \bigg(\sum_{i} a_i(a_i-1) \Big| (f_i(t_n))_{i\geq 1}\bigg) =  \sum_{i}\mathrm{Var}\big( a_i(a_i-1) \big| (f_i(t_n))_{i\geq 1}\big) \\
    &= \OP\bigg( \pi_n^2 \sum_i f_i^2(t_n) + \pi_n^3 \sum_i f_i^3(t_n) \bigg) = \OP(n^{1-\delta} + n^{3\alpha}) = \OP(n^{3\alpha}).
\end{eq}
Therefore, for any $\varepsilon>0$,
 \begin{equation}\label{eqn::prob:ineq:third}
 \begin{split}
 &\mathbbm{P} \bigg( \Big|\sum_{i} a_i(a_i-1)- \pi_n^2 \sum_{i} f_i(t_n)(f_i(t_n)-1) \Big| >\varepsilon n^{1-\delta} \Big| (f_i(t_n))_{i\geq 1}\bigg)  = \OP(n^{3\alpha - 2+2\delta}),
 \end{split}
 \end{equation}
which is $\oP(1)$ since $\delta<\eta$ and  $3\alpha-2+2\eta = 3\alpha-2 +2 - 4\alpha <0$.
We conclude that
\begin{equation}\label{eq:a-i-1}
 \sum_{i}a_i(a_i-1) = (1+\oP(1)) \pi_n^2 \sum_i f_i(t_n)(f_i(t_n)-1),
\end{equation}
and the required asymptotics for $\nu_n(\bld{a})$ follows.
To see \eqref{eq:a-third-moment}, note that
 $\E\big[\sum_{i>K}a_i(a_i-1)(a_i-2)|(f_i(t_n))_{i\geq 1}\big] = \pi_n^3\sum_{i>K}f_i(t_n)^3,$
 and the proof follows again by using the condition on $s_3^\omega$ in Theorem~\ref{th:open-he-entrance}.
\end{proof}

\noindent 
Consider the  exploration of the graph $\mathcal{G}_n'(\pi_n)$ via Algorithm~\ref{algo:explor-barely-sub}, but now the first vertex is chosen proportional to its degree.
Define the exploration process by $\bld{S}_n$ similarly as the process $\bld{S}_n^j(l)$ in Section~\ref{sec:barely-subcritical-mass}.
 Call a vertex \emph{discovered} if it is either active or killed. Let $\mathscr{V}_l$ denote the set of vertices discovered up to time $l$ and $\mathcal{I}_i^n(l):=\ind{i\in\mathscr{V}_l}$. Note that 
   \begin{equation}\label{expl-process-CM}
    S_n(l)= \sum_{i} a_i \mathcal{I}_i^n(l)-2l=\sum_{i} a_i \left( \mathcal{I}_i^n(l)-\frac{a_i}{\ell^a_n}l\right)+\left( \nu_n(\bld{a})-1\right)l,
   \end{equation}where $\ell_n^a = \sum_i a_i$. Consider the re-scaled version $\bar{\mathbf{S}}_n$ defined as $\bar{S}_n(t)=n^{-\alpha}S_n(\lfloor tn^{\rho-\delta} \rfloor)$. 
   Define the limiting process
   \begin{equation}\label{eq:lim-blob-process}
   S(t) = \sum_{i=1}^\infty\frac{\theta_i}{\nu}\bigg(\ind{\mathrm{Exp}(\theta_i/(\mu\nu))\leq t}-\frac{\theta_i}{\mu\nu}t\bigg) + \lambda t.
   \end{equation}
   \begin{proposition} \label{thm::convegence::exploration-process-blob} As $n\to\infty$, $\bar{\mathbf{S}}_n \dto \mathbf{S}$ with respect to the Skorohod $J_1$ topology.
\end{proposition}The proof of Proposition~\ref{thm::convegence::exploration-process-blob} can be carried out using similar ideas as \cite[Theorem 8]{DHLS16}. A sketch of the proof is given in Appendix~\ref{sec:appendix-perc-blob}. 
 The excursion lengths of the exploration process give the number of edges in the explored components. 
 Now, at each step $l$, the probability of discovering a surplus edge, conditioned on the past, is approximately the proportion of half-edges that are active. 
 Note that the number of active half-edges is the reflected version of $\mathbf{S}_n$ given by $\mathrm{refl}(S_n(t)) = S_n(t) - \inf_{u\leq t} S_n(u)$. 
 Thus, conditional on $(S_n(l))_{l\leq tn^{\rho-\delta}}$, the rate at which a  surplus edge appears at time $tn^{\rho-\delta}$ is approximately  
$n^{\rho-\delta}\frac{\mathrm{refl}(S_n(tn^{\rho-\delta}))}{\sum_{i}a_i} = \frac{1}{\mu}\refl{\bar{S}_n(t)}(1+\oP(1)).$
 Therefore, Proposition~\ref{thm::convegence::exploration-process-blob} implies that for each $K\geq 1$, there exists components $C_1,\dots,C_K \subset \mathcal{G}_n'(\pi_n)$ such that
 \begin{equation}\label{eq:blob-perc-limit}
  \big(n^{-\rho +\delta}|C_i|, \mathrm{SP} (C_i)\big)_{i\in [K]} \dto \big(\xi_i, \mathscr{N}_i \big)_{i\in[K]},
 \end{equation} where $\xi_i$ and $\mathscr{N}_i$ are defined in Proposition~\ref{prop:comp-size}.
 We refer to \cite[Section 5.4]{DHLS16} for more details regarding the proof of \eqref{eq:blob-perc-limit}.
Here we have also used the fact that the ordered excursion lengths of the process $(S(t))_{t\geq 0}$, defined in \eqref{eq:lim-blob-process},
 are identically distributed as the ordered excursion lengths of $(S(t)/\mu)_{t\geq 0}$.
We can now combine \eqref{eq:blob-graphs} and \eqref{eq:blob-perc-limit} to obtain the asymptotics for the number of blobs in the largest connected components and $\mathrm{SP}'(\cdot)$.
Denote $\mathscr{B} (\mathscr{C}) = |\mathfrak{B}(\mathscr{C})|$ for a component $\mathscr{C}\subset \mathcal{G}_n(t_c(\lambda))$. 
\begin{lemma}\label{thm:superstrcut-comp-surp-original}
 For $K\geq 1$, there exist components $\mathscr{C}^1,\dots,\mathscr{C}^K \subset \mathcal{G}_n(t_c(\lambda))$ such that the following convergence holds:
  $(n^{-\rho +\delta}\mathscr{B}(\mathscr{C}^i), \mathrm{SP}' (\mathscr{C}^i))_{i\in [K]} \dto (\xi_i, \mathscr{N}_i  )_{i\in[K]}.
 $
\end{lemma}
\begin{lemma}\label{thm:comp-blob-same-whp}
 For any $K\geq 1$, $\mathscr{C}^i = \mathscr{C}_{\sss (i)}(\lambda)$, $\forall i\in [K]$ with high probability.
\end{lemma}
\begin{proof}
 Notice that, $\sum_{j\leq i}|\mathscr{C}^j|\leq \sum_{j\leq i}|\mathscr{C}_{\sss (j)} (\lambda)|$ for all $i\in [K]$, almost surely. 
 Thus, it is enough to prove that $|\mathscr{C}^i|$ and  $|\mathscr{C}_{\sss (i)} (\lambda)|$ involve the same re-scaling factor and have the same scaling limit. 
We again make use of the inclusions in graphs in \eqref{eq:blob-graphs}.
Algorithm~\ref{algo:explor-barely-sub} explores the components of $\mathcal{G}_n'(\pi_n)$ in a size-biased manner with the sizes being $(a_i)_{i\geq 1}$. 
An application of Lemma~\ref{lem:size-biased} with $y_i = \Cli(t_n)$ yields that, for any $t>0$, uniformly for $l\leq tn^{\rho-\delta}$,
 \begin{equation}\label{blob-vs-comp1}
  \sum_i|\Cli(t_n)| \mathcal{I}_i^n(l) = \sum_i |\Cli(t_n)| \frac{a_i}{\sum_i a_i}l +\oP(n^{\rho}).
 \end{equation} 
Since $a_i\sim\mathrm{Bin}(f_i(t_n),\pi_n)$, we can apply concentration inequalities like \cite[Corollary 2.27]{JLR00} and use the asymptotics from Theorem~\ref{th:open-he-entrance} to conclude that
 \begin{equation}\label{blob-vs-comp2}
 \begin{split}
  n^{-\delta}\frac{\sum_i a_i|\Cli(t_n)|}{\sum_ia_i} 
  & = \frac{\frac{\mu(\nu-1)}{\nu^2}}{\frac{\mu(\nu-1)}{\nu}}(1+\oP(1))=\frac{1}{\nu}(1+\oP(1)).
  \end{split}
 \end{equation}
 Thus, \eqref{blob-vs-comp1} and \eqref{blob-vs-comp2}, together with \eqref{eq:blob-graphs}, imply that
$  \frac{\nu |\mathscr{C}^i|}{n^{\delta}\mathscr{B}(\mathscr{C}^i)} \xrightarrow{\sss \PR} 1,$ and it follows from Lemma~\ref{thm:superstrcut-comp-surp-original} and Lemma~\ref{lem:rescale} that 
$  (n^{-\rho} |\mathscr{C}^i|)_{i\in [K]} \xrightarrow{\sss d} (\frac{1}{\nu}\xi_i )_{i\in[K]}. $
\end{proof}

\begin{proof}[Proof of Proposition~\ref{thm:comp-functional-original}]
We are now finally in the position to prove Proposition~\ref{thm:comp-functional-original}.
Using Lemmas~\ref{thm:superstrcut-comp-surp-original},~\ref{thm:comp-blob-same-whp}, and Proposition~\ref{thm:comp-functionals-original} together with \eqref{eq:blob-graphs}, we directly conclude Part~(a) from \eqref{eq:surp-sam-dist-limit}.
For Part~(b),  we can follow the same arguments as~\eqref{blob-vs-comp1} to conclude that, uniformly for $l\leq tn^{\rho-\delta}$,
 \begin{equation}\label{blob-vs-comp3}
  \sum_if_i(t_n) \mathcal{I}_i^n(l) = \sum_i f_i(t_n) \frac{a_i}{\sum_i a_i}l +\oP(n^{\rho}),
 \end{equation} where 
$  n^{-\delta}\frac{\sum_i a_if_i(t_n)}{\sum_ia_i} = \frac{\nu-1}{\nu} (1+\oP(1)).$
 Now, \eqref{blob-vs-comp1} and \eqref{blob-vs-comp3} together with \eqref{eq:blob-graphs} prove Part~(b).
\end{proof}

\subsection{Coupling with the multiplicative coalescent}
\label{sec:coupling-mul-coal}
Recall the definitions of $t_c(\lambda)$ and $t_n$ from \eqref{eq:def-crit-time} and \eqref{eq:def-subcrit-time}.  
 Now, let us investigate the dynamics of $\bld{f}(t) $ starting from time $t_n$. Notice that, in the time interval $ [t_n,t_c(\lambda)]$, components with masses $f_i(t)$ and $f_j(t)$ merge at rate 
 \begin{equation}
  f_i(t)\frac{f_j(t)}{s_1(t)-1}+f_j(t)\frac{f_i(t)}{s_1(t)-1}=\frac{2f_i(t)f_j(t)}{s_1(t)-1}\approx \frac{2\nu f_i(t)f_j(t)}{\mu(\nu-1)n}, 
 \end{equation}and create a component with $f_i(t)+f_j(t)-2$ open half-edges. 
 Thus $\bld{f}(t)$ does not exactly evolve as a multiplicative coalescent, but it is close. 
 We define an exact multiplicative coalescent that approximates the above process:
 \begin{algo}[Modified process] \label{algo:modified-MC} \normalfont Conditionally on $\mathcal{G}_n(t_n)$, associate a rate $2/(s_1(t_n)-1)$ Poisson process $\mathcal{P}(e,f)$ to each pair of unpaired-half-edges $(e,f)$. 
An edge $(e,f)$ is created between the vertices incident to $e$ and $f$ at the instance when $\mathcal{P}(e,f)$ rings. 
However, the half-edges are not discarded after the pairing.
At time $t>t_n$, the obtained \emph{modified} graph $\bar{\mathcal{G}}_n(t)$ consists of the edges of $\mathcal{G}_n(t_n)$, and the edges created by this algorithm between times $t_n$ and $t$.
 \end{algo}
\begin{proposition}\label{prop-coupling-order}
There exists a coupling such that $\mathcal{G}_n(t)\subset\bar{\mathcal{G}}_n(t)$ for all $t>t_n$ with probability one.
\end{proposition}
\begin{proof}
Recall the construction of $\mathcal{G}_n(t)$ from Algorithm~\ref{algo:dyn-cons}. 
We modify (S1) as follows: whenever two half-edges are paired, we do not kill the corresponding half-edges and do not discard the associated exponential clocks. 
Instead we reset the corresponding exponential clocks. 
The graphs generated by this modification of Algorithm~\ref{algo:dyn-cons} have the same distribution as $\bar{\mathcal{G}}_n(t)$, conditionally on $\mathcal{G}_n(t_n)$.
Moreover, the above also gives a natural coupling such that $\mathcal{G}_n(t)\subset\bar{\mathcal{G}}_n(t)$, by viewing the event times of Algorithm~\ref{algo:dyn-cons} as a thinning of the event times of the modified process.
\end{proof}

Henceforth, we will always assume that we are working on a probability space such that Proposition~\ref{prop-coupling-order} holds.
 Recall that the connected components at time $t_n$, $(\Cli(t_n))_{i\geq 1}$, are regarded as blobs, and we can also view  $\bar{\mathcal{G}}_n(t)$ as a super-graph with the superstructure being determined by the edges appearing after time $t_n$ in Algorithm~\ref{algo:modified-MC}.  
  Let us denote the ordered connected components of $\bar{\mathcal{G}}_n(t)$ by $(\bar{\mathscr{C}}_{\sss (i)}(t))_{i\geq 1}$. 
  Define
 $$ \bar{\mathcal{F}}_{i}(t) = \sum_{b\in \bl(\bar{\mathscr{C}}_{\sss (i)}(t))}f_{b}(t_n),$$ where $\bl(\mathscr{C}):=\{b: \Clb(t_n)\subset\mathscr{C}, \Clb(t_n) \neq \varnothing\}$.
 The $\bar{\mathcal{F}}$-value is regarded as the mass of component $\bar{\mathscr{C}}_{\sss (i)}(t)$ at time $t$.
Note that for the modified process in Algorithm~\ref{algo:modified-MC}, conditionally on $\mathcal{G}_n(t_n)$, at time $t\in[t_n,t_c(\lambda)]$, $\bar{\mathscr{C}}_{\sss (i)}(t)$ and $\bar{\mathscr{C}}_{\sss (j)}(t)$ merge at exact rate $2 \bar{\mathcal{F}}_i(t)\bar{\mathcal{F}}_j(t)/(s_1(t_n)-1)$ and the new component has mass $\bar{\mathcal{F}}_i(t)+\bar{\mathcal{F}}_j(t)$. 
Thus, the vector of masses $(\bar{\mathcal{F}}_i(t))_{i\geq 1}$ merge as an exact multiplicative coalescent.

\subsection{Properties of the modified process}
\label{sec:modified-graph}
Notice that, conditionally on $\mathcal{G}_n(t_n)$, blobs $b_i$ and $b_j$ are connected in $\bar{\mathcal{G}}_n(t_c(\lambda))$ with probability
\begin{equation}\label{eq:pij-value}
p_{ij} =1- \exp\Big(-f_{b_i}(t_n) f_{b_j}(t_n)\Big[\frac{1}{n^{1+\delta}}\frac{\nu^2 }{\mu(\nu-1)^2} + \frac{1}{n^{1+\eta}}\frac{\nu^2}{\mu(\nu-1)^2}\lambda\Big] (1+\oP(1))\Big),
\end{equation}where the $\oP(\cdot)$ term appearing above is uniform in $i,j$.
Thus, using Theorem~\ref{th:open-he-entrance}, \eqref{eq:pij-value} is of the form $1-\e^{-qx_ix_j(1+\oP(1))}$ with
\begin{equation}\label{eq:parameters-inhom}
 x_i^n = n^{-\rho}f_{b_i}(t_n), \quad q = \frac{1}{\sigma_2(\bld{x}^n)} + \frac{\nu^2}{\mu(\nu-1)^2}\lambda,
\end{equation}where $\sigma_r(\bld{x}^n) = \sum_{i\geq 1} (x_i^n)^r$. 
By Theorem~\ref{thm:susceptibility}, the sequence $\bld{x}^n$ satisfies the entrance boundary conditions of \cite{AL98}, i.e.,
\begin{equation}
\begin{split}
 \frac{\sigma_3(\bld{x}^n)}{(\sigma_2(\bld{x}^n))^3} \pto \frac{1}{\mu^3(\nu-1)^3}\sum_{i=1}^\infty \theta_i^3,\quad \frac{x_i^n}{\sigma_2(\bld{x}^n)}\pto \frac{1}{\mu(\nu-1)}\theta_i, \quad \sigma_2(\bld{x}^n)\pto 0.
 \end{split}
\end{equation}
To simplify the notation, we write $\bar{\mathcal{F}}_i(\lambda)$ for  $\bar{\mathcal{F}}_i(t_c(\lambda))$ and $\bar{\mathscr{C}}_{\sss (i)}(\lambda)$ for $\bar{\mathscr{C}}_{\sss (i)}(t_c(\lambda))$.
The following result is a consequence of \cite[Proposition~7]{AL98}, \cite[Lemma 5.3]{BHS15}, and Lemma~\ref{lem:rescale}:
\begin{proposition}\label{thm:modified-open-he-limit} As $n\to\infty$, $\big(n^{-\rho}\bar{\mathcal{F}}_i(\lambda)\big)_{i\geq 1} \dto \frac{\nu-1}{\nu}\bld{\xi}$
with respect to the $\ell^2_{\shortarrow}$ topology, where $\bld{\xi}$ is defined in {\rm Proposition~\ref{prop:comp-size}}.
\end{proposition}


We next relate $(\bar{\mathcal{F}}_i(\lambda))_{i\geq 1}$ to $(\bar{\mathscr{C}}_{\sss (i)}(\lambda))_{i\geq 1}$, for each fixed $i$.
\begin{proposition}\label{thm:mod-comp-openhe}
 As $n\to\infty$, 
  $\bar{\mathcal{F}}_i(\lambda) = (\nu-1)|\bar{\mathscr{C}}_{\sss (i)}(\lambda)| + \oP(n^{\rho}).$
 Consequently, \linebreak
  $\big(n^{-\rho}|\bar{\mathscr{C}}_{\sss (i)}(\lambda)|\big)_{i\geq 1} \dto \frac{1}{\nu}\bld{\xi}$
with respect to the product topology.
\end{proposition}
\noindent We will need the following lemma, the proof of which is same as \cite[Lemma 8.2]{BSW14}:
\begin{lemma}[{\cite[Lemma 8.2]{BSW14}}]\label{lem:size-biased} Consider two ordered weight sequences $\bld{x} = (x_i)_{i\in [m]} $ and $\bld{y} = (y_i)_{i\in [m]}$. 
Consider the size-biased reordering $(v(1),v(2),\dots)$ of  $[m]$ with respect to the weights $\bld{x}$ and let $V(i):= \{v(1),\dots,v(i)\}$. Denote $m_{rs} = \sum_i x_i^{r}y_i^{s}$, define $\mathrm{M}_n = m_{11}/m_{10}$ and assume that $\mathrm{M}_n>0$ for each $n$. 
Suppose that the following conditions hold:
\begin{equation}\label{eq:size-biased-conditions}
 \frac{lm_{21}}{m_{10}m_{11}} \to 0, \quad \frac{m_{12}m_{10}}{lm_{11}^2}\to 0,\quad \frac{lm_{20}}{m_{10}^2}\to 0, \quad \text{ as }n\to\infty.
\end{equation}
Then, as $n\to\infty$,
 $\sup_{k\leq l}\big|\frac{1}{l\mathrm{M}_n}\sum_i y_i \ind{i\in V(k)}-\frac{k}{l}\big| \pto 0.$
\end{lemma}
\begin{proof}[Proof of Proposition~\ref{thm:mod-comp-openhe}]
We only prove the asymptotic relation of $\bar{\mathcal{F}}_1(\lambda)$ and $|\mathscr{C}_{\sss (1)}(\lambda)|$.
Consider the breadth-first exploration of the \emph{supestructure} of graph $\bar{\mathcal{G}}_n(t_c(\lambda))$ (which is also a rank-one inhomogeneous random graph) using the Aldous-Limic construction from \cite[Section 2.3]{AL98}. 
Notice that the vertices are explored in a size-biased manner with the sizes being $\bld{x} = (x_i)_{i\geq 1}$, where $x_i =x_i^n= n^{-\rho} f_{b_i}(t_n)$ are as defined in \eqref{eq:parameters-inhom}.  
Let $v(i)$ be the $i$-th vertex explored.
Further, let $\bar{\mathscr{C}}_{\sss (i)}^{\sss \mathrm{st}}(\lambda)$ denote the component $\bar{\mathscr{C}}_{\sss (i)}(\lambda)$, where the blobs have been shrunk to single vertices.
Then, from \cite{AL98}, one has the following:
\begin{enumerate}[(i)]
\item there exists random variables $m_L,m_R$ such that $\bar{\mathscr{C}}_{\sss (i)}^{\sss \mathrm{st}}(\lambda)$ is explored between $m_L+1$ and $m_R$;
\item $\sum_{i\leq m_R}x_{\sss v(i)}$ is tight;
\item $\sum_{i=m_L+1}^{m_R}x_{\sss v(i)} \dto \gamma$, where $\gamma$ is some non-degenerate, positive random variable. 
\end{enumerate}  
Let $y_i = n^{-\rho}|\Clbi(t_n)|$. Using  Theorem~\ref{th:open-he-entrance},  Remark~\ref{rem:3rd-suscep} and Remark~\ref{rem:entrance-comp-size}, it follows that
 $\sum_i x_i^{r} y_i^{s} = \OP(n^{3\delta-3\eta});$ for $r+s=3$, $\sum_ix_i = \OP(n^{1-\rho}),$ and 
 $\sum_{i} x_i^{r} y_i^{s} = \OP(n^{-2\rho+1+\delta});$ for  $r+s = 2$.
Below, we show that 
\begin{equation}\label{eqn:ratio-comp-he}
 \frac{\sum_{i = m_L+1}^{m_R}y_{\sss v(i)}}{\sum_{i = m_L+1}^{m_R}x_{\sss v(i)}}\times \frac{\sum_ix_i^2}{\sum_i x_i y_i} \pto 1.
\end{equation} The proof of Proposition~\ref{thm:mod-comp-openhe} follows from \eqref{eqn:ratio-comp-he} by using Theorem~\ref{th:open-he-entrance} and observing that
$ \frac{\sum_{i}x_i^2}{\sum_ix_iy_i} = \frac{s_2^\omega(t_n)}{s_{pr}^\omega(t_n)}\pto \nu-1$.
To prove \eqref{eqn:ratio-comp-he}, we will now apply Lemma~\ref{lem:size-biased}.  Denote $m_0 = \sum_i x_i/\sum_ix_i^2$ and consider $l = 2Tm_0$ for some fixed $T>0$.
 Using Theorem~\ref{th:open-he-entrance}, an application of Lemma~\ref{lem:size-biased} yields
\begin{align}
 \sup_{k\leq 2Tm_0} \bigg|\sum_{i=1}^k x_{\sss v(i)} - \frac{k}{m_0}\bigg|\pto 0.
\end{align}
Now, for any $\varepsilon > 0$, $T>0$ can be chosen large enough such that $\sum_{i=1}^{m_R}x_{\sss v(i)}>T$ has probability at most $\varepsilon$ and on the event 
$\big\{\sup_{k\leq 2Tm_0} \big|\sum_{i=1}^k x_{\sss v(i)} - \frac{k}{m_0}\big|\leq \varepsilon \big\}\cap \big\{\sum_{i=1}^{m_R}x_{\sss v(i)}\leq T\big\},$ 
one has $m_L<m_R<2Tm_0$. Thus, it follows that 
\begin{equation}\label{weight-total-larg-comp-1}
 \bigg|\sum_{i=m_L+1}^{m_R}x_{\sss v(i)} - \frac{m_R-m_L}{m_0}\bigg| \pto 0.
\end{equation}An identical argument as above shows that 
\begin{equation}\label{weight-total-large-comp-2}
 \bigg|\sum_{i=m_L+1}^{m_R}y_{\sss v(i)} - \frac{m_R-m_L}{m_0'}\bigg| \pto 0,
\end{equation}where $m_0'=\sum_{i}x_i/\sum_ix_iy_i$. The proof of \eqref{eqn:ratio-comp-he} now follows from \eqref{weight-total-larg-comp-1} and \eqref{weight-total-large-comp-2}. 
The asymptotic distribution for $(n^{-\rho}|\bar{\mathscr{C}}_{\sss (i)}(\lambda)|)_{i\geq 1}$ can be obtained using Proposition~\ref{thm:modified-open-he-limit}.
\end{proof}
Recall that $\omega_i(t_n)$ denotes the number of open-half edges attached to vertex $i$ in the graph $\mathcal{G}_n(t_n)$. 
We now equip $\bar{\mathscr{C}}_{\sss (i)}(\lambda)$ with the probability measure $\mu_{\sss \mathrm{fr}}^i$ given by $\mu_{\sss \mathrm{fr}}^i(A) = \sum_{k\in A}\omega_k(t_n)/\mathcal{F}_i(\lambda)$ for $A\subset \bar{\mathscr{C}}_{\sss (i)}(\lambda)$, and denote the corresponding measured metric space by $\bar{\mathscr{C}}^{\sss \mathrm{fr}}_{\sss (i)}(\lambda)$.
\begin{theorem}\label{thm:mspace-limit-modified} Under {\rm Assumption~\ref{assumption1}}, as $n\to\infty$,
\begin{equation}\label{eq:limit-component-free}
 \big(n^{-\eta}\bar{\mathscr{C}}_{\sss (i)}^{\sss \mathrm{fr}}(\lambda)\big)_{i\geq 1} \dto (M_{i})_{i\geq 1},
\end{equation}with respect to the $\mathscr{S}_*^\N$ topology, where $M_i$ is defined in Section~\ref{sec:limit-component}. 
\end{theorem}
\begin{proof}
We just consider the metric space limit of $\bar{\mathscr{C}}_{\sss (i)}^{\sss \mathrm{fr}}(\lambda)$ for each fixed $i\geq 1$ and the joint convergence in \eqref{eq:limit-component-free} follows using the joint convergence of different functionals used throughout the proof. 
Recall the notation $\bl(\mathscr{C}):=\{b: \Clb(t_n)\subset\mathscr{C}, \Clb(t_n) \neq \varnothing\}$ for a component $\mathscr{C}$.
Now, $\bar{\mathscr{C}}^{\sss \mathrm{fr}}_{\sss (i)}(\lambda)$ can be seen as a super-graph as defined in Section~\ref{defn:super-graph} with \begin{enumeratei}
\item the collection of blobs $\{\Clb(t_n):b\in \bl(\bar{\mathscr{C}}_{\sss (i)}(\lambda))\}$ and within blob measure $\mu_b$ given by $\mu_b(A) = \sum_{k\in A}\omega_k(t_n)/f_b(t_n)$, $A\subset \Clb(t_n)$,  $b\in  \bl(\bar{\mathscr{C}}_{\sss (i)}(\lambda))$;
\item the supersturcture consisting of the edges appearing during $[t_n,t_c(\lambda)]$ in Algorithm~\ref{algo:modified-MC} and weight sequence $(f_b(t_n)/\bar{\mathcal{F}}_i(\lambda):b\in  \bl(\bar{\mathscr{C}}_{\sss (i)}(\lambda)))$.
\end{enumeratei}  
Recall $x_i =x_i^n= n^{-\rho} f_{b_i}(t_n)$ defined in \eqref{eq:parameters-inhom}. 
Let $\dst(\cdot,\cdot)$ denote the graph distance on $\bar{\mathscr{C}}_{\sss (i)}(\lambda)$ and define 
 \begin{equation}
  u_i = \sum_{v_1,v_2\in \Clbi(t_n)} \frac{\omega_{v_1}\omega_{v_2}}{f_{b_i}^2(t_n)}\dst(v_1,v_2), \quad B_n^{\sss (i)}= \frac{\sum_{b_j \in \bl(\bar{\mathscr{C}}_{\sss (i)}(\lambda))}x_ju_j}{\sum_{b_j\in \bl(\bar{\mathscr{C}}_{\sss (i)}(\lambda))}x_j}.
 \end{equation}
 Here $u_i$ gives the average distance within blob $\Clbi(t_n)$.
Using Lemma~\ref{lem:size-biased}, we will show  
   \begin{equation}\label{eq:measure-limit}
    B_n^{\sss (i)} \times \frac{\sum_i x_i^2}{\sum_i x_i^2u_i}\pto 1.
   \end{equation}
The argument is same as the proof of \eqref{eqn:ratio-comp-he}. We only have to ensure that \eqref{eq:size-biased-conditions} holds with $y_i = x_iu_i$. 
   Thus, we need to show that 
   \begin{equation}\label{cond-check-dist-blob}
    \frac{n^{\rho-\delta}\sum_ix_i^3u_i}{\sum_ix_i\sum_ix_i^2u_i}\pto 0,\quad \text{and}\quad \frac{\sum_i x_i^3 u_i^2\sum_ix_i}{n^{\rho-\delta}\big(\sum_ix_i^2u_i\big)^2}\pto 0.
   \end{equation}First, notice that, by Lemma~\ref{lem:total-open-he} and Theorem~\ref{th:open-he-entrance},
   \begin{equation} \label{order:cn}
   \begin{split}
    \mathrm{M}_n &= \frac{\sum_i x_i^2u_i}{\sum_i x_i} = (1+\oP(1))\frac{\nu n^{-1+\rho}}{\mu(\nu-1)} n^{-2\rho} \sum_b f_b^2(t_n) \sum_{v_1,v_2\in \Clb(t_n)} \frac{\omega_{v_1}\omega_{v_2}}{f_b^2(t_n)}\dst(v_1,j) \\
    &=  \frac{\nu n^{-1+\rho}}{\mu(\nu-1)} n^{1-2\rho} \mathcal{D}_n^\omega  = n^{2\delta-\rho} \frac{\nu-1}{\nu^2}(1+\oP(1)). 
    \end{split} 
   \end{equation}
   Also, recall from Theorem~\ref{thm:diam-max} that $u_{\max}= \max_{b}u_b =\OP(n^{\delta}\log(n))$.  
   Now, 
   \begin{align*}
    &\frac{n^{\rho-\delta}\sum_ix_i^3u_i}{\sum_i x_i \sum_ix_i^2u_i} \leq \frac{n^{\rho-\delta}u_{\max}\sum_{i}x_i^3}{\sum_{i} x_i \sum_ix_i^2u_i}  \\
    &\qquad = \OP\bigg(\frac{n^{\rho-\delta}n^{\delta}\log^2(n)n^{-3\rho}n^{3\alpha+3\delta}}{n^{1-\rho}n^{2\delta-\rho}n^{1-\rho}}\bigg) = \OP(n^{\delta-\eta}\log^2(n))=\oP(1),\\
      &\frac{\sum_i x_i^3 u_i^2\sum_ix_i}{n^{\rho-\delta}\big(\sum_ix_i^2u_i\big)^2}\leq \frac{x_{\max}u_{\max}\sum_ix_i}{n^{\rho-\delta}\sum_ix_i^2u_i} \\
      &\qquad = \OP\bigg(\frac{n^{-\rho}n^{\alpha+\delta}n^{\delta}\log^2(n)}{n^{\rho-\delta}n^{2\delta-\rho}}\bigg) = \OP(n^{\delta-\eta}\log^2(n))=\oP(1),
   \end{align*}and \eqref{cond-check-dist-blob} follows, and hence the proof of \eqref{eq:measure-limit} also follows.    
   Recall that the superstructure of $\bar{\mathcal{G}}_n(t_c(\lambda))$ has the same distribution as a Norros-Reittu random graph $\mathrm{NR}_n(\bld{x}, q)$ with the parameters given by \eqref{eq:parameters-inhom}. 
   Thus, using Proposition~\ref{prop:generate-nr-given-partition}, we now aim to use Theorem~\ref{thm:univesalty} on $\mathscr{C}_{\sss (i)}^{\sss \mathrm{fr}}(\lambda)$ with the blobs being $(\Cli(t_n))_{i\geq 1}$, and $\mathbf{p}_n^{\sss (i)}$, $a_n^{\sss (i)}$ given by \eqref{eq:p-n-a-NR}. 
 Define $\Upsilon_n^{\sss (i)} = \big(p_b/\sigma(\mathbf{p}_n^{\sss (i)}):b\in \bl(\bar{\mathscr{C}}_{\sss (i)}(\lambda))\big)$. Let $\mathcal{N}(\R_+)$ denote the space of all counting measures equipped with the vague topology and let $\mathbbm{S} := \R_+^3\times \mathcal{N}(\R_+)$ denote the product space. Define 
   \begin{equation}
    \mathscr{P}_n = \Big(a_n^{\sss (i)}\sigma(\mathbf{p}_n^{\sss (i)}), \sum_{b\in \bl(\bar{\mathscr{C}}_{\sss (i)}(\lambda))}x_b, \frac{1}{\sigma_2^2(\bld{x}^n)}\sum_{b\in \bl(\bar{\mathscr{C}}_{\sss (i)}(\lambda))}x_b^2, \Upsilon_n^{\sss (i)}\Big)_{i\geq 1},
   \end{equation}  viewed as an element of $\mathbbm{S}^\N$.  
  Recall the definition of $\xi_i^*$ and $\Xi_i^*$ from Section~\ref{sec:limit-component}.
   Define
   \begin{equation}
    \mathscr{P}^\infty = \bigg( \frac{\xi_i^*}{\mu(\nu-1)}\bigg(\sum_{v\in \Xi_i^*}\theta_v^2\bigg)^{1/2},\ \xi_i^*,\ \frac{1}{\mu^2(\nu-1)^2}\sum_{v\in \Xi_i^*}\theta_v^2,\ \bigg(\frac{\theta_j}{\sum_{v\in \Xi_i^*}\theta_v^2}: j\in \Xi_i^*\bigg) \bigg)_{i\geq 1}
   \end{equation}
   The following is a consequence of \cite[Proposition 5.1, Lemma 5.4]{BHS15}:
   \begin{equation}\label{con-com-parameters}
    \sigma(\mathbf{p}_n^{\sss (i)})\pto 0,\quad  \text{and} \quad \mathscr{P}_n\dto \mathscr{P}_{\infty} \text{ on } \mathbbm{S}^\N.
   \end{equation}  Without loss of generality, we assume that the convergence in \eqref{con-com-parameters} holds almost surely. Now, using \eqref{eq:measure-limit}, it follows that
   \begin{align*}
    &\frac{\sigma(\mathbf{p}_n^{\sss (i)})}{1+B_n^{\sss (i)}} = \frac{\sigma_2(\bld{x}^n)\big( \sum_{v\in \Xi_i^*}\theta_v^2\big)^{1/2}}{\mu(\nu-1)\xi_i^*}\times\frac{\sum_i x_i^2}{\sum_i x_i^2u_i}(1+o(1))\\
  &= \frac{\sigma_2^2(\bld{x}^n)\big( \sum_{v\in \Xi_i^*}\theta_v^2\big)^{1/2}}{\mu(\nu-1)\xi_i^*\sum_i x_i^2u_i} (1+o(1))= n^{-\eta}\frac{\nu-1}{\nu}  \frac{1}{\xi_i^*}\bigg( \sum_{v\in \Xi_i^*}\theta_v^2\bigg)^{1/2}(1+o(1)),
   \end{align*}    
    where the last step follows from Theorem~\ref{th:open-he-entrance}, \eqref{order:cn} and \eqref{con-com-parameters}. The proof of Theorem~\ref{thm:mspace-limit-modified} is now complete using Theorem~\ref{thm:univesalty}.
\end{proof}

 In the final part of the proof, we will also need an estimate of the surplus edges in the components $\bar{\mathscr{C}}_{\sss (i)}(\lambda)$, that can be obtained by following the exact same argument as the proof outline of Lemma~\ref{thm:superstrcut-comp-surp-original}. Recall that the superstructure on the graph $\bar{\mathcal{G}}_n(t_c(\lambda))$ is a rank-one inhomogeneous random graph $\mathrm{NR}_n(\bld{x},q)$. 
 The connection probabilities given by \eqref{eq:pij-value} can be written as $1-\exp(-z_i(\lambda)z_j(\lambda)/\sum_kz_k(\lambda))$, where 
 $  z_i(\lambda) = \frac{f_i(t_n)\sum_jf_j(t_n)}{\sum_j f_j^2(t_n)}\big(1+\lambda n^{-\eta+\delta}+\oP(n^{-\eta+\delta})\big) .$
 Moreover, using Theorem~\ref{th:open-he-entrance}, it follows that 
 \begin{equation}
  n^{-\alpha}z_i(\lambda) \pto \frac{\theta_i}{\nu}, \ \frac{z_i(\lambda)}{\sum_jz_j(\lambda)} \pto \frac{\theta_i}{\mu\nu}, \ \nu_n(\bld{z}) = \frac{\sum_iz_i^2(\lambda)}{\sum_i z_i(\lambda)} = 1+\lambda n^{-\eta+\delta}+\oP(n^{-\eta+\delta}).
 \end{equation}
 Now, we may consider the breadth-first exploration of the above graph  and define the exploration process 
$  S_n^{\sss \mathrm{NR}}(l) = \sum_i z_i(\lambda)\mathcal{I}_i^n(l) - l,$ as in \eqref{expl-process-CM}. The only thing to note here is that the component sizes are not necessarily encoded by the excursion lengths above the past minima of $\mathbf{S}_n^{\sss \mathrm{NR}}$.
 However, if $  \tilde{S}_n^{\sss \mathrm{NR}}(l) = \sum_i \mathcal{I}_i^n(l) - l$, then it can be shown that (see \cite[Lemma 3.1]{BHL12}) $\tilde{\mathbf{S}}_n^{\sss \mathrm{NR}}$ and $\mathbf{S}_n^{\sss \mathrm{NR}}$ have the same distributional limit. 
 Thus, a conclusion identical to Proposition~\ref{thm::convegence::exploration-process-blob} follows for $\bar{\mathcal{G}}_n(t_c(\lambda))$.  
 Due to the size-biased exploration of the components one can also obtain analogues of Lemmas~\ref{thm:superstrcut-comp-surp-original}~and \ref{thm:comp-blob-same-whp} for $\bar{\mathcal{G}}_n(t_c(\lambda))$. 
\begin{proposition}\label{thm:surp-superstruc-modi}
 For fixed  $K\geq 1$,
$  (n^{-\rho+\delta}\mathscr{B}(\bar{\mathscr{C}}_{\sss (i)}(\lambda)), \mathrm{SP}' (\bar{\mathscr{C}}_{\sss (i)}(\lambda)))_{i\in [K]} \dto (\xi_i, \mathscr{N}_i  )_{i\in[K]},$ as $n\to\infty$.
\end{proposition}

\subsection{Comparing modified and the original process: Completing the proof of Theorem~\ref{thm:main}}
\label{sec:proof-thm1}
In this section, we finally conclude the proof of Theorem~\ref{thm:main}. 
We start with the following:
\begin{lemma}\label{lem:sandwich-components}
For each fixed $i\geq 1$, $\mathscr{C}_{\sss (i)}(\lambda) \subset \bar{\mathscr{C}}_{\sss (i)}(\lambda)$ with high probability.
\end{lemma}
The proof is identical to  \cite[Lemma 31]{DHLS16}.
Recall Theorem~\ref{thm:mspace-limit-modified} and the terminologies therein.
Let $n^{-\eta}\mathscr{C}_{\sss (i)}^{\sss \mathrm{fr}}(\lambda)$ denote the measured metric space with measure $\mu_{\sss \mathrm{fr}}^i$ and the distances multiplied by $n^{-\eta}$.
At this moment, let us bring together the relevant properties $\mathscr{C}_{\sss (i)}(\lambda)$ and $\bar{\mathscr{C}}_{\sss (i)}(\lambda)$: 
\begin{enumerate}[(A)]
\item By Lemma~\ref{lem:sandwich-components}, 
it follows that with high probability $\mathscr{C}_{\sss (i)}(\lambda) \subset \bar{\mathscr{C}}_{\sss (i)}(\lambda)$ for any fixed $i\geq 1$. 
\item Part (A) implies $\bar{\mathcal{F}}_i(\lambda) \geq \mathcal{F}_i(\lambda) $ with high probability. 
The same scaling limits from  Propositions~\ref{thm:comp-functional-original}~(b)~and~\ref{thm:modified-open-he-limit} implies $\bar{\mathcal{F}}_i(\lambda) - \mathcal{F}_i(\lambda) \xrightarrow{\sss \PR} 0$ and consequently $\mu_{\sss \mathrm{fr}}^i(\bar{\mathscr{C}}_{\sss (i)}(\lambda)\setminus \mathscr{C}_{\sss (i)}(\lambda)) \xrightarrow{\sss \PR} 0$. 
\item By Propositions~\ref{thm:comp-functional-original}~(a)~and~\ref{thm:surp-superstruc-modi}, there are no surplus edges with one endpoint in  $\bar{\mathscr{C}}_{\sss (i)}(\lambda)\setminus \mathscr{C}_{\sss (i)}(\lambda)$ with high probability. 
Moreover, with high probability there is no surplus edge within the blobs by \eqref{eq:surp-sam-dist-limit}.
This implies that, for any pair of vertices $u,v\in \mathscr{C}_{\sss (i)}(\lambda)$, with high probability, the shortest path between them is exactly the same in $\mathscr{C}_{\sss (i)}(\lambda)$ and $\bar{\mathscr{C}}_{\sss (i)}(\lambda)$.
\end{enumerate}
Thus, from the definition of Gromov-weak convergence in Section~\ref{sec:defn:GHP-weak}, an application of Theorem~\ref{thm:mspace-limit-modified} yields that
 $\big(n^{-\eta}\mathscr{C}_{\sss (i)}^{\sss \mathrm{fr}}\big)_{i\geq 1} \xrightarrow{\sss d} (M_{i} )_{i\geq 1},$ 
The only thing remaining to show is that we can replace the measure $\mu_{\sss\mathrm{fr}}^i$ by $\mu_{\sss\mathrm{ct},i}$. Now, using Propositions~\ref{thm:comp-functionals-original}~and~\ref{thm:comp-functional-original}~(b), it is enough to show  that
\begin{equation}\label{counting-measure-switch}
 \sum_{b\in \bl( \mathscr{C}_{\sss (i)}(\lambda))}\big|f_{b}(t_n) - (\nu-1) |\Clb(t_n)|\big| = \oP(n^{\rho}).
\end{equation}
Indeed, during the breadth-first exploration of the superstructure of $\mathcal{G}_n(t_c(\lambda))$, the blobs are explored in a size-biased manner with the sizes being $(f_i(t_n))_{i\geq 1}$. 
Therefore, one can again use Lemma~\ref{lem:size-biased}. 
Recall that, by Lemma~\ref{thm:comp-blob-same-whp}, for any $\varepsilon>0$, one can choose $T>0$ so large that the probability of exploring $\mathscr{C}_{\sss (i)}(\lambda)$ within time $Tn^{\rho-\delta}$ is at least $1-\varepsilon$. 
Thus, if $\mathscr{V}^b_l$ denotes the set of blobs explored before time $l$, then, for any $T>0$,  
\begin{align*}
  &\sum_{b\in \mathscr{V}^b_{\sss Tn^{\rho-\delta}}}\big|f_{b}(t_n) - (\nu-1)|\Clb(t_n)|\big|\\
  &= (1+\oP(1)) Tn^{\rho-\delta} \sum_i \frac{f_i(t_n)}{\sum_if_i(t_n)}\big|f_{\sss i}(t_n) - (\nu-1)|\Cli(t_n)|\big|.
\end{align*}
Using the Cauchy-Schwarz inequality and Theorem~\ref{th:open-he-entrance} it now follows that the above term is $o(n^{\rho})$.
Therefore \eqref{counting-measure-switch} follows.
Finally, to conclude the result for percolated graphs, we  
use Lemma~\ref{lem:coupling-whp}. 
In fact, if $\mathscr{C}_{\sss (i)}^+(\lambda)$ and $\mathscr{C}_{\sss (i)}^-(\lambda)$ denote the $i$-th largest component of $\mathcal{G}_n(t_c(\lambda)+\varepsilon_{n})$ and $\mathcal{G}_n(t_c(\lambda)-\varepsilon_{n})$ respectively, then analogously to Lemma~\ref{lem:sandwich-components}, we can conclude that with high probability
\begin{eq}
\mathscr{C}_{\sss (i)}^-(\lambda) \subseteq \mathscr{C}_{\sss (i)}^p(\lambda) \subseteq   \mathscr{C}_{\sss (i)}^+(\lambda),
\end{eq}for any fixed $i\geq 1$.
This completes the proof of Theorem~\ref{thm:main}.
\qed

\begin{remark} \label{rem:switch-measure}
The fact that the measure can be changed from $\mu_{\sss \mathrm{fr}}^i$ to $\mu_{\sss\mathrm{ct},i}$ in $n^{-\eta}\mathscr{C}_{\sss (i)}^{\sss \mathrm{fr}}$ follows only from~\eqref{counting-measure-switch}, which again follows from the entrance boundary conditions. 
However, the entrance boundary conditions in Theorem~\ref{thm:susceptibility} hold for weight sequences $\bld{w}=(w_i)_{i\in [n]}$ under rather general assumptions (see Assumption~\ref{assumption-w}). 
Therefore, one could also replace the measure $\mu_{\sss\mathrm{ct},i}$ by $\mu_{w,i}$, where $\mu_{w,i} = \sum_{i\in A}w_i/\sum_{k\in \mathscr{C}_{\sss (i)}(\lambda)}w_i$ and $\bld{w}$ satisfies Assumption~\ref{assumption-w}.
\end{remark}

\subsection{Graphs conditioned on simplicity: Proof of Theorem~\ref{thm:main-simple}}
\label{sec:simple}
We will use the following joint construction of the $\mathrm{CM}_n(\bld{d},p_n(\lambda))$ and $\mathrm{CM}_n(\bld{d})$.
\begin{algo}\label{algo:perc-CM-joint-construction}\normalfont 
\begin{itemize}
    \item[(S0)] Let $\ell_n^p = 2X$, where $X\sim \mathrm{Bin}(\ell_n/2,p_n(\lambda))$. Pick $\ell_n^p$ many half-edges uniformly at random and color them blue. Color the rest of the half-edges red. 
    \item[(S1)] Pair the blue half-edges using a uniform perfect matching.
    \item[(S2)] Pair the red half-edges using another independent uniform perfect matching.
\end{itemize}
\end{algo}
If $G_I$ is the graph obtained after (SI) with I=1,2, then $(G_1,G_2)$ is jointly distributed as $(\mathrm{CM}_n(\bld{d},p_n(\lambda)),\mathrm{CM}_n(\bld{d}))$ \cite[Lemmas 8.1, 8.2]{DHLS15}. 
Let $d_i^p$ be the number of blue half-edges incident to $i$ and let $\bld{d}^p = (d_i^p)_{i\in [n]}$. 
Then, by construction, $G_1$, conditionally on  $\bld{d}^p$, is distributed as $\mathrm{CM}_n(\bld{d}^p)$.
To complete the proof of Theorem~\ref{thm:main-simple}, consider the exploration algorithm given by Algorithm~\ref{algo:explor-barely-sub}, now on the graph $G_1$, conditionally on the blue half-edges selected in Algorithm~\ref{algo:perc-CM-joint-construction}~(S0).
The starting vertex is chosen in a size biased manner with sizes proportional to the degrees $\bld{d}^p$. 
Let $\mathscr{F}_l$ denote the sigma-algebra generated by the exploration process up to time $l$.
Let $\mathcal{I}_i^n(l)$ denote the indicator that vertex $i$ is discovered upto time $l$ and note that Algorithm~\ref{algo:explor-barely-sub} will explore the vertices in a size-biased manner with sizes being $\bld{d}^p$.
For convenience, we denote $X = (\mathscr{C}_{\sss (i)}^p(\lambda))_{i\leq K}$ in this section. 
Consider a bounded continuous function $f:(\mathscr{S}_*)^K\mapsto \R$. 
Recall from \cite[Theorem 1.1]{J09c} that $$\liminf_{n\to\infty}\prob{G_2 \text{ is simple}}>0.$$
Thus, it is enough to show that 
\begin{equation}\label{simple-fact}
\expt{f(X)\ind{G_2 \text{ is simple}}} - \expt{f(X)} \prob{G_2 \text{ is simple}} \to 0.
\end{equation}
Now, for any $T>0$, let $\mathcal{A}_{n,T}$ denote the event that $X$ is explored before time $Tn^{\rho}$ by the exploration algorithm. 
Using \cite[Lemma 13]{DHLS16}, it follows that 
\begin{equation}
\lim_{T\to\infty}\limsup_{n\to\infty} \prob{\mathcal{A}_{n,T}^c} = 0.
\end{equation}
Let $X_{ T}$ denote the random vector consisting of $K$ largest ones among the components explored before time $Tn^{\rho}$.
 Thus,
\begin{align*}
&\lim_{T\to\infty}\limsup_{n\to\infty}\expt{f(X)\ind{G_2 \text{ is simple}}\1[\mathcal{A}_{n,T}^c]} \\
&\leq \|f\|_{\infty} \lim_{T\to\infty}\limsup_{n\to\infty}\prob{\mathcal{A}_{n,T}^c}=0, 
\end{align*}which implies that
\begin{equation}\label{simple-fact-3}
\lim_{T\to\infty}\limsup_{n\to\infty}\big|\expt{f(X)\ind{G_2 \text{ is simple}}} - \expt{f(X_{T})\ind{G_2 \text{ is simple}}}\big| = 0.
\end{equation}
Further, let $\mathcal{B}_{n,T}$ denote the event that a vertex $v$ is explored before time $Tn^{\rho}$ such that $v$ is involved in a self-loop or a multiple edge in $G_2$. 
Let $v_l$ denote the \emph{exploring} vertex in the exploration at time~$l$. 
Without loss of generality, we assume that during the sequential pairing of the half-edges in Algorithm~\ref{algo:perc-CM-joint-construction}~(S2), we first pair the red half-edges associated to $(v_l)_{l\leq Tb_n}$. 
Let $\ell_n':=\ell_n^p-2Tb_n-d_{1}^p+1$ and $\ell_n'':= (\ell_n-\ell_n^p) - \sum_{i\in [n]} d_{i}\mathcal{I}_i(Tn^{\rho})$. 
Note that, by Assumption~\ref{assumption1}~(ii), $\sum_{i\in V}d_i = o(n)$ whenever $|V| = o(n)$.
Also, using concentration inequalities for the  Binomial distribution, $\ell_n^p = p_n(\lambda) \ell_n (1+o(1))$ almost surely. 
Thus, we assume that $\ell_n', \ell_n'' \geq c_1\ell_n$ with probability 1 for some $0<c_1< 1$.
Note that, uniformly over $l\leq Tb_n$, any blue half-edge of $v_l$ creates a self-loop in $G_1$ with probability at most $d_{v_l}^p/\ell_n'$, and any red half-edge creates a self-loop with probability at most $(d_{v_l}-d_{v_l}^p)/\ell_n''$. 
Thus the expected number of self-loops incident to $v_l$ is at most $\E[d_{v_l}^2] /c_1\ell_n$.
Moreover, the expected number of blue-blue multiple edge attached to $v_l$ in $G_1$, is at most
\begin{eq}
\E[d_{v_l}^p(d_{v_l}^p-1)] \frac{ \sum_{i\in[n]} d_i(d_{i}-1)}{c_1\ell_n(c_1\ell_n-1)} \leq  \frac{C\E[d_{v_l}^2]}{\ell_n},
\end{eq}where we have used Assumption~\ref{assumption1}.
While counting the multiple edges incident to $v_l$ in $G_2$, we have to take care of (i) the creation of a red edge between two vertices having a blue edge and (ii) the creation of two red edges between two vertices.
Using identical arguments, the expected number of multiple edges incident to $v_l$ in $G_2$ is at most $C\E[d_{v_l}^2]/\ell_n$.
Therefore,
\begin{equation}
 \expt{\#\{\text{self-loops or multiple edges discovered while }v_l \text{ is exploring}\}}\leq \frac{C\E[d_{v_l}^2]}{\ell_n}.
\end{equation}
Thus, 
\begin{align*}
 &\PR(\mathcal{B}_{n,T})\leq \frac{C}{\ell_n}\E\bigg[\sum_{i\in [n]}d_i^2\mathcal{I}^n_i(Tn^{\rho})\bigg]=\frac{C}{\ell_n}\bigg(\E\bigg[\sum_{i=1}^K d_i^2\mathcal{I}^n_i(Tn^{\rho})\bigg]+\E\bigg[\sum_{i=K+1}^nd_i^2\mathcal{I}^n_i(Tn^{\rho})\bigg]\bigg).
 \end{align*} Now, using Assumption~\ref{assumption1}, for every fixed $K\geq 1$,
\begin{equation}
 \frac{1}{\ell_n}\E\bigg[\sum_{i=1}^K d_i^2\mathcal{I}^n_i(Tn^\rho)\bigg]\leq \frac{1}{\ell_n} \sum_{i=1}^K d_i^2 \pto 0.
\end{equation} 
Further, conditionally on Algorithm~\ref{algo:perc-CM-joint-construction}~(S0), the vertices are explored in a size-biased manner with sizes being $(d_i^p/\ell_n^p)_{i\in [n]}$.
Therefore, using that $\ell_n^p = p_n(\lambda)\ell_n(1+o(1))$,
\begin{equation}\label{simple-calc-1}
 \begin{split} 
  \frac{1}{\ell_n}\E_p\bigg[\sum_{i=K+1}^nd_i^2\mathcal{I}_i^n(Tn^{\rho})\bigg]\leq \frac{Tn^{\rho}}{\ell_n}\sum_{i=K+1}^nd_i^2 \E\Big[\frac{d_i^p}{\ell_n^p}\Big]
  \leq C\bigg(n^{-3\alpha}\sum_{i=K+1}^nd_i^3 \bigg).
 \end{split}
\end{equation}
Now, by Assumption~\ref{assumption1}, the final term in \eqref{simple-calc-1} tends to zero if we first take $\limsup_{n\to\infty}$ and then take $\lim_{K\to\infty}$. 
Consequently, for any fixed $T>0$, 
\begin{equation}\label{eq:B-n-t}
 \lim_{n\to\infty}\prob{\mathcal{B}_{n,T}}= 0.
\end{equation}
Let $\mathcal{E}_{n,T}$ denote the event that no self-loops or multiple edges are attached to the vertices in $G_2$ that are discovered after time $Tn^{\rho}$. Then \eqref{simple-fact-3} and \eqref{eq:B-n-t} implies that
\begin{eq}\label{simple-fact-2}
&\lim_{n\to\infty}\expt{f(X)\ind{\CM \text{ is simple}}} = \lim_{T\to\infty}\lim_{n\to\infty}\expt{f(X_T)\1[\mathcal{E}_{n,T}]}\\
&= \lim_{T\to\infty}\lim_{n\to\infty}\expt{f(X_T)\prob{\mathcal{E}_{n,T}\vert \mathscr{F}_{Tn^{\rho}}}} = \lim_{T\to\infty}\lim_{n\to\infty}\expt{f(X_T)\prob{\mathcal{E}_{n,T}\vert \mathscr{F}_{Tn^{\rho}}, \mathcal{B}_{n,T}}}.
\end{eq}
Let $\cG^*_{Tn^{\rho}}$ denote the graph obtained from $G_2$ after removing the vertices discovered upto time $Tn^{\rho}$.
Then, conditionally on $\mathscr{F}_{Tn^{\rho}}\cap \mathcal{B}_{n,T}$, $\mathcal{E}_{n,T}$ happens if and only if $\cG^*_{Tn^{\rho}}$ is simple.
Also, $\cG^*_{Tn^{\rho}}$ is distributed as a  configuration model conditional on its degree sequence, and since only $o(n)$ vertices have been removed, the corresponding $\nu_n$ in $\cG^*_{Tn^{\rho}}$ converges in probability to 1. Thus, \cite[Theorem 7.12]{RGCN1} implies that 
\begin{equation}
\prob{\cG^*_{Tn^{\rho}} \text{ is simple}\vert \mathscr{F}_{tn^{\rho}}} \pto \e^{-3/4},
\end{equation}
and also $\prob{G_2 \text{ is simple}}\to \e^{-3/4},$ 
so that 
\begin{eq}
\prob{\cG^*_{Tn^{\rho}} \text{ is simple}\vert \mathscr{F}_{tn^{\rho}}} - \prob{G_2 \text{ is simple}} \pto 0.
\end{eq}
Now, using \eqref{simple-fact-2},  \eqref{simple-fact} follows, and the proof of Theorem~\ref{thm:main-simple} is complete. \qed
\vspace{.5cm}

\bibliographystyle{apa}
\bibliography{project3}




\appendix
\section{Rescaling the excursions: Proof of Proposition~\ref{prop:comp-size}}\label{sec:appendix-rescaling}
Note that due to the difference in the choice of $p_n(\lambda)$ in \cite[Assumption 2]{DHLS16} and this paper, $\lambda$ must be replaced by $\lambda \nu$.
Let $\mathcal{E}(\cdot)$ denote the operator that maps a process to its ordered vector of excursion lengths, and $\mathcal{A}(\cdot)$  maps a process to the vector of areas under those excursions. 
Let us use $\mathrm{Exp}(b)$ as a generic notation to write an exponential random variable with rate $b$.
Now, 
 \begin{equation}\label{eq:excursion-param}
 \begin{split}
  &\frac{1}{\sqrt{\nu}}\mathcal{E}\bigg(\sum_{i\geq 1}\frac{\theta_i}{\sqrt{\nu}}\big(\ind{\mathrm{Exp}(\theta_i/(\mu\sqrt{\nu}))\leq t} - (\theta_i/(\mu\sqrt{\nu}))t\big)+\lambda\nu t\bigg)\\
  &\hspace{1cm} \eqd \frac{1}{\nu} \mathcal{E}\bigg(\sum_{i\geq 1}\frac{\theta_i}{\sqrt{\nu}}\big(\ind{\mathrm{Exp}(\theta_i/(\mu\nu))\leq u} - (\theta_i/(\mu\nu))u\big)+\lambda u\sqrt{\nu}\bigg)\\
  &\hspace{1cm} \eqd \frac{1}{\nu} \mathcal{E}\bigg(\sum_{i\geq 1}\frac{\theta_i}{\mu\nu}\big(\ind{\mathrm{Exp}(\theta_i/(\mu\nu))\leq u} - (\theta_i/(\mu\nu))u\big)+\frac{\lambda}{\mu} u\bigg),
   \end{split}
 \end{equation} where the last step follows by rescaling the space by $\mu\sqrt{\nu}$ and noting that the rescaling of space does not affect excursion lengths.
 Again, 
 \begin{eq}
  &\mathcal{A}\bigg(\sum_{i\geq 1}\frac{\theta_i}{\mu\sqrt{\nu}}\big(\ind{\mathrm{Exp}(\theta_i/(\mu\sqrt{\nu}))\leq t} - (\theta_i/(\mu\sqrt{\nu}))t\big)+\frac{\lambda\nu}{\mu} t\bigg)\\
  & \hspace{1cm}\eqd \mathcal{A}\bigg(\sum_{i\geq 1}\frac{\theta_i}{\mu\nu}\big(\ind{\mathrm{Exp}(\theta_i/(\mu\nu))\leq u } - (\theta_i/(\mu\nu))t\big)+\frac{\lambda}{\mu} u\bigg),
 \end{eq}which is obtained by rescaling both the space and time by $\sqrt{\nu}$.  
 Thus, the proof follows.
\section{Barely subcritical exploration process: Proofs of Lemmas \ref{lem:exploration::subcritical-j} and \ref{lem:weight-prop-l}}
\label{sec:appendix-barely-subcrit}

\begin{proof}[Proof of Lemma~6.3]
Recall the representation of $\bar{S}_n^j(t)$. It is enough to show that 
\begin{equation}
\sup_{t\in [0,T]} n^{-\alpha} \bigg|\sum_{i\in [n]}d_i'\bigg( \mathcal{I}_i^n(tn^{\alpha+\delta})-\frac{d_i'}{\ell_n'}tn^{\alpha+\delta} \bigg)\bigg|= \sup_{t\in [0,T]}n^{-\alpha}|M_n(tn^{\alpha+\delta})| \pto 0.
\end{equation}
Fix any $T>0$ and define $\ell_n'(T)=\ell_n'-2Tn^{\alpha+\delta}-1$, and $ M_n'(l) =\sum_{i\in [n]} d_i'(\mathcal{I}_i^n(l)- (d_i/\ell_n'(T))l)$. Note that 
\begin{equation}
 \sup _{t\in [0,T]}n^{-\alpha}|M_n(tn^{\alpha+\delta})-M_n'(tn^{\alpha+\delta})| \leq Tn^{\delta}\frac{(2Tn^{\alpha+\delta}-1)\sum_{i\in [n]}d_i'^2}{\ell_n'(T)^2} = \oP(1),
\end{equation}and thus the proof reduces to showing that 
\begin{equation} \label{tail::martingale}
\sup_{t\in [0,T]}n^{-\alpha}|M_n'(tn^{\alpha+\delta})| \pto 0.
\end{equation}
Note that, uniformly over $l\leq Tn^{\alpha+\delta}$, 
  \begin{equation}\label{eq:prob-ind}
  \prob{\mathcal{I}_i^n(l+1)=1\big| \mathscr{F}_l} \leq \frac{d_i'}{\ell_n'(T)}\quad\text{ on the set } \{\mathcal{I}_i^n(l)=0\}.
  \end{equation} 
 Therefore,
 \begin{align*}
  &\E\big[M_n'(l+1)-M_n'(l) \big| \mathscr{F}_l\big]\\
  &=\E\bigg[\sum_{i\in [n]} n^{-\alpha}d_i' \left(\mathcal{I}^n_i(l+1)-\mathcal{I}_i^n(l)-\frac{d_i'}{\ell_n'(T)}\right)\Big| \mathscr{F}_l\bigg]\\
  &= \sum_{i\in [n]} n^{-\alpha}d_i' \left(\E\big[\mathcal{I}^n_i(l+1)\big| \mathscr{F}_l\big]\ind{\mathcal{I}_i^n(l)=0} - \frac{d_i'}{\ell_n'(T)} \right)\leq 0.
  \end{align*} Thus $(M_n'(l))_{l= 1}^{Tn^{\alpha+\delta}}$ is a super-martingale.  Further, uniformly for all $l\leq Tn^{\alpha+\delta}$,
  \begin{equation}\label{prob-ind-lb}
   \prob{\mathcal{I}_i^n(l)=0} \leq \left(1-\frac{d_i'}{\ell_n'} \right)^l.
  \end{equation}
  Thus, Assumption~2 gives
  \begin{align*}
    n^{-\alpha}\big| \E[M_n'(l)]\big| &\leq n^{-\alpha} \sum_{i\in [n]} d_i'\left( 1-\left(1-\frac{d_i'}{\ell_n'} \right)^l-\frac{d_i'}{\ell_n'}l \right)+n^{-\alpha}l\sum_{i\in [n]}d_i'^2\left(\frac{1}{\ell_n'(T)}-\frac{1}{\ell_n'}\right)\\
    &\leq \frac{l^2}{2\ell_n'^2 n^{\alpha} } \sum_{i\in [n]} d_i'^3+o(1) = o(1),
  \end{align*} 
 where we have used the fact that 
  $n^{-\alpha}l\sum_{i\in [n]}d_i'^2(1/\ell_n'(T)-1/\ell_n')$ $=O(n^{2\rho+1-\alpha-2})$ $=O(n^{(\tau-4)/(\tau-1)}),$ uniformly for $l\leq Tn^{\alpha+\delta}$ and, in the last step, that fact that $\delta<\eta$. Therefore, uniformly over $l\leq Tn^{\alpha+\delta}$,
  \begin{equation}\label{expectation::M_n^K}
  \lim_{n\to\infty}\big| \E[M_n'(l)]\big|=0.
  \end{equation} 
  Now, note that for any $(x_1,x_2,\dots)$, $0\leq a+b \leq x_i$ and $a,b>0$ one has $\prod_{i=1}^R(1-a/x_i)(1-b/x_i)\geq \prod_{i=1}^R (1-(a+b)/x_i)$. Thus, for all $l\geq 1$ and $i\neq j$, 
  \begin{equation}\label{neg:correlation}
  \prob{\mathcal{I}_i^n(l)=0, \mathcal{I}_j^n(l)=0}\leq \prob{\mathcal{I}_i^n(l)=0}\prob{\mathcal{I}_j^n(l)=0}
  \end{equation} and therefore $\mathcal{I}_i^n(l)$ and $\mathcal{I}^n_j(l)$ are negatively correlated. Observe also that, uniformly over $l\leq Tb_n$, 
  \begin{equation}\label{var-ind-ub}
   \var{\mathcal{I}_i^n(l)}\leq  \prob{\mathcal{I}_i^n(l)=1} \leq \sum_{l_1=1}^l\prob{\text{vertex  }i \text{ is first discovered at stage }l_1 }\leq \frac{ld_i' }{\ell_n'(T)}.
  \end{equation}  
  Therefore, using the negative correlation in \eqref{neg:correlation}, uniformly over $l\leq Tn^{\alpha+\delta}$, 
  \begin{equation} \label{variance::M_n^k}
   \begin{split}
    n^{-2\alpha}\var{M_n'(l)}&
    \leq \frac{l}{\ell_n'(T)n^{2\alpha}}\sum_{i\in [n]} d_i'^3 =o(1).
   \end{split}
  \end{equation} Now we can use the super-martingale inequality \cite[Lemma 2.54.5]{RW94} stating that for any super-martingale $(M(t))_{t\geq 0}$, with $M(0)=0$, 
 \begin{equation}\label{eqn:supmg:ineq}
  \varepsilon \prob{\sup_{s\leq t}|M(s)|>3\varepsilon}\leq 3\expt{|M(t)|}\leq 3\left(|\expt{M(t)}|+\sqrt{\var{M(t)}}\right).
 \end{equation}
  Thus \eqref{tail::martingale} follows using \eqref{expectation::M_n^K}, \eqref{variance::M_n^k}, and \eqref{eqn:supmg:ineq}.
\end{proof}
\begin{proof}[Proof of Lemma~6.4]
Fix any $T>0$ and recall that $\ell_n(T)=\ell_n'-2Tn^{\alpha+\delta}-1$. Denote $W(l)=\sum_{i\in [n]}w_i\mathcal{I}_i^n(l)$. Firstly, observe that
\begin{align*}
  \E[W(l+1)-W(l) | \mathscr{F}_l]= \sum_{i\in [n]} w_i\E\big[\mathcal{I}^n_i(l+1)\big| \mathscr{F}_l\big]\ind{\mathcal{I}_i^n(l)=0} \leq \frac{\sum_{i\in [n]}d_i'w_i}{\ell_n'(T)},
  \end{align*} 
uniformly over $l\leq Tn^{\alpha+\delta}$. Therefore, $(\tilde{W}(l))_{l=1}^{Tn^{\alpha+\delta}}$ is a super-martingale, where $\tilde{W}(l)=W(l)-(\sum_{i\in [n]}d_i'w_i/\ell_n')l$. Again, the goal is to use \eqref{eqn:supmg:ineq}. Using \eqref{prob-ind-lb}, we can show that
$ \big|\E[\tilde{W}(l)]\big|
 =o(n^{\alpha+\delta}),$
uniformly over $l\leq Tn^{\alpha+\delta}$. Also, using \eqref{neg:correlation} and \eqref{var-ind-ub} and Assumption~\ref{assumption-w},
$\mathrm{var}(\tilde{W}(l))\leq \sum_{i\in [n]}w_i^2 \mathrm{var}(\mathcal{I}_i^n(l))= o(n^{2(\alpha+\delta)}),$ uniformly over $l\leq Tn^{\alpha+\delta}$. Finally, using \eqref{eqn:supmg:ineq},  we conclude the proof.
\end{proof}
\section{Barely subcritical exploration process: Proof of Fact~\ref{fact:total-edges-random-comp}}
\label{sec:appendix-barely-subcrit-large-deviation}
Consider exploring $\rCM_n(\bld{d}')$ by Algorithm~\ref{algo:explor-barely-sub} with $V_n^*$ being the starting vertex. 
Let us denote  the degree of the vertex found at step $l$ by $d_{\sss (l)}'$. 
If no new vertex is found at step $l$, then $d_{\sss (l)}' = 0$.
Also, let $\mathscr{F}_l$ denote the sigma-algebra containing all the information revealed by the exploration process upto time $l$.
Thus, 
\begin{eq}
S_n(0) = d_{V_n^*}', \quad \text{and}\quad  S_n(l) = S_n(l-1) + (d_{\sss (l)}'-2),
\end{eq}and when $\bld{S}_n$ hits zero, then $\sC'(V_n^*)$ has been explored. 
Using the Doob-Meyer decomposition, one can write
\begin{equation}
S_n(l) = S_n(0)+M_n(l) + A_n(l), 
\end{equation}where $M_n$ is a martingale with respect to $(\mathscr{F}_l)_{l\geq 1}$. 
The drift $A_n$ and the quadratic variation $\langle M_n \rangle$ of $M_n$ are given by 
\begin{equation}
   A_{n}(l)= \sum_{j=1}^{l} \mathbbm{E}\big[d_{\sss(j)}'-2 \vert \mathscr{F}_{j-1} \big], \qquad 
   \langle M_n \rangle(l)= \sum_{j=1}^{l} \var{d_{\sss(j)}'\vert \mathscr{F}_{j-1}} .
  \end{equation}
Let $t_n = n^{\alpha + \delta + c_0}$,
where we choose $c_0>0$ (sufficiently small) such that $\alpha + \delta + c_0 < 1-\delta$. Such a choice of $c_0$ is always possible since $\alpha +2\delta < \alpha +2\eta = 2 - 3\alpha <1$ as $3\alpha>1$.
We will show that, for all sufficiently large $n$,
\begin{eq}\label{fact-drift-bound}
A_n(t_n) \leq - \frac{\lambda_0}{2}  t_n n^{-\delta} \quad \text{almost surely}, 
\end{eq}
where $\lambda_0$ is given by \eqref{defn:barely-subcrit}, and for any $\varepsilon>0$, 
\begin{eq}\label{quadratic-variation-bound}
\PR(M_n(t_n) > \varepsilon t_n n^{-\delta}) \leq C'\e^{-C''\varepsilon^2 n^{\varepsilon_0}},
\end{eq}
where $C',C''>0$ are constants. 
Note that $S_n(0) = d_{V_n^*} \leq Cn^{\alpha} = o( t_n n^{-\delta})$. 
Thus, if $M_n(t_n) \leq \varepsilon t_n n^{-\delta}$ and \eqref{fact-drift-bound} holds, then $S_n(t_n) <0$ and hence $\sC'(V_n^*)$ is explored before time $t_n$. 
This in turn implies that $\sum_{i\in \sC'(V_n^*)} d_i' \leq 2t_n$, since one edge is explored per step. 
Therefore it is enough to prove \eqref{fact-drift-bound} and \eqref{quadratic-variation-bound}. 
Let $\mathscr{V}_j$ denote the set of vertices explored up to time $j$.
Recall that, by the definition of $\nu_n'$ in \eqref{defn:barely-subcrit}, $\frac{1}{\ell_n'}\sum_{i\in [n]} d_i'^2 -2= \lambda_0 n^{-\delta} +o(n^{-\delta})$.
Then, uniformly over $j\leq t_n$,
\begin{eq}
\mathbbm{E}\big[d_{\sss(j)}'-2 \vert \mathscr{F}_{j-1} \big] &= \mathbbm{E}\big[d_{\sss(j)}' \vert \mathscr{F}_{j-1} \big] - 2 \leq \frac{\sum_{i\notin \mathscr{V}_{j-1}} d_i'^2}{\ell_n' - 2t_n+1}-2 \\
&\leq \frac{\sum_{i\in [n]} d_i'^2}{\ell_n' - 2t_n+1} -2 = \frac{\ell_n (2-\lambda_0 n^{-\delta} +o(n^{-\delta}))}{\ell_n' - 2t_n+1} -2 \\
&= - \lambda_0 n^{-\delta} + O(t_n/n) +o(n^{-\delta}) \leq - \frac{\lambda_0}{2} n^{-\delta},
\end{eq} for all sufficiently large $n$, where in the final step we have used the fact that $t_n = o(n^{1-\delta})$. 
Thus \eqref{fact-drift-bound} follows. 

To prove \eqref{quadratic-variation-bound}, we  use Freedman's inequality \cite[Proposition 2.1]{Fre75} which says that if $Y(k) = \sum_{j\leq k} X_j$ with $\E[X_j\vert \mathcal{F}_{j-1}] =0$ (for some filtration $(\mathcal{F}_j)_{j\geq 1}$) and $\PR(|X_j|\leq R, \ \forall j\geq 1)=1$, then, for any $a,b >0$,
\begin{eq}\label{eq:freedman}
\PR(Y(k) \geq a, \text{ and } \langle Y\rangle(k)  \leq b ) \leq \exp\bigg(-\frac{a^2}{2(Ra+b)}\bigg).
\end{eq}
Note that, uniformly over $j\leq t_n$, 
\begin{equation}
\var{d_{\sss (j)}' \vert \mathscr{F}_{j-1}} \leq \E[d_{\sss (j)}'^2 \vert \mathscr{F}_{j-1}]  = \frac{\sum_{j\notin \mathscr{V}_{j-1}} d_j'^3}{\ell_n' - 2t_n +1} \leq \frac{\sum_{j\in [n]} d_j'^3}{\ell_n' - 2t_n+1}\leq Cn^{3\alpha - 1},
\end{equation}so that, almost surely,
\begin{equation}\label{eq:bound-QV}
\langle M_n\rangle (t_n) \leq C t_n n^{3\alpha-1}.
\end{equation}
Also, $d_{\sss (j)} \leq C n^{\alpha}$ almost surely.
Thus, applying \eqref{eq:freedman} with $a= \varepsilon t_n n^{-\delta}$, $b=Ct_n n^{3\alpha-1}$ and $R=Cn^{\alpha}$, and also using \eqref{eq:bound-QV}, it follows that 
\begin{eq}
\PR(M_n(t_n) > \varepsilon t_n n^{-\delta}) \leq \exp \bigg(- C' \frac{\varepsilon^2 t_n^2 n^{-2\delta}}{2  (t_n n^{3\alpha - 1} + \varepsilon n^\alpha t_n n^{-\delta})}\bigg) \leq C'\e^{-C'' \varepsilon n^{\varepsilon_0}},
\end{eq} 
where in the last step we have used the fact that $\alpha - \delta > \alpha - \eta = 3\alpha -1$ and $t_n = n^{\alpha+\delta+\varepsilon_0}$.
Thus the proof of \eqref{quadratic-variation-bound} follows.
\qed

\section{Limit of exploration process: Proof sketch for Proposition~\ref{thm::convegence::exploration-process-blob}}
\label{sec:appendix-perc-blob}
The proof of Proposition~\ref{thm::convegence::exploration-process-blob} can be carried out using similar ideas as \cite[Theorem 8]{DHLS16}. 
The key idea to prove Proposition~\ref{thm::convegence::exploration-process-blob} is that the scaling limit is governed by the vertices having large degrees only. More precisely, for any $\varepsilon > 0$ and $T>0$,
\begin{equation}
 \lim_{K\to\infty}\limsup_{n\to\infty}\PR\bigg(\sup_{t\leq T}n^{-\alpha}\bigg|\sum_{i>K}a_i\Big( \mathcal{I}_i^n(tn^{\rho-\delta})-\frac{a_i}{\ell_n^a}tn^{\rho-\delta} \Big)\bigg| > \varepsilon \bigg) = 0.
\end{equation} This can be proved using martingale estimates, see \cite[Section 4]{DHLS16}. Thus, if one considers the truncated sum 
\begin{align*}
\sum_{i\leq K} a_i \left( \mathcal{I}_i^n(l)-\frac{a_i}{\ell^a_n}l\right)+\left( \nu_n(\bld{a})-1\right)l,
\end{align*}with the first $K$ (fixed) terms
it is enough to show that the iterated limit of the truncated process (first taking $\lim_{n\to\infty}$ and then $\lim_{K\to\infty}$) converges to $\mathbf{S}$ with respect to the Skorohod $J_1$ topology. Now, using the fact that $a_i/\sum_ia_i\xrightarrow{\sss \PR} \theta_i/(\mu\nu)$, and the fact that the vertices are explored in a size-biased manner with sizes being $(a_i)_{i\geq 1}$, it follows that (see \cite[Lemma 9]{DHLS16}), for each fixed $K\geq 1$,
\begin{equation}
 \big(\mathcal{I}_i^n(tn^{\rho-\delta})\big)_{i\in [K],t\geq 0} \dto \big(\ind{\mathrm{Exp}(\theta_i/(\mu\nu))\leq t}\big)_{i\in [K],t\geq 0}.
\end{equation}
This concludes the proof of Proposition~\ref{thm::convegence::exploration-process-blob}. 


\end{document}